\documentclass{amsart}

\usepackage[T1]{fontenc}
\usepackage{enumerate}
\usepackage{epsfig, latexsym, amsfonts, amsthm, amsmath, pstricks, mathtools}

\DeclareMathOperator{\Suz}{Suz}
\DeclareMathOperator{\MSuz}{MouSu}
\newcommand{\N}{\mathsf{N}}
\numberwithin{equation}{section}

\newtheorem{theorem}[equation]{Theorem}
\newtheorem{lemma}[equation]{Lemma}
\newtheorem{proposition}[equation]{Proposition}
\newtheorem{corollary}[equation]{Corollary}
\theoremstyle{definition}
\newtheorem{definition}[equation]{Definition}

\newtheorem{remark}[equation]{Remark}

\newtheorem{notation}[equation]{Notation}

\newtheorem{convention}[equation]{Convention}
\newtheorem{hypothesis}[equation]{Hypothesis}

\newtheorem{example}[equation]{Example}
\newtheorem{examples}[equation]{Example}

\begin{document}
\title{Tits Endomorphisms and Buildings of Type $F_4$}
\author{Tom De Medts, Yoav Segev and Richard M. Weiss}
\address{Department of Mathematics \\
	Ghent University \\
	9000 Gent, Belgium}
\email{Tom.DeMedts@UGent.be}
%\author{Yoav Segev}
\address{Department of Mathematics \\
        Ben Gurion University
        Beer Sheva 84105, Israel}
\email{yoavs@cs.bgu.ac.il}
%\author{Richard M. Weiss}
\address{Department of Mathematics \\
         Tufts University \\
         Medford, MA 02155, USA}
\email{rweiss@tufts.edu}


\date{March 3, 2017}

\begin{abstract}
The fixed point building of a polarity of a Moufang quadrangle of type $F_4$
is a Moufang set, as is the fixed point building of a semi-linear automorphism
of order~$2$ of a Moufang octagon that stabilizes at least two panels of
one type but none of the other. We show that these two classes of Moufang
sets are, in fact, the same, that each member of this class can be constructed as
the fixed point building of a group of order~$4$ acting on a building of type $F_4$
and that the group generated by all the root groups of any one of these Moufang sets is
simple.

\bigskip\noindent
{\sc R\'esum\'e.}
L'immeuble de points fixes d'une polarit\'e d'un quadrangle de Moufang de type $F_4$
est un ensemble de Moufang. Il en va de m\^eme pour l'immeuble de points fixes d'un
automorphisme semi-lin\'eaire d'ordre 2 d'un octogone de Moufang qui stabilise au moins
deux cloisons d'un type mais aucun de l'autre. Nous montrons que ces deux classes
d'ensembles de Moufang sont en fait identiques, que chaque membre de cette classe peut
\^etre construit comme l'immeuble de points fixes d'un groupe d'ordre~4 agissant sur
un immeuble de type $F_4$, et que pour chacun de ces ensembles de Moufang,
le groupe engendr\'e par tous les sous-groupes radiciels est un groupe simple.
\end{abstract}
\maketitle

%\tableofcontents

\section{Introduction}\label{sec1}

The notion of a building was introduced by J. Tits in order to give a uniform geometric/combinatorial
description of the groups of rational points of an isotropic absolutely simple group.
The buildings that arise in this context are spherical. In \cite{bn}, Tits classified
irreducible spherical buildings of rank at least~$3$ and this classification was
extended to the rank~$2$ case in \cite{TW} under the assumption that the building
is Moufang (which is automatic when the rank is at least~$3$).

The result of this classification is that most Moufang buildings (as defined in
\ref{nzz6x}) are the
spherical buildings associated with isotropic absolutely simple algebraic groups. The
exceptions are buildings determined by algebraic data involving infinite
dimensional structures, defective quadratic or pseudo-quadratic forms,
inseparable field extension and/or the square root of a Frobenius endomorphism.
Among these exceptions are the mixed buildings of type $B_2$, $G_2$ and $F_4$,
the Moufang quadrangles of type $F_4$ and the Moufang octagons.

The classification results in \cite{bn} and \cite{TW} are proved by coordinatizing
with an appropriate algebraic structure. These methods do not reveal the connection
between the automorphism group of a Moufang
building and an associated split group which is the central
concern in the theory of Galois descent. In \cite{MPW}, this shortcoming
is remedied with a theory of descent for buildings.
This theory provides, in particular,
a combinatorial interpretation of the Tits indices which appear in \cite{boulder}.

In \ref{nzz4}, we define the notion of a {\it descent group} of a building $\Delta$ and
in \ref{goo6}, we define the notion of a {\it Galois subgroup} of ${\rm Aut}(\Delta)$
in the case that $\Delta$ is Moufang.
One of the fundamental results in the theory of Galois descent in buildings in \cite{MPW}
says that the set of residues fixed by
a descent group $\Gamma$ has, in a canonical way, the structure of a building, which we call
the {\it fixed point building} of $\Gamma$ (see \ref{nzz8}). This result
applies to arbitrary descent groups acting on arbitrary buildings. A second
fundamental result (proved in \cite{rhizo}) says that if $\Delta$ is Moufang and $\Gamma$ is
a Galois subgroup of ${\rm Aut}(\Delta)$ acting with finite orbits on $\Delta$ and stabilizing at
least one proper residue of $\Delta$, then $\Gamma$ is a descent group.

A third fundamental result says that if the fixed point building
of a descent group of a Moufang building $\Delta$ has rank~$1$, then the fixed point
building inherits from the Moufang condition on $\Delta$ the structure of a {\it Moufang set}.
Moufang sets, which were first introduced in \cite[4.4]{twin},
are a class of 2-transitive permutation groups. The notion of a Moufang set is
closely related to the notion of a split $BN$-pair of rank~$1$.
All known Moufang sets which are proper (i.e.~not sharply $2$-transitive)
arise as fixed point buildings
of Moufang buildings. For a survey of recent results in the study of Moufang sets, see \cite{yoav}.

Polarities (i.e.~non-type-preserving automorphisms of order~$2$)
of Moufang buildings of type $B_2$, $G_2$ and $F_4$ are a second source of descent groups.
In the case that the building is {\it pseudo-split} (as defined in \cite[28.16]{MPW}),
polarities give rise to the Suzuki and Ree groups.
In \cite{octagons}, Tits characterized Moufang octagons as the fixed
point building of an arbitrary polarity of a building of type $F_4$.

The Moufang quadrangles of type $F_4$ (a class of non-pseudo-split Moufang buildings
of type $B_2$) were discovered in the course of
classifying Moufang buildings of rank~$2$ in \cite{TW}. Subsequently, it was shown
(in \cite{canada}) that these quadrangles can be constructed as the fixed
point building of a type-preserving Galois involution acting on a building of type $F_4$.
(It was due to this result that the designation ``of type $F_4$'' was
chosen in \cite{TW}.)

Among all {\it non}-pseudo-split Moufang buildings of type $B_2$, $G_2$ or $F_4$,
the Moufang quadrangles of type $F_4$ are the only ones that can have a polarity.
If a Moufang quadrangle $\Xi$ of type $F_4$ has a polarity $\rho$,
then the centralizer of $\rho$ in
${\rm Aut}(\Xi)$ induces a Moufang set on the the set $\Xi^{\langle\rho\rangle}$ of
chambers fixed by $\rho$. The first examples of such Moufang
sets were constructed in \cite{mvm}. It is also possible for a Moufang octagon $\Omega$ to have a
type-preserving Galois involution $\kappa$ such that the centralizer of $\kappa$
in ${\rm Aut}(\Omega)$ induces a Moufang set on the set $\Omega^{\langle\kappa\rangle}$
of panels fixed by $\kappa$. In \cite{mvm}, it was conjectured that these two classes
of Moufang sets are the same. Our main goal here is to prove this conjecture.

In the course of verifying this conjecture, we show that each of these Moufang sets is,
in fact, the fixed point
building of an elementary abelian subgroup $\Gamma$ of order~$4$ of the automorphism
group of a building $\Delta$ of type $F_4$. The group $\Gamma$ is not type-preserving.
For this reason, we say (in \ref{nzz29}) that these Moufang sets are {\it of outer $F_4$-type}.
The fixed point building of each of the two polarities in $\Gamma$ is a Moufang octagon (but
the two octagons are not,
in general, isomorphic to each other) and the fixed point building of the third
involution in $\Gamma$ is a Moufang quadrangle of type $F_4$. This lattice of fixed point
buildings is all the more interesting in light of the fact that there is no obvious
connection between an arbitrary Moufang quadrangle of type $F_4$ and an arbitrary Moufang
octagon except that they both ``descend from'' a building of type $F_4$ in characteristic~$2$.

The root groups of a Moufang set of outer $F_4$-type are nilpotent of class~$3$ as are the
root groups in the Ree groups of type $G_2$. In every other known proper Moufang set, the root
groups are either abelian or nilpotent of class~$2$.

In \S\ref{onk60a}--\S\ref{onk60x}, we show that the group generated by all the root groups of a
Moufang set of outer $F_4$-type is simple and describe a number
of other properties of these Moufang sets.

The basic invariant of a Moufang quadrangle $\Xi$ of type~$F_4$ is a pair of quadratic
forms $q$ and $\hat q$ of type $F_4$ (as defined in \ref{def1} below),
one defined over a field $K$ of characteristic~$2$ and the
other over a field $F$ such that $F^2\subset K\subset F$. The quadratic forms $q$ and $\hat q$
are anisotropic but have a defect of dimension
$\dim_{K^2}F$, respectively, $\dim_FK$. (We do not make any assumptions about these
dimensions; in particular, either one or both can be infinite.)

Let $V$ denote the vector space over $K$ on which $q$
is defined. We show that if the Moufang quadrangle $\Xi$ has a polarity $\rho$,
then the quadratic forms $q$ and $\hat q$ are similar to each other and
there is a {\it Tits endomorphism} $\theta$ of $K$
(i.e.~an endomorphism whose square is the Frobenius endomorphism) with image $F$
and a non-associative algebra structure on $V$ with respect to which the quadratic form $q$
satisfies the identity
\begin{equation}\label{nzz14}
q(uv)=q(u)q(v)^\theta
\end{equation}
for all $u,v\in V$ (see \ref{abc51} and \ref{q(uv)}). Thus $q$ is multiplicative ``with a twist''
(cf. \cite{elduque}, for instance).

We call the non-associative algebras that arise in this way {\it polarity algebras}.
In \S\ref{sec5}, we describe polarity algebras in terms of a system of axioms
and deduce from these axioms
equation \eqref{nzz14} along with a number of other intriguing identities.
In \ref{nzz30}, we use some of these identities to
show that $q$ is, up to similarity, an invariant not only of the quadrangle $\Xi$,
but also of the Moufang set $\Xi^{\langle\rho\rangle}$. See also \cite{koen}.

Tits endomorphisms and their extensions
play a central role in the study of polarities. The first thorough study
of this connection is in \cite{ree}; in fact, it was the influence of this paper
which led to the common attribution of these endomorphisms
to Tits. See, in particular, \S\ref{sec6}.

\begin{convention}\label{nzz91}
If $x$ and $y$ are two elements of a group, we set $x^y=y^{-1}xy$ and
$$[x,y]=x^{-1}y^{-1}xy=(y^{-1})^xy=x^{-1}x^y$$
(as in \cite{TW}). As a consequence of these conventions, the following two identities
\begin{equation}\label{nzz90}
[xy,z]=[x,z]^y\cdot[y,z]\text{\quad and\quad}[x,yz]=[x,z]\cdot[x,y]^z
\end{equation}
hold.
\end{convention}

\bigskip
\noindent
{\sc Acknowledgment:} Much of this work was carried out while the authors were
guests of the California Institute of Technology.
The third author was partially supported by DFG-Grant MU 1281/5-1 and
NSA-Grant H982301-15-1-0009. The authors would like to thank the referee for
the extraordinary care he or she took with the manuscript.

\section{Fixed Point Buildings and Moufang Sets}\label{sec2}

Before we can give precise formulations of our main results in \S\ref{nzz33},
we need to introduce some basic notions.

\begin{definition}\label{abc0}
Let $E$ be a field of positive characteristic~$p$.
We will denote the Frobenius endomorphism $x\mapsto x^p$ of $E$
by ${\rm Frob}_E$. A {\it Tits endomorphism} of
$E$ is an endomorphism of $E$ whose square is the Frobenius endomorphism.
An {\it octagonal set} is a pair $(E,\theta)$, where $E$ is a field of characteristic~$2$
and $\theta$ is a Tits endomorphism of $E$.
\end{definition}

\begin{definition}\label{abc3}
Let $\Delta$ be a building. An {\it involution} of $\Delta$ is an automorphism of order~$2$.
A {\it polarity} of $\Delta$ is an involution which is not type-preserving.
(We will only use this term when $\Delta$ is of type $B_2$, $F_4$ or $G_2$.)
\end{definition}

\begin{notation}\label{abc4}
Let $(E,\theta)$ be an octagonal set.
We denote by ${\mathcal O}(E,\theta)$ the Moufang octagon defined in \cite[16.9]{TW}
and we denote by $F_4(E,\theta)$ the building of type $F_4$ called $F_4(E,E^\theta)$ in
\cite[30.15]{affine}.
\end{notation}

\begin{definition}\label{abc6}
A {\it Moufang set} is a pair $(X,\{U_x\}_{x\in X})$, where $X$ is
a set of cardinality at least~$3$,
$U_x$ is a subgroup of ${\rm Sym}(X)$ fixing $x$ and
acting sharply transitively on $X\backslash\{x\}$ for all $x\in X$ and
$g^{-1}U_xg=U_{(x)g}$ for all $x\in X$ and all $g\in G^\dagger:=\langle U_x\mid x\in X\rangle$.
The subgroups $U_x$ are called the {\it root groups} of the Moufang set. A Moufang
set is {\it proper} if the group $G^\dagger$ does not act sharply $2$-transitively on $X$.
\end{definition}

\begin{definition}\label{nzz6x}
As in {\rm\cite[11.2]{spherical}},
we call a building {\it Moufang} if it is spherical, irreducible and of rank at least~$2$
and for each root $\alpha$, the root group $U_\alpha$ (as defined in \cite[11.1]{spherical})
acts transitively on the set of apartments containing $\alpha$.
(There are more general notions of a Moufang building, but they are not relevant in this paper.)
\end{definition}

\begin{proposition}\label{abc7}
Let $\kappa$ be an involution acting on a spherical building $\Delta$.
Then there exists an apartment of $\Delta$ stabilized by $\kappa$.
\end{proposition}

\begin{proof}
This holds by \cite[25.15]{MPW}.
\end{proof}

\begin{definition}\label{nzz1}
Let $\Delta$ be a building and let $\Gamma$ be a subgroup of ${\rm Aut}(\Delta)$.
A {\it $\Gamma$-residue} is a residue of $\Delta$ stabilized by $\Gamma$.
A {\it $\Gamma$-chamber} is a $\Gamma$-residue which is minimal with respect
to inclusion. A {\it $\Gamma$-panel} is a $\Gamma$-residue $P$ such that for some
$\Gamma$-chamber $C$, $P$ is minimal in the set of all $\Gamma$-residues containing
$C$.
\end{definition}

\begin{definition}\label{nzz2}
Let $\Delta$ and $\Gamma$ be as in \ref{nzz1}. A residue of $\Delta$ is {\it proper} if
it is different from $\Delta$ itself. (In particular, chambers are proper residues.)
The group $\Gamma$ is {\it anisotropic} if $\Gamma$ stabilizes no proper residues of $\Delta$,
and $\Gamma$ is {\it isotropic} if this is not the case.
Thus $\Gamma$ is isotropic if and only if there exist $\Gamma$-panels. An automorphism $\xi$
of $\Delta$ is isotropic (or anisotropic) if $\langle\xi\rangle$ is isotropic (or anisotropic).
\end{definition}

\begin{notation}\label{nzz3}
Let $\Delta$ be a building and let $\Gamma$ be an isotropic subgroup of ${\rm Aut}(\Delta)$.
We denote by $\Delta^\Gamma$ the graph with vertex set the set of all $\Gamma$-chambers,
where two $\Gamma$-chambers are joined by an edge of $\Delta^\Gamma$ if and only if
there is a $\Gamma$-panel containing them both.
\end{notation}

\begin{definition}\label{nzz4}
Let $\Delta$ be a building. A {\it descent group} of $\Delta$ is an isotropic subgroup
$\Gamma$ of ${\rm Aut}(\Delta)$ such that each $\Gamma$-panel contains at least three
$\Gamma$-chambers.
\end{definition}

\begin{proposition}\label{nzz6}
Suppose that a building $\Delta$ is Moufang as defined in {\rm\ref{nzz6x}}.
Let $R$ be a residue of $\Delta$, let $\Sigma$ be an apartment
containing chambers of $R$ and let
$U_R$ denote the subgroup generated by the root groups $U_\alpha$ for all
roots $\alpha$ of $\Sigma$ containing $R\cap\Sigma$. Then $U_R$ is independent
of the choice of $\Sigma$.
\end{proposition}

\begin{proof}
This holds by \cite[24.17]{MPW}
\end{proof}

\begin{definition}\label{nzz7}
Let $R$ and $\Delta$ be as in \ref{nzz6}. The group $U_R$ is called the
{\it unipotent radical} of $R$.
\end{definition}

\begin{proposition}\label{onk98}
Let $R$, $\Delta$ and $U_R$ be as in {\rm\ref{nzz7}}. Then $U_R$ acts sharply transitively
on the residues of $\Delta$ opposite $R$
\end{proposition}

\begin{proof}
This holds by \cite[24.21]{MPW}.
\end{proof}

\begin{definition}\label{nzz25}
Let $\Pi$ be a Coxeter diagram and let $(W,S)$ be the corresponding
Coxeter system. Thus $S$ is both a distinguished set of generators of the Coxeter
group $W$ and the vertex set of $\Pi$. Let $J$ be a subset of $S$ such that
the subdiagram $\Pi_J$ spanned by $J$ is spherical and let $w_J$ denote the
longest element of the Coxeter system $(W_J,J)$, where $W_J$ denotes the
subgroup of $W$ generated by $J$. By \cite[5.11]{spherical}, the map $s\mapsto w_Jsw_J$
is an automorphism of $\Pi$. We denote this automorphism by ${\rm op}_J$.
This map is called the {\it opposite map} of the diagram $\Pi_J$ (and ought, in fact,
to be denoted by ${\rm op}_{\Pi_J}$).
This map stabilizes every connected component of $\Pi_J$ and acts non-trivially
on a given connected component if and only if it is
isomorphic to the Coxeter diagram $A_n$ for some $n\ge2$, to
$D_n$ for some odd $n\ge5$, to $E_6$ or to $I_2(n)$ for some odd $n\ge5$.
In particular, its order is at most~$2$.
\end{definition}

\begin{definition}\label{nzz24}
A {\it Tits index} is a triple $(\Pi,\Theta,A)$ where $\Pi$ is a Coxeter diagram,
$\Theta$ is a subgroup of ${\rm Aut}(\Pi)$ and $A$ is a $\Theta$-invariant subset
of the vertex set $S$ of $\Pi$ such that for each $s\in S\backslash A$,
the subdiagram of $\Pi$ spanned by $A\cup\Theta(s)$ is spherical and $A$ is invariant
under the opposite map ${\rm op}_{A\cup\Theta(s)}$ defined in \ref{nzz25}. Here
$\Theta(s)$ denotes the $\Theta$-orbit containing $s$.
\end{definition}

\begin{definition}\label{nzz26}
Let $T=(\Pi,\Theta,A)$ be a Tits index. For each $s\in S\backslash A$,
let $\tilde s=w_Aw_{A\cup\Theta(s)}$, where $w_J$ for $J=A$ and $J=A\cup\Theta(s)$
is as in \ref{nzz25}. Thus there is one element $\tilde s$ for each $\Theta$-orbit
in $S\backslash A$. Let $\tilde S$ be the set of all these elements $\tilde s$.
By \cite[20.32]{MPW}, $(\tilde W,\tilde S)$ is a Coxeter system. Let $\tilde\Pi$
be the corresponding Coxeter diagram. We call $\Pi$ the {\it absolute Coxeter diagram}
of $T$ and $\tilde\Pi$ the {\it relative Coxeter diagram} of $T$. An algorithm
for calculating the relative Coxeter diagram of a Tits index is described in
\cite[42.3.5(c)]{TW}.
\end{definition}

\begin{examples}\label{nzz27}
Let $T=(\Pi,\Theta,A)$, where $\Pi$ is the Coxeter diagram $F_4$. If $\Theta={\rm Aut}(\Pi)$
and $A=\emptyset$, then $T$ is a Tits index with relative Coxeter diagram is $I_2(8)$.
If $\Theta$ is trivial and $A=\{2,3\}$ with respect to the standard numbering of the
vertex set of $\Pi$, then $T$ is a Tits index with relative Coxeter diagram $B_2$.
\end{examples}

\begin{theorem}\label{nzz8}
Let $\Gamma$ be a descent group of a building $\Delta$. Let $\Pi$ be the
Coxeter diagram of $\Delta$, let $S$ denote the vertex set of $\Pi$
and let $\Theta$ denote the subgroup of ${\rm Aut}(\Pi)$
induced by $\Gamma$. Then the following hold:
\begin{enumerate}[\rm(i)]
\item The graph $\Delta^\Gamma$ defined in {\rm\ref{nzz3}} is a building with respect to a canonical
coloring of its edges.
\item All $\Gamma$-chambers are residues of $\Delta$ of the same type $A\subset S$
and the rank of $\Delta^\Gamma$ is the number of $\Theta$-orbits in $S$ disjoint from $A$.
\item If $A$ is spherical, then the triple
$T:=(\Pi,\Theta,A)$ is a Tits index and $\Delta^\Gamma$ is a building
of type $\tilde\Pi$, where $\tilde\Pi$ is the relative Coxeter diagram of $T$.
\item If $\Delta$ is Moufang and the rank of $\Delta^\Gamma$ is at least~$2$,
then $\Delta^\Gamma$ is also Moufang.
\item Suppose that $\Delta$ is Moufang and that the rank of $\Delta^\Gamma$ is~$1$ and
let $X$ be the set of all $\Gamma$-chambers,
For each $C\in X$, let $\tilde U_C$ denote the subgroup of ${\rm Sym}(X)$ induced
by the centralizer $C_{U_C}(\Gamma)$ of $\Gamma$ in the unipotent radical $U_C$. Then
$$(X,\{\tilde U_C\mid C\in X\})$$
is a Moufang set.
\end{enumerate}
\end{theorem}

\begin{proof}
Assertions (i) and (ii) hold by \cite[22.20(v) and~(viii)]{MPW}, assertion~(iii) holds
by \cite[22.20(iv) and~(viii)]{MPW} and the remaining two
assertions hold by \cite[24.31]{MPW}.
\end{proof}

\begin{definition}\label{nzz9}
The building $\Delta^\Gamma$ in \ref{nzz2}(i) is called the
{\it fixed point building} of $\Gamma$. The rank of $\Delta^\Gamma$ is called
the {\it relative rank} of $\Gamma$.
If the relative rank of $\Gamma$ is~$1$, we interpret $\Delta^\Gamma$ to mean the Moufang set
described in \ref{nzz8}(v).
\end{definition}

\section{Polarities and Galois subgroups}

In this section we describe two ways in which descent groups arise.

\begin{proposition}\label{nzz5}
Let $\Delta$ be a Moufang building of type $\Pi$, where $\Pi$ is the Coxeter diagram
$B_2$, $G_2$ or $F_4$. Suppose that $\sigma$ is a polarity of $\Delta$ as
defined in {\rm\ref{abc3}} and let $\Gamma=\langle\sigma\rangle$. Then
$\Gamma$ is a descent group of $\Delta$ and
$\Gamma$-chambers are chambers of $\Delta$. If $\Pi=B_2$ or $G_2$,
the fixed point building
$\Delta^\Gamma$ is a Moufang set and if $\Pi=F_4$, the fixed point building $\Delta^\Gamma$ is
a Moufang octagon.
\end{proposition}

\begin{proof}
By \ref{abc7}, we can choose an apartment $\Sigma$ stabilized by $\sigma$.
By \ref{nzz25}, the automorphism of $\Sigma$ sending each chamber to its unique opposite is
color-preserving. By \cite[25.17]{MPW}, therefore, there exists a $\Gamma$-residue
$R$ containing chambers of $\Sigma$.
We can assume that $R$ is minimal with respect to containment.

Since $\Gamma$ is not type-preserving,
the type $J$ of $R$ is $\Theta$-invariant, where $\Theta$ is the automorphism group
of $\Pi$. Suppose that $J\ne\emptyset$. Since $J$ is $\Theta$-invariant, we
must have $\Pi=F_4$ and $J$ is either $\{1,4\}$ or $\{2,3\}$
(with respect to the standard numbering of the vertex set of the Coxeter diagram
$F_4$). Thus $R\cap\Sigma$ is a thin building of type $A_1\times A_1$ or $B_2$. By \ref{nzz25},
the map which sends each chamber of $R\cap\Sigma$ to its unique opposite
is again color-preserving. Another application of \cite[25.17]{MPW} thus
implies that $\sigma$ stabilizes a proper residue of $R$. This contradicts
the choice of $R$. With this contradiction, we conclude that there are chambers
of $\Sigma$ fixed by $\sigma$.

Let $c$ be a chamber of $\Sigma$ fixed by $\sigma$ and
let $d$ be the unique chamber of $\Sigma$ opposite $c$. Since $\sigma$
stabilizes $\Sigma$ and $c$, it fixes $d$ as well.
Suppose that $\Pi=B_2$ or $G_2$. Among the roots of $\Sigma$ containing $c$,
there are two such that $c$ is at maximal distance from the root that is opposite in $\Sigma$.
We call these two roots $\alpha$ and $\beta$. They are interchanged by $\sigma$ and
by \cite[5.5 and 5.6]{TW}, we have $U_\alpha\cap U_\beta=1$ and
$[U_\alpha,U_\beta]=1$.
Suppose, instead, that $\Pi=F_4$. Let $c_1$ denote the unique chamber of $\Sigma$ that
is $1$-adjacent to $c$, let $\alpha$ denote the unique root of $\Sigma$ containing
$c$ but not $c_1$ and let $\beta=\alpha^\sigma$. By \cite[11.28(i) and (iii)]{spherical}
with $\{i,j\}=\{1,4\}$, we have $U_\alpha\cap U_\beta=1$ and
$[U_\alpha,U_\beta]=1$ also in this case. We now return to the assumption that $\Pi$ is in
any one of the three cases $B_2$, $G_2$ or $F_4$ and let $u$ be a non-trivial element of $U_\alpha$.
Both $U_\alpha$ and $U_\beta$ are contained in the unipotent radical $U_c$ and
hence $b:=uu^\sigma\in U_c$. Since $U_\alpha\cap U _\beta=1$, we have $b\ne1$ and
since $[U_\alpha,U_\beta]=1$, we have $b^\sigma=b$.
Thus by \ref{onk98}, $d\ne d^b$.
We conclude that $c$, $d$ and $d^b$ are distinct chambers fixed by $\Gamma$.

Suppose that $\Pi$ is $B_2$ or $G_2$. Since the type of a $\Gamma$-residue
is $\Theta$-invariant, $\Delta$ itself is the unique
$\Gamma$-panel. Since there at least three $\Gamma$-chambers, $\Gamma$ is a descent group.
By \ref{nzz8}, the Tits index of $\Gamma$ is $(\Pi,\Theta,\emptyset)$ and
$\Delta^\Gamma$ (interpreted as in \ref{nzz9}) is a Moufang set.

Suppose now that $\Pi=F_4$.
In this case, we let $R_{ij}$ be the unique $\{i,j\}$-residue
containing the chamber $c$ and we let $U_{ij}$ denote the group
induced on $R_{ij}$ by the unipotent radical $U_c$ for $\{i,j\}=\{1,4\}$ and $\{2,3\}$.
Then $R_{14}$ and $R_{23}$ are the two $\Gamma$-panels containing $c$.
Choose $ij=14$ or $23$ and
let $d_{ij}$ be the unique chamber of $R_{ij}\cap\Sigma$ opposite $c$.
Just as above, we can choose a non-trivial element $b_{ij}$ in $U_{ij}$ centralized by
$\sigma$. By \cite[24.8(iii) and 24.33]{MPW}, $U_{ij}$ acts
sharply transitively on the set of chambers of $R_{ij}$ opposite $c$.
Thus $c$, $d_{ij}$ and $d_{ij}^{b_{ij}}$ are
distinct $\Gamma$-chambers contained in $R_{ij}$.
We conclude that both $\Gamma$-panels $R_{14}$ and $R_{23}$ contain at least three $\Gamma$-chambers.
By \cite[22.37]{MPW}, therefore, $\Gamma$ is a descent group of $\Delta$.
By \ref{nzz8}(iii), therefore,
the fixed point building $\Delta^\Gamma$ is of type $\tilde\Pi$,
where $\tilde\Pi$ is the relative Coxeter diagram of the Tits index $(\Pi,\Theta,\emptyset)$.
By \ref{nzz27}, the relative Coxeter diagram of this Tits index is $I_2(8)$.
By \ref{nzz8}(iv), we conclude that $\Delta^\Gamma$ is a Moufang octagon.
\end{proof}

\begin{notation}\label{goo3}
Suppose that $\Delta$ is Moufang.
Let $G^\circ$ denote the group of all type-preserving automorphisms
of $\Delta$, let $G={\rm Aut}(\Delta)$ if $\Delta$ is simply laced
and let $G=G^\circ$ if $\Delta$ is not simply laced. (Thus if $\Delta$ and $\rho$ are
as in \ref{nzz5}, then $\rho\not\in G$.) Let
$G^\dagger$ denote the subgroup of ${\rm Aut}(\Delta)$ generated by all the root groups of $\Delta$.
Root groups are type-preserving, so $G^\dagger\subset G^\circ$.
\end{notation}

\begin{notation}\label{onk1}
Let $\Pi$ be a Coxeter diagram,
let $\Delta$ be a building of type $\Pi$ (as defined
in \cite[7.1]{spherical}) and suppose that $\Delta$ is Moufang as defined in \ref{nzz6x}.
Let $c$ be a chamber of $\Delta$ and let $\Sigma$
be an apartment containing $c$. Let $B_\Pi$ be the set of ordered pairs of $(i,j)$
such that $\{i,j\}$ is an edge of $\Pi$. For each $(i,j)\in B_\Pi$, let $R_{ij}$
be the unique $\{i,j\}$-residue of $\Delta$ containing $c$ and let $\Omega_{ij}$
be the root group sequence of $R_{ij}$ based at $(\Sigma\cap R_{ij},c)$ as defined in
\cite[3.1--3.3]{MPW}.
The first term of $R_{ij}$ acts non-trivially on the $i$-panel of $R_{ij}$ containing $c$.
\end{notation}

\begin{remark}\label{onk3}
Let $(i,j)\in B_\Pi$. Interchanging $i$ and $j$ if necessary,
there exists an isomorphism from $\Omega_{ij}$ to one of the root group sequences
described in \cite[16.1--16.9]{TW} (by the classification of Moufang polygons).
We say that an element $(i,j)$ of $B_\Pi$ is {\it standard} if there is such an isomorphism.
\end{remark}

\begin{notation}\label{onk2x}
Let $F$ be a field of characteristic~$p>0$ and
let $E/F$ be an extension such that $E^p\subset F$. By identifying
$E$ with $E^p$ via ${\rm Frob}_E$, we can regard $F$ as
an extension of $E$. We can recover the extension $E/F$ from
the extension $F/E$ by the same trick. We describe this situation
by saying simply that we have a pair of extensions $\{E/F,F/E\}$.
Let ${\rm Aut}(E,F)$ be the set of all elements of ${\rm Aut}(E)$
stabilizing $F$. This group is canonically isomorphic to the group ${\rm Aut}(F,E)$
of all elements of ${\rm Aut}(F)$ stabilizing $E$.
\end{notation}

\begin{notation}\label{onk2}
Suppose that $\Delta$ is Moufang. By \cite[28.8]{MPW},
the building $\Delta$ has either a field of definition $F$ or a pair $\{E/F,F/E\}$
of defining extensions as in \ref{onk2x}. In the first case, we set $A={\rm Aut}(F)$ and in
the second case, we let $A$ denote the group ${\rm Aut}(E,F)$ defined in \ref{onk2x}.
\end{notation}

\begin{remark}\label{goo5x}
If the building $\Delta$ has a pair $\{E/F,F/E\}$ of defining extensions
rather than a field of definition $F$, then $\Delta\cong B_2^{\mathcal D}(\Lambda)$
for some indifferent set $\Lambda$, $\Delta\cong B_2^{\mathcal F}(\Lambda)$ for some
quadratic space $\Lambda$ of type $F_4$, $\Delta\cong G_2(\Lambda)$ for some inseparable hexagonal
system $\Lambda$ or $\Delta\cong F_4(\Lambda)$ for some
inseparable composition algebra $\Lambda$. (See \cite[30.15 and 30.23]{affine} for the definition
of these terms.) By \cite[35.9, 35.12 and 35.13]{TW}, the pair $\{E/F,F/E\}$ is an invariant
of $\Delta$ in all these cases even though neither $E$ nor $F$ is an invariant.
(By \cite[35.6-35.8,35.10, 35.11 and 35.14]{TW}, the field of definition $F$ {\it is} an
invariant of $\Delta$ in every other case. Thus the group $A$ is an invariant of $\Delta$
in {\it every} case.)
\end{remark}

\begin{notation}\label{goo5}
Suppose that $\Delta$ is Moufang, let $G$ be as in \ref{goo3}, let $\Sigma$, $c$, $B_\Pi$
and $\Omega_{ij}$ be as in \ref{onk1} and let $A$ be as in \ref{onk2}. Let $(s,t)$ be a standard
element of $B_\Pi$ as defined in \ref{onk3} and let $\varphi$ be
an isomorphism from $\Omega_{st}$ to $\Theta$, where $\Theta$ is
one of the root group sequences described in \cite[16.1--16.9]{TW}.
For each $g\in G$ acting trivially on $\Sigma$,
let $g_{st}$ denote the automorphism of $\Omega_{st}$ induced
by $g$ and let $g^*_{st}$ denote the automorphism $\varphi^{-1}g_{st}\varphi$ of $\Theta$.
By \cite[(29.22) and 29.23--29.25]{MPW},
there exists a unique homomorphism $\psi$ from $G$ to $A$ such that the following hold:
\begin{enumerate}[\rm(i)]
\item $G^\dagger$ is contained in the kernel of $\psi$.
\item For each $g\in G$ acting trivially on $\Sigma$,
$\psi(g)$ is equal to the element called $\lambda_\Omega(h)$ in \cite[29.5]{MPW},
where $\Omega=\Theta$ and $h=g^*_{st}$.
\item \cite[4.7(iii)]{rhizo} holds for all non-type-preserving elements $g\in G$.
\end{enumerate}
(We are only interested here in applying $\psi$ to type-preserving elements,
so we do not take the trouble to state (iii) more precisely.)
A homomorphism $\psi\colon G\to A$ satisfying conditions (i)--(iii) for some choice
of $(s,t)$ and $\varphi$
is called a {\it Galois map} of $\Delta$. If $\psi$ and $\psi_1$ are two
Galois maps of $\Delta$, then there is an inner automorphism $\iota$ of $A$ such
that $\psi_1=\psi\cdot\iota$ (by \cite[29.25]{MPW}).
Thus, in particular, all Galois maps of $\Delta$ have the same kernel.
\end{notation}

\begin{definition}\label{goo6}
A {\it Galois subgroup} of ${\rm Aut}(\Delta)$ is a subgroup $\Gamma$ of the group $G$ defined
in \ref{goo3} whose intersection with the kernel of a Galois map of $\Delta$ is trivial.
Since two Galois maps differ by an inner automorphism of the group $A$,
this notion is independent of the choice of the Galois map.
\end{definition}

\begin{theorem}\label{toy1}
Suppose that $\Delta$ is Moufang and that $\Gamma$ is an isotropic
Galois subgroup of ${\rm Aut}(\Delta)$
that acts on the set of chambers of $\Delta$ with finite orbits.
Then $\Gamma$ is a descent group of $\Delta$.
\end{theorem}

\begin{proof}
This is \cite[12.2(ii)]{rhizo}.
\end{proof}

\begin{notation}\label{nzz10}
Let $\Delta$ be Moufang and let $\psi$ be a Galois map of $\Delta$.
A {\it Galois involution} $g$ of $\Delta$ is an element of order~$2$ in ${\rm Aut}(\Delta)$
such that $\langle g\rangle$ is a Galois subgroup. A {\it $\chi$-involution} of $\Delta$
for some $\chi\in A$ is a Galois involution $g$ such that $\chi=\psi(g)$.
\end{notation}

\begin{proposition}\label{nzz22}
Let $\Delta$ be a building of type $F_4$ and
suppose that $\Gamma$ is a type-preserving
Galois subgroup of ${\rm Aut}(\Delta)$ acting on the set of chambers of $\Delta$ with finite orbits
such that $\Gamma$-chambers are residues
of type $\{2,3\}$ with respect to the standard numbering of the vertex set of the Coxeter diagram
$F_4$. Then $\Gamma$ is a descent group and $\Delta^\Gamma$
is a Moufang quadrangle of type $F_4$.
\end{proposition}

\begin{proof}
By \ref{toy1}, $\Gamma$ is a descent group.
Let $T$ denote the Tits index $(\Pi,\Theta,A)$, where $\Pi$
is the Coxeter diagram $F_4$, $\Theta$ is the trivial subgroup of ${\rm Aut}(\Pi)$
and $A=\{2,3\}$. By \ref{nzz27}, the relative Coxeter diagram of this index is $B_2$.
By \ref{nzz8}(iii), the fixed point building $\Delta^\Gamma$ is of type $B_2$.
By \ref{nzz8}(iv), we conclude that $\Delta^\Gamma$ is a Moufang quadrangle.
By \ref{onk13x}(ii) below,
$\Delta^\Gamma$ is, in fact, a Moufang quadrangle of type $F_4$.
\end{proof}

\begin{proposition}\label{nzz37}
Let $\Omega={\mathcal O}(E,\theta)$ for some octagonal set $(E,\theta)$ and let
$\Gamma$ be a Galois subgroup of ${\rm Aut}(\Omega)$
acting on the set of chambers of $\Delta$ with finite orbits and
fixing panels of one type but none of the other. Then $\Gamma$ is a descent group and
$\Omega^\Gamma$ is a Moufang set.
\end{proposition}

\begin{proof}
This holds by \ref{nzz8}(v) and \ref{toy1}.
\end{proof}

\section{Main Results}\label{nzz33}

We can now state our main results. The proofs of \ref{abc1} and \ref{abc2} are
in \S\ref{nzz40}.

\begin{theorem}\label{abc1}
Let $\Xi$ be a Moufang quadrangle of type $F_4$ with a polarity $\rho$.
Then there exists an octagonal set $(E,\theta)$
and an automorphism $\chi$ of the field $E$ of order~$2$
that commutes with the Tits endomorphism $\theta$ such that the following hold:
\begin{enumerate}[\rm(i)]
\item The building $\Delta:=F_4(E,\theta)$ possesses a type-preserving $\chi$-involution
$\xi$ and a polarity $\sigma$ such that
$$\Gamma:=\langle\xi,\sigma\rangle\subset{\rm Aut}(\Delta)$$
is a descent group of order~$4$.
\item There exists an isomorphism from $\Xi$ to the fixed point building
$\Delta^{\langle\xi\rangle}$ which carries the polarity $\rho$ to the restriction of $\sigma$ to
$\Delta^{\langle\xi\rangle}$.
\item The fixed point buildings $\Delta^{\langle\sigma\rangle}$ and
$\Delta^{\langle\sigma\xi\rangle}$ are Moufang octagons,
one isomorphic to ${\mathcal O}(E,\theta)$ and the other to ${\mathcal O}(E,\theta\chi)$.
\item The Moufang sets $\Delta^\Gamma$, $(\Delta^{\langle\xi\rangle})^{\langle\sigma\rangle}$,
$(\Delta^{\langle\sigma\rangle})^{\langle\xi\rangle}$
and $(\Delta^{\langle\sigma\xi\rangle})^{\langle\xi\rangle}$ are canonically isomorphic.
\item The restriction of $\xi$ to each of the two octagons in {\rm(iii)} is a $\chi$-involution
which fixes panels of one type and none of the other.
\end{enumerate}
\end{theorem}

\begin{theorem}\label{abc2}
Let $(E,\theta)$ be an octagonal set, let $\chi$ be an automorphism of $E$ of order~$2$,
let $\Omega={\mathcal O}(E,\theta)$, let $\Delta=F_4(E,\theta)$ and suppose that
there exists a $\chi$-involution $\kappa$ of $\Omega$ which
fixes panels of one type but not of the other type.
Then there is a type-preserving $\chi$-involution $\xi$ of $\Delta$ and a polarity
$\sigma$ of $\Delta$ such that the following hold:
\begin{enumerate}[\rm(i)]
\item $\Gamma=\langle\xi,\sigma\rangle\subset{\rm Aut}(\Delta)$ is a descent group of
$\Delta$ of order~$4$.
\item There is an isomorphism
from $\Omega$ to the fixed point building $\Delta^{\langle\sigma\rangle}$ which carries $\kappa$ to
the restriction of $\xi$ to $\Delta^{\langle\sigma\rangle}$.
\item $\langle\xi\rangle$-chambers are
residues of type $\{2,3\}$ with respect to the standard numbering of the vertex set of
the Coxeter diagram $F_4$,
the fixed point building $\Xi:=\Delta^{\langle\xi\rangle}$ is a Moufang quadrangle of type~$F_4$
and the polarity $\sigma$ of $\Delta$ induces a polarity $\rho$ on $\Xi$.
\item The Moufang sets $\Delta^\Gamma$, $(\Delta^{\langle\xi\rangle})^{\langle\sigma\rangle}$,
$(\Delta^{\langle\sigma\rangle})^{\langle\xi\rangle}$
and $(\Delta^{\langle\sigma\xi\rangle})^{\langle\xi\rangle}$ are canonically isomorphic.
\item The automorphism $\chi$ commutes with the Tits endomorphism $\theta$ and
the fixed point building
$\Delta^{\langle\xi\sigma\rangle}$ is isomorphic to ${\mathcal O}(E,\theta\chi)$.
\end{enumerate}
\end{theorem}

\begin{corollary}\label{nzz28}
The class of Moufang sets of the form $\Xi^{\langle\rho\rangle}$, where $\Xi$ and $\rho$ are as in
{\rm\ref{abc1}}, coincides with the class of Moufang sets of the form $\Omega^{\langle\kappa\rangle}$,
where $\Omega$ and $\kappa$ are as in {\rm\ref{abc2}}.
\end{corollary}

\begin{proof}
Let $\Xi$, $\rho$, $\Delta$ and $\Gamma$
be as in \ref{abc1}. By \ref{abc1}(iv), $\Xi^{\langle\rho\rangle}\cong\Delta^\Gamma$
and hence by \ref{abc1}(v), $\Xi^{\langle\rho\rangle}$ is isomorphic to a Moufang set of the form
$\Omega^{\langle\kappa\rangle}$, where $\Omega$ and $\kappa$ are as in \ref{abc2}. Suppose,
conversely, that $\Omega$, $\kappa$, $\Delta$ and $\Gamma$ are as in \ref{abc2}. By
\ref{abc2}(iv), $\Omega^{\langle\kappa\rangle}\cong\Delta^\Gamma$. By \ref{abc2}(iii), it follows
that $\Omega^{\langle\kappa\rangle}$ is isomorphic to a Moufang set of the form $\Xi^{\langle\rho\rangle}$,
where $\Xi$ and $\rho$ are as in \ref{abc1}.
\end{proof}

\begin{remark}\label{onk90}
In both \ref{abc1} and \ref{abc2}, we are using the notion of a $\chi$-involution
(defined in \ref{nzz10})
with respect to Galois maps $\psi_\Delta$, $\psi_\Omega$ and $\psi_{\Omega_\xi}$
as in \ref{onk30x} and \ref{onk32}(iv).
\end{remark}

\begin{definition}\label{nzz29}
We call a Moufang set of the form $\Delta^\Gamma$,
where $\Delta$ and $\Gamma$ are as in \ref{abc1} (or \ref{abc2})
a Moufang set {\it of outer $F_4$-type}
(to distinguish them from other Moufang sets which arise as
the fixed point buildings of type-preserving descent groups of a building of type~$F_4$;
see \cite{tom-vm}).
\end{definition}

\begin{remark}\label{onk95}
With the conventions described in \cite[34.2]{MPW},
$$\includegraphics[scale=.7]{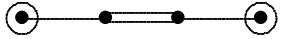}$$
is the Tits index of $\langle\xi\rangle$,
$$\includegraphics[scale=.7]{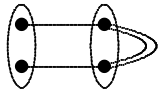}$$
is the Tits index of $\langle\sigma\rangle$ and $\langle\sigma\xi\rangle$ and
$$\includegraphics[scale=.7]{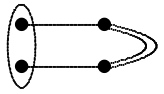}$$
is the Tits index of $\Gamma$, where $\Gamma=\langle\xi,\sigma\rangle$ is as in \ref{abc1}
(or \ref{abc2}).
\end{remark}

The following is proved in \S\ref{nzz60}.

\begin{theorem}\label{nzz13}
The group generated by all the root groups of a Moufang set of outer $F_4$-type is simple.
\end{theorem}

\section{Moufang Quadrangles of Type $F_4$}\label{sec4}

The Moufang quadrangles of type $F_4$ were first described in \cite[Chapter~14 and 16.7]{TW};
see also \cite{tom1}, \cite{tom2} and \cite{qa}. In \cite[Chapter~16]{MPW}, it is shown
that these quadrangles arise in the study of pseudo-reductive quotients of parahoric subgroups
of groups of absolute type $E_6$, $E_7$ and $E_8$.

\begin{definition}\label{def0}
Let $K$ be a field of characteristic~$2$ and let $F$ be a subfield such that
$K^2\subset F$. We set $t*s={\rm Frob}_K(t)s=t^2s$ for all $t\in K$ and all $s\in F$
and $q_{F/K}(s):=s$ for all $s\in F$. The map $*$
makes $F$ into a vector space over $K$ on which the map $q_{F/K}$ is
a quadratic form. We write $[F]_K$ to refer to $F$ considered as a vector space over $K$ with
respect to $*$.
\end{definition}

\begin{notation}\label{def0a}
An {\it $F_4$-datum} is a 4-tuple $S=(E/K,F,\alpha,\beta)$, where
$E/K$ is a separable quadratic extension with ${\rm char}(K)=2$,
$F$ is a subfield of $K$ containing $K^2$, $\alpha$ is a non-zero element of $F$,
and $\beta$ is a non-zero element of $K$,
such that the quadratic form on $V_S:=E\oplus E\oplus[F]_K$
given by
\begin{align*}
(a,b,s)\mapsto&\beta^{-1}(N(a)+\alpha N(b))+q_{F/K}(s)\\
=&\beta^{-1}(N(a)+\alpha N(b))+s
\end{align*}
is anisotropic, where $N=N_{E/K}$ is the norm of the extension $E/K$
and $q_{F/K}$ is as in \ref{def0}. We denote this quadratic form by $q_S$.
\end{notation}

\begin{definition}\label{def1}
A quadratic form $q$ over a field $K$
is of {\it type $F_4$} if $q$ is similar to $q_S$ for
some $F_4$-datum $(E/K,F,\alpha,\beta)$. If this case, we will say that
$S$ is a {\it standard decomposition} of $q$
and that $E$ is a {\it splitting field} of $q$. We say that a quadratic space
$(K,V,q)$ is of type $F_4$ if the quadratic form $q$ is of type $F_4$.
\end{definition}

\begin{definition}\label{def2}
Let $S=(E/K,F,\alpha,\beta)$ be an $F_4$-datum,
let $D$ denote the composite field $FE^2$ and let $[K]_F$ denote $K$ regarded as
a vector space over its subfield $F$ in the standard way. By \cite[14.8]{TW},
the quadratic form $\hat q_S$ on $\hat V_S:=D\oplus D\oplus[K]_F$ over $F$ given by
\begin{equation}\label{def2x}
(x,y,t)\mapsto\alpha(N(x)+\beta^2 N(y))+t^2
\end{equation}
is a quadratic form of type $F_4$ with standard decomposition $(D/F,K^2,\beta^2,\alpha^{-1})$.
\end{definition}

\begin{notation}\label{abc83}
Let $S=(E/K,F,\alpha,\beta)$ be an $F_4$-datum, let $q:=q_S$ and $V=V_S$ be as in \ref{def0a}
and let $\hat V=\hat V_S$ and $\hat q:=\hat q_S$ be as in \ref{def2}.
Let $f$ and $\hat f$ denote the bilinear forms $\partial q$ and $\partial\hat q$ associated
with $q$ and $\hat q$ and
let $x\mapsto\bar x$ be the non-trivial element of ${\rm Gal}(E/K)$.
\end{notation}

We now introduce the Moufang quadrangle that corresponds to this data.

\begin{notation}\label{abc61}
Let $S$, $\hat V$, $V$, etc. be as in \ref{abc83}, let
$$U_+:=U_1U_2U_3U_4$$
be the group defined in terms of the isomorphisms $x_i\colon V\to U_i$
for $i=2$ and $4$ and $x_i\colon\hat V\to U_i$ for
$i=1$ and $3$ the following commutator
relations taken from \cite[16.7]{TW}):
\begin{align*}
[x_1(x,y,t), x_3(x',y',t')]
&= x_2 \bigl(0,\ 0,\ \alpha \bigl(x\bar{x}'+x'\bar{x}+\beta^2(y\bar{y}'+y'\bar{y})\bigr) \bigr),\\
[x_2(u,v,s), x_4(u',v',s')]
&=x_3\bigl(0,\ 0,\ \beta^{-1}\bigl(u\bar{u}'+u'\bar{u}+\alpha(v\bar{v}'+v'\bar{v}) \bigr) \bigr), \\
[x_1(x,y,t), x_4(u,v,s)]
&=x_2\bigl(tu+\alpha(\bar{x}v+\beta y\bar{v}),\ tv+xu+\beta y\bar{u}, \\
&\hspace*{10ex} t^2s + s\alpha(x\bar{x}+\beta^2y\bar{y})\\
&\hspace*{15ex}
+ \alpha \bigl( u^2x\bar{y}+\bar{u}^2\bar{x}y+\alpha(\bar{v}^2xy+v^2\bar{x}\bar{y})\bigr)\bigr)\\
& \hspace*{2.6ex} \cdot x_3 \bigl( sx+\bar{u}^2y+\alpha v^2\bar{y},
\ sy+\beta^{-2}(u^2x+\alpha v^2\bar{x}), \\
&\hspace*{10ex} st+t\beta^{-1}(u\bar{u}+\alpha v\bar{v})\\
&\hspace*{15ex}
+ \alpha \bigl( \beta^{-1}(xu\bar{v}+\bar{x}\bar{u}v)+y\bar{u}\bar{v}+\bar{y}uv \bigr) \bigr)
\end{align*}
for all $(u,v,s),(u',v',s')\in V$ and all $(x,y,t),(x',y',t')\in\hat V$ and
$$[U_1,U_2]=[U_2,U_3]=[U_3,U_4]=1.$$
The group $U_+$, its subgroups $U_1,\ldots,U_4$ and the isomorphisms
$x_1,\ldots,x_4$ depend only on the $F_4$-datum $S$. By \cite[16.7 and 32.11]{TW},
$$(U_+,U_1,U_2,U_3,U_4)$$
is a root group sequence. Let
$${\mathcal Q}(S)$$
denote the Moufang quadrangle,
$\Sigma$ the apartment of ${\mathcal Q}(S)$ and $c$ the chamber of $\Sigma$ obtained
by applying \cite[8.3]{TW} to this root group sequence.
There is a canonical identification of $U_1,\ldots,U_4$ with the
root groups of $\Xi:={\mathcal Q}(S)$
associated with the four roots of $\Sigma$ containing $c$ and we always identify
$U_+$ with the subgroup of ${\rm Aut}(\Xi)$ generated by these four root groups.
\end{notation}

\begin{definition}\label{abc85}
A {\it Moufang quadrangle of type $F_4$} is a Moufang quadrangle isomorphic
to ${\mathcal Q}(S)$ (as defined in \ref{abc61}) for some $F_4$-datum $S$.
\end{definition}

\begin{notation}\label{abc85y}
Let $S=(E/K,F,\alpha,\beta)$ and $\Omega:=(U_+,U_1,\ldots,U_4)$ be as in \ref{abc61}.
By \cite[28.45]{TW}, there is an anti-isomorphism (as defined in \cite[8.9]{TW})
from $\Omega$ to the root group sequence obtained by applying \ref{abc61} to the $F_4$-datum
$(D/F,K^2,\beta^2,\alpha^{-1})$ in \ref{def2}.
\end{notation}

\begin{notation}\label{abc85x}
Let $S$ and $(K,V,q)$ be as in \ref{abc83} and let $\Omega=(U_+,U_1,\ldots,U_4)$ and
$\Xi={\mathcal Q}(S)$ be as in
\ref{abc61}. The quadrangle $\Xi$
is  called ${\mathcal Q}_{\mathcal F}(K,V,q)$
in \cite[16.7]{TW} and $B_2^{\mathcal F}(K,V,q)$ in \cite[30.15]{affine}.
Suppose that $S'$ is any other $F_4$-datum and let $\Omega'=(U_+',U_1',\ldots,U_4')$
be the root group sequence obtained by applying \ref{abc61} to $S'$.
Then by \cite[35.12]{TW}, there is a type-preserving isomorphism
from $\Omega$ to $\Omega'$ if and only if the quadratic form $q_{S'}$ is
similar to $q$, where $q_{S'}$ is the quadratic form obtained by applying \ref{def0a} to $S'$.
In particular, $\Xi\cong{\mathcal Q}(S')$ for every standard decomposition $S'$ of $q$
(as defined in \ref{def1}).
\end{notation}

\begin{remark}\label{onk105}
Let $\Xi={\mathcal Q}(S)$ for some $F_4$ datum $S=(E/K,F,\alpha,\beta)$.
By \cite[28.4]{MPW} (see \ref{onk2}), $\{K/F,F/K\}$ is the pair of defining extensions of $\Xi$.
By \cite[35.12]{TW}, it is an invariant of $\Xi$.
\end{remark}

\begin{remark}\label{def5}
Let $s_0\in F^*$ and let
$S_1=(E/K,F,\alpha,\beta/s_0)$. Then $S_1$ is an $F_4$-datum,
$V_{S_1}=V$, $\hat V_{S_1}=\hat V$ and
the maps $x_1(x,y,t)\mapsto x_1(x,s_0y,t)$,
$x_3(x,y,t)\mapsto x_3(x/s_0,y,t/s_0)$ and $x_i(u,v,s)\mapsto x_i(u/s_0,v/s_0,s/s_0)$
for $i=2$ and $4$ extend to an isomorphism from ${\mathcal Q}(S)$ to ${\mathcal Q}(S_1)$
(by \cite[7.5]{TW}). Thus by reparametrizing $U_+$, we can
replace the element $x_4(0,0,s_0)$ by $x_4(0,0,1)$ without changing the
element $x_1(0,0,1)$.
\end{remark}

\begin{notation}\label{abc80}
We define two maps, one from $V\times\hat V$ to $V$ and the other
from $\hat V\times V$ to $\hat V$, both denoted either by $\cdot$ or juxtaposition,
so that
\begin{equation}\label{abc9}
[x_1(\hat v),x_4(v)]=x_2(v\cdot\hat v)x_3(\hat v\cdot v)
\end{equation}
in $U_+$ for all $(\hat v,v)\in\hat V\times V$.
(Note that we will also denote scalar multiplication by $\cdot$ or juxtaposition,
but this should not cause any confusion.)
\end{notation}

\begin{remark}\label{tru2}
Let $(K,V,q)$ and $f$ be as in \ref{abc83} and
let $d,e$ be linearly independent elements of $V$ such that $f(d,e)=1$.
Then $q(d)x^2+x+q(e)=q(xd+e)\ne0$ for all $x\in K$ since $q$ is anisotropic,
and the restriction of $q/q(d)$ to $\langle e,d\rangle$ is
isometric to the norm of the quadratic extension $L/K$,
where $L$ is the splitting field of the polynomial
$q(d)x^2+x+q(e)$ over $K$. Note that $L$ is also the splitting field of the polynomial
$x^2+x+q(d)q(e)$ over $K$.
\end{remark}

\begin{theorem}\label{abc77}
Let $(K,V,q)$, $\hat V$ and $f$ be as in {\rm\ref{abc83}} and let
$(U_+,U_1,\ldots,U_4)$ and $x_1,\ldots,x_4$ be as in {\rm\ref{abc61}}.
Let $d,e$ be two elements of $V$ and let $\xi$ be an element of $\hat V$ such that
$f(d,e)=1$ and $f(d,e\xi)=0$. Let $\alpha_0=f(d\xi,e\xi)$, let $\beta_0=q(d)^{-1}$,
let $L$ be the splitting field of the polynomial $p(x)=q(d)x^2+x+q(e)$ over
$K$ and let $\omega$ be a root of $p(x)$ in $L$. Then the following hold:
\begin{enumerate}[\rm(i)]
\item $S_0:=(L/K,F,\alpha_0,\beta_0)$ is a standard decomposition of $q$.
\item There exists an isometry $\pi$ from $(K,V,q)$ to $(K,V_{S_0},q_{S_0})$
sending the elements $d$, $e$, $d\xi$ and $e\xi$ to $(1,0,0)$, $(\omega,0,0)$, $(0,1,0)$ and
$(0,\omega,0)$, respectively, and $(0,0,s)$ to $(0,0,s)$ for all $s\in F$.
\item There exists an isometry $\hat\pi$ from $(F,\hat V,\hat q)$ to
$(F,\hat V_{S_0},\hat q_{S_0})$
sending the elements $\xi$, $\xi e\cdot d^{-1}$, $\xi e\cdot d^{-1}$ and $\xi d^{-1}$
to $(1,0,0)$, $(\omega^2,0,0)$, $(0,1,0)$ and $(0,\omega^2,0)$, respectively, and
$(0,0,t)$ to $(0,0,t)$ for all $t\in K$.
\item Let $(\tilde U_+,\tilde U_1,\ldots,\tilde U_4)$ and $\tilde x_1,\ldots,\tilde x_4$
be the root group sequence and the isomorphisms obtained by applying {\rm\ref{abc61}}
to $S_0$. Then there is an isomorphism from $U_+$ to $\tilde U_+$ extending the maps
$x_i(v)\mapsto\tilde x_i(\pi(v))$ for $i=1$ and $3$ and $x_i(v)\mapsto\tilde x_i(\hat\pi(v))$
for $i=2$ and $4$.
\end{enumerate}
\end{theorem}

\begin{proof}
This is proved in \cite[8.98-8.106]{tom2}. See, in particular,
the equations at the top of page~77 in \cite{tom2}.
\end{proof}

\begin{notation}\label{abc25}
Let $[s]_K=(0,0,s)\in V$ for each $s\in F$ and $[t]_F=(0,0,t)\in\hat V$ for each $t\in K$.
Thus $t[s]_K=[t^2s]_K$ and $s[t]_F=[st]_F$ for all $s\in F$ and all $t\in K$.
\end{notation}

\begin{proposition}\label{abc30}
The following identities hold for all $t\in K$, all $s\in F$, all $u,v,w\in V$ and all
$\hat u,\hat v,\hat w\in \hat V$:
\begin{itemize}
\item[\rm(F0)] $x\mapsto x\hat w$ and $\hat x\mapsto\hat xw$ are linear maps from $V$ to $V$
and from $\hat V$ to $\hat V$.
\item[(F1)] $v[t]_F = tv$.
\item[(F2)] $\hat v[s]_K=s\hat v$.
\item[(F3)] $v\cdot s\hat w = v\hat w\cdot [s]_F$.
\item[(F4)] $\hat v\cdot tw = \hat vw\cdot [t^2]_K$.
\item[(F5)] $[t]_Fv = [tq(v)]_F$.
\item[(F6)] $[s]_K\hat v = [s\hat q(\hat v)]_K$.
\item[(F7)] $v\hat w\hat w = v\cdot[\hat q(\hat w)]_F$
\item[(F8)] $\hat vww = \hat v\cdot [q(w)^2]_K$.
\item[(F9)] $v\cdot \hat wv = q(v)v\hat w$.
\item[(F10)] $\hat v\cdot w\hat v = \hat q(\hat v)\hat vw$.
 \item[(F11)] $w(\hat u+\hat v)=w\hat u+w\hat v+[\hat f(\hat uw,\hat v)]_K$.
\item[(F12)] $\hat w(u+v)=\hat wu+\hat wv+[f(u\hat w,v)]_F$.
\end{itemize}
\end{proposition}

\begin{proof}
Comparing \cite[16.7]{TW} with \eqref{abc9}, we can write the products $\hat v\cdot v$
and $v\cdot\hat v$ in terms of the functions
given in \cite[14.15--14.16]{TW}. The identities (F0)--(F12) can then be verified
with the help of the identities in \cite[14.18]{TW}.
\end{proof}

\noindent
The identities (F0)--(F12) are the axioms of a {\it radical quadrangular system}
as defined in \cite[Appendix A.3.2]{tom2}. These axioms can be used
to characterize Moufang quadrangles of type $F_4$ (defined in \ref{abc85} below);
see \cite[\S8.5]{tom2} and \cite[Chapter~28]{TW} for details.

\begin{notation}\label{abc12}
Using \ref{abc80}, we can re-write
the commutator relations in \ref{abc61} as follows:
\begin{align*}
[x_1(\hat v), x_3(\hat u)] &=x_2([\hat f(\hat v, \hat u)]_K)\\
[x_2(v), x_4(u)] &=x_3([f(v, u)]_F)\\
[x_1(\hat v), x_4(v)] &=x_2(v\hat v)x_3(\hat v v)\notag
\end{align*}
for all $u,v\in V$ and all $\hat u,\hat v\in\hat V$ as well as $[U_1,U_2]=[U_2,U_3]=[U_3,U_4]=1$.
\end{notation}

\section{The Polarity $\rho$}\label{sec4a}

We continue with all the notation of the previous section.

\begin{hypothesis}\label{abc84}
We suppose now that the Moufang quadrangle $\Xi={\mathcal Q}(S)$ introduced in \ref{abc61}
has a polarity $\rho$.
\end{hypothesis}

\begin{remark}\label{abc84x}
By \ref{abc7}, the polarity $\rho$ fixes an apartment. Since $\rho$ is a non-type-preserving
involution and apartments are circuits of length~$8$, $\rho$ fixes
two opposite chambers of this apartment. Since ${\rm Aut}(\Xi)$
acts transitively on incident pairs of apartments and chambers (by \cite[11.12]{spherical}),
we can assume that $\rho$ fixes the apartment $\Sigma$ and the chamber $c$
in \ref{abc61}. This means that
\begin{equation}\label{abc10}
\rho U_i\rho=U_{5-i}
\end{equation}
for all $i\in[1,4]$.
\end{remark}

\begin{notation}\label{abc97}
Let $\hat\varphi,\ \hat\varphi_1\colon \hat V\to V$ and
$\varphi,\ \varphi_1\colon V\to \hat V$ be the unique
additive bijections such that
\begin{align*}
\rho(x_1(\hat v)) &=x_4(\hat\varphi(\hat v))\\
\rho(x_2(v)) &= x_3(\varphi_1(v))\\
\rho(x_3(\hat v))&=x_2(\hat\varphi_1(\hat v))\\
\rho(x_4(v))&=x_1(\varphi(v))
\end{align*}
for all $v\in V$ and all $\hat v\in\hat V$.
By \ref{def5}, we can assume that $\varphi([1]_K)=[1]_F$. Since
$\rho$ is of order~$2$, we have
$$\hat\varphi=\varphi^{-1}\text{ and } \hat\varphi_1=\varphi_1^{-1}.$$
\end{notation}

\begin{lemma}\label{abc8}
The following hold:
\begin{enumerate}[\rm(i)]
\item $\varphi=\varphi_1$.
\item $\varphi([\hat q(\varphi(v))]_K)=[q(v)]_F$ for all $v\in V$.
\end{enumerate}
\end{lemma}

\begin{proof}
Let $v\in V$.
By \eqref{abc9}, (F1) and (F5), we have
$$[x_1([1]_F),x_4(v)]=x_2(v)x_3([q(v)]_F).$$
Applying $\rho$, we obtain
$$[x_1(\varphi(v)),x_4([1]_K)]=x_2(\varphi_1^{-1}([q(v)]_F))x_3(\varphi_1(v)).$$
By \eqref{abc9}, (F2) and (F6), on the other hand,
$$[x_1(\varphi(v)),x_4([1]_K)]=x_2([\hat q(\varphi(v))]_K)x_3(\varphi(v)).$$
Therefore $\varphi(v)=\varphi_1(v)$ and $\varphi_1^{-1}([q(v)]_F)=[\hat q(\varphi(v))]_K$.
\end{proof}

\begin{lemma}\label{abc14}
$\varphi^{-1}([f(u,v)]_F)=[\hat f(\varphi(u),\varphi(v))]_K$ for all $u,v\in V$
and
$$\varphi([\hat f(\hat u,\hat v)]_K)=[f(\varphi^{-1}(\hat u),\varphi^{-1}(\hat v))]_F$$
for all $\hat u,\hat v\in\hat V$. In particular,
$\varphi([F]_K)=[K]_F$ and $\varphi^{-1}([K]_F)=[F]_K$.
\end{lemma}

\begin{proof}
This follows from \ref{abc8}(ii).
\end{proof}

\begin{proposition}\label{abc15}
The following hold for all $v\in V$ and all $\hat v\in\hat V$:
\begin{enumerate}[\rm(i)]
\item $\varphi(v\hat v)=\varphi(v)\varphi^{-1}(\hat v)$.
\item $\varphi^{-1}(\hat v v)=\varphi^{-1}(\hat v)\varphi(v)$.
\end{enumerate}
\end{proposition}

\begin{proof}
Applying $\rho$ to the identity \eqref{abc9}, we obtain both claims by
\ref{abc8}(i).
\end{proof}

\begin{notation}\label{abc40}
By \ref{abc14}, there exists a unique additive bijection $\theta$ from $K$ to $F$
such that $\varphi^{-1}([t]_F)=[t^\theta]_K$. Note that by \ref{abc14},
this means that
\begin{equation}\label{abc69}
\hat f(\varphi(u),\varphi(v))=f(u,v)^\theta
\end{equation}
for all $u,v\in V$.
\end{notation}

\begin{proposition}\label{abc41}
The following hold:
\begin{enumerate}[\rm(i)]
\item The map $\theta$ defined in {\rm\ref{abc40}} is a Tits endomorphism of $K$.
\item $\varphi(tv)=t^\theta\varphi(v)$ for all $v\in V$ and all $t\in K$.
\item $q(v)^\theta=\hat q(\varphi(v))$ for all $v\in V$.
\item $u\cdot\varphi(tv)=t^\theta\cdot u\varphi(v)$ for all $u,v\in V$ and all $t\in K$.
\end{enumerate}
\end{proposition}

\begin{proof}
Choose $t\in K$ and $s\in F$. By \ref{abc25} and (F2), we have $[t]_F\cdot[s]_K=[st]_F$.
Applying $\varphi^{-1}$, we obtain $[t^\theta]_K\cdot[s^{\theta^{-1}}]_F=[(st)^\theta]_K$.
By \ref{abc25} and (F1), on the other hand, we have
$[t^\theta]_K\cdot[s^{\theta^{-1}}]_F=[(s^{\theta^{-1}})^2t^\theta]_K$.
Therefore
\begin{equation}\label{abc41a}
(ts)^\theta=(s^{\theta^{-1}})^2t^\theta.
\end{equation}
Setting $t=1$ in \eqref{abc41a}, we obtain
\begin{equation}\label{abc41b}
s^\theta=(s^{\theta^{-1}})^2.
\end{equation}
Substituting this back into \eqref{abc41a}, we then have
\begin{equation}\label{abc41c}
(ts)^\theta=s^\theta t^\theta.
\end{equation}
Let $x=s^{\theta^{-1}}$. Substituting $x^\theta$ for $s$ in \eqref{abc41b},
we conclude that
\begin{equation}\label{abc41d}
(x^\theta)^\theta=x^2.
\end{equation}
Thus
\begin{equation}\label{abc41e}
(x^2)^\theta=(x^{\theta^2})^\theta=(x^\theta)^{\theta^2}=(x^\theta)^2
\end{equation}
for all $x\in K$.
Choose $u\in K$. Then $u^2$ and $t^2$ are in $F$, so
$$(t^2u^2)^\theta=(t^2)^\theta(u^2)^\theta$$
by \eqref{abc41c}. By \eqref{abc41e}, it follows
that $\theta$ is multiplicative. By \eqref{abc41d}, therefore,
$\theta$ is a Tits endomorphism. Thus (i) holds.

Now let $v\in V$ and $t\in K$. Then
\begin{alignat}{2}
\varphi(tv)&=\varphi(v[t]_F)& &\qquad\text{by (F1)}\notag\\
&=\varphi(v)\varphi^{-1}([t]_F)& &\qquad\text{by \ref{abc15}(i)}\notag\\
&=\varphi(v)[t^\theta]_K=t^\theta\varphi(v)& &\qquad\text{by (F2),}\notag
\end{alignat}
so (ii) holds, and
\begin{alignat}{2}
[q(v)^\theta]_K&=\varphi^{-1}[q(v)]_F\notag\\
&=[\hat q(\varphi(v))]_K& &\qquad\text{by \ref{abc8}(ii),}\notag
\end{alignat}
so (iii) holds.

Now choose another $u\in V$. Then
\begin{alignat}{2}
u\varphi(tv)&=u\cdot t^\theta\varphi(v)& &\qquad\text{by (ii)}\notag\\
&=u\varphi(v)\cdot [t^\theta]_F& &\qquad\text{by (F3)}\notag\\
&=t^\theta\cdot u\varphi(v)& &\qquad\text{by (F1).}\notag
\end{alignat}
Thus (iv) holds.
\end{proof}

\begin{notation}\label{abc62}
Let $x_1,\ldots,x_4$ be as in \ref{abc61}.
We replace $x_i$ by $x_i\cdot\varphi$ for $i=1$ and $3$.
After these replacements,
we have $x_i\colon V\to U_i$ for all $i\in[1,4]$,
$$[x_1(u),x_3(v)]=x_2([\hat f(\varphi(u),\varphi(v))]_K)$$
and
\begin{align*}
[x_2(u),x_4(v)]&=x_3(\varphi^{-1}([f(u,v)]_F))\\\notag
&=x_3([\hat f(\varphi(u),\varphi(v)]_K)
\end{align*}
for all $u,v\in V$ by \ref{abc14},
\begin{align*}
[x_1(u),x_4(v)]&=x_2(v\varphi(u))x_3\big(\varphi^{-1}(\varphi(u)v)\big)\\\notag
&=x_2(v\varphi(u))x_3(u\varphi(v))
\end{align*}
for all $u,v\in V$ by \ref{abc15} and
\begin{equation}\label{abc23}
x_i(u)^\rho=x_{5-i}(u)
\end{equation}
for all $u\in V$ and for all $i\in[1,4]$.
We define a product on $V$ by
\begin{equation}\label{abc53}
uv=u\varphi(v)
\end{equation}
and a symmetric map $g\colon V\times V\to[F]_K\subset V$ by
\begin{equation}\label{abc54}
g(u,v)=[\hat f(\varphi(u),\varphi(v))]_K
\end{equation}
for all $u,v\in V$. With this notation, we have
\begin{align}
[x_1(u),x_3(v)]&=x_2(g(u,v))\notag\\
[x_2(u),x_4(v)]&=x_3(g(u,v))\label{abc20}\\
[x_1(u),x_4(v)]&=x_2(vu)x_3(uv)\notag
\end{align}
for all $u,v\in V$ as well as $[U_1,U_2]=[U_2,U_3]=[U_3,U_4]=1$.
\end{notation}

\begin{notation}\label{nzz34}
From now on, we set $[t]=[t^\theta]_K$ for all $t\in K$. Thus
$$[K]:=\{[t]\mid t\in K\}=[F]_K$$
is a vector space over $K$ with scalar multiplication given by $a[t]=[a^\theta t]$
for all $a,t\in K$, and
\begin{equation}\label{nzz35}
g(u,v)=[f(u,v)^\theta]_K=[f(u,v)]
\end{equation}
for all $u,v\in V$ by \ref{abc69} and \eqref{abc54}.
\end{notation}

\begin{proposition}\label{abc51}
Let the multiplication on $V=E\oplus E\oplus[K]$
and the map $g$ be as in {\rm\ref{abc53}} and {\rm\ref{nzz35}}.
Then the following hold for all $u,v\in V$ and all $t\in K$:
\begin{itemize}
\item[(R1)] The map $x\mapsto xv$ from $V$ to itself is $K$-linear.
\item[(R2)] $v[t]=tv$.
\item[(R3)] $u\cdot tv=t^{\theta}\cdot uv$.
\item[(R4)] $[t]v=[tq(v)]$.
\item[(R5)] $uv\cdot v=q(v)^{\theta}\cdot u$.
\item[(R6)] $v\cdot uv=q(v)\cdot vu$.
\item[(R7)] $u(v+w)=uv+uw+g(vu,w)$.
\end{itemize}
\end{proposition}

\begin{proof}
Just for this proof, we will denote by $*$ both the map from $V\times\hat V$ to
$V$ and the map from $\hat V\times V$ to $V$ defined in \ref{abc80} to distinguish them
from the multiplication on $V$ defined
in \ref{abc53}. Thus, in particular, $uv=u*\varphi(v)$ for all $u,v\in V$.

Let $u,v,w\in V$ and $t\in K$. The assertion (R1) is just a special case of (F0). We have
$$v[t]=v[t^\theta]_K=v*[t]_F=tv$$
by (F1). Thus (R2) holds. The assertion (R3) follows from \ref{abc41}(iv).
To see that (R4) holds, we observe that
\begin{alignat}{2}
[t]v&=[t^\theta]_K*\varphi(v)\\
&=[t^\theta\hat q(\varphi(v)]_K& &\qquad\text{by (F6)}\notag\\
&=[t^\theta q(v)^\theta]_K& &\qquad\text{by \ref{abc41}(iii)}\notag\\
&=[tq(v)].\notag
\end{alignat}
Next we have
\begin{alignat}{2}
uv\cdot v&=u*[\hat q(\varphi(v)]_F& &\qquad\text{by (F7)}\notag\\
&=u*[q(v)^\theta]_F& &\qquad\text{by \ref{abc41}(iii)}\notag\\
%&=u[q(v)^2]_K& &\qquad\text{by \ref{abc41}(i)}\notag\\
&=u[q(v)^\theta]\notag\\
&=q(v)^\theta\cdot u& &\qquad\text{by (R2),}\notag
\end{alignat}
so (R5) holds, and
\begin{alignat}{2}
v\cdot uv&=v*\varphi\big(u*\varphi(v)\big)\notag\\
&=v*\big(\varphi(u)*v)& &\qquad\text{by \ref{abc15}(i)}\notag\\
&=q(v)\cdot vu& &\qquad\text{by (F9),}\notag
\end{alignat}
so (R6) holds. Finally, we have
\begin{alignat}{2}
u(v+w)&=u*\varphi(v+w)\notag\\
&=uv+uw+[\hat f\big(\varphi(v)*u,\varphi(w)\big)]_K& &\qquad\text{by (F11)}\notag\\
&=uv+uw+[\hat f\big(\varphi\big(v*\varphi(u)\big),\varphi(w)\big)]_K
& &\qquad\text{by \ref{abc15}(ii)}\notag\\
&=uv+uw+[\hat f\big(\varphi(vu),\varphi(w)\big)]_K\notag\\
&=uv+uw+g(vu,w)& &\qquad\text{by \eqref{abc54}.}\notag
\end{alignat}
Thus (R7) holds.
\end{proof}

\section{Polarity Algebras}\label{sec5}

In this section we introduce polarity algebras and prove a series of identities.
Some (for instance \ref{q(uv)}) we have included only because they are
compelling, not because we will apply them later on.

\begin{definition}\label{fat1}
A {\it polarity algebra} is a 6-tuple $(K,V,q,\theta,t\mapsto[t],\cdot)$,
where $K$ is a field of characteristic~$2$,
$(K,V,q)$ is an anisotropic quadratic space such that the bilinear form $f:=\partial q$
is not identically zero,
$\theta$ is a Tits endomorphism of $K$,
$t\mapsto[t]$ is a $K$-linear embedding of the $K$-vector space $[K]=[K^\theta]_K$
defined in \ref{def0} and \ref{nzz34} into the radical of $f$,
so
\begin{equation}\label{abc50a}
a[t]=[a^\theta t]
\end{equation}
for all $a,t\in K$,
and $(u,v)\mapsto u\cdot v$ is a multiplication on $V$ (which we
often denote by juxtaposition), satisfying
the conditions (R1)--(R7) in \ref{abc51} with
$$g(u,v)=[f(u,v)]$$
for all $u,v\in V$ in (R7).
\end{definition}

Throughout this section we assume that
$(K,V,q,\theta,t\mapsto[t],\cdot)$ is a polarity algebra. We let $f$ and $g$
be as in \ref{fat1} and we set
$$v^{-1}=q(v)^{-1}v$$
for all non-zero $v\in V$. Since $q$ is anisotropic, this is allowed.

\begin{remark}\label{abc70}
Let $K$ and $f$ be as in \ref{fat1}.
An anisotropic form over a finite field is either 1-dimensional
or similar to the norm of a quadratic extension.
Since $f$ is not identically zero and its radical is non-trivial,
we conclude that $K$ is infinite.
\end{remark}

\begin{proposition}\label{abc24}
The following hold for all $u,v,w\in V$ and all $t\in K$:
\begin{enumerate}[\rm(i)]
\item $g(u,uw)=0$.
\item $g(u,vw)=g(uw,v)$.
\item $g(u,v)w=g(q(w)u,v)$.
\item $tg(u,v)=g(t^\theta u,v)$.
\end{enumerate}
\end{proposition}

\begin{proof}
Assertions (i) and (ii) follow immediately from (R7) and assertion (iii)
follows immediately from (R4). We have $tg(u,v)=t[f(u,v)]=[t^\theta f(u,v)]=g(t^\theta u,v)$
for all $u,v\in V$ and all $t\in K$, so also assertion (iv) holds.
\end{proof}

\begin{proposition}\label{yht1}
$v\cdot wu=f(u,vw)u+f(u,v)uw+q(u)vw$ for all $u,v,w\in V$.
\end{proposition}

\begin{proof}
Let $t\in K$. By (R6), we have
\begin{align}
(u+tv)\cdot(w(u+tv))&=q(u+tv)(u+tv)w\notag\\
&=(q(u)+tf(u,v)+t^2q(v))(u+tv)w\notag\\
&=q(u)uw+t\big(q(u)vw+f(u,v)uw\big)+t^2\big(q(v)uw\notag\\
&\qquad\ +f(u,v)vw\big)+t^3q(v)vw.\label{yht1a}
\end{align}
By (R7) and (R3), on the other hand, we have
\begin{align}
(u+tv)\cdot(w(u+tv))&=(u+tv)\big(wu+t^{\theta}wv+g(u,tvw)\big)\notag\\
&=u(wu+t^{\theta}wv+g(u,tvw))+tv(wu+t^{\theta}wv+g(u,tvw))\notag\\
&=u(wu)+u\cdot t^{\theta}wv+ug(u,tvw)+g(wu\cdot u, t^{\theta}wv)\notag\\
&\qquad+tv\cdot wu+t^3v\cdot wv+tv\cdot g(u,tvw)+g(wu,t^{\theta}wv\cdot tv)\notag\\
&=q(u)uw+t^2u\cdot wv+ug(u,tvw)+g(wu\cdot u, t^{\theta}wv)\notag\\
&\qquad+tv\cdot wu+t^3q(v)vw+tv\cdot g(u,tvw)+g(wu,t^{\theta}wv\cdot tv).\label{yht1b}
\end{align}
Thus the sum of the expressions \eqref{yht1a} and \eqref{yht1b} is zero.
By (R2) and \ref{abc24}(iv), this sum lies in $K[t]$ and the
the coefficient of $t$ is
$$q(u)vw+f(u,v)uw+f(u,vw)u+v\cdot wu.$$
To verify this, we need only observe that
$ug(u,tvw)=u[f(u,tvw)]=f(u,tvw)u$ by (R2) and $g(wu\cdot u,wv)=q(u)g(w,wv)=0$
by (R5) and \ref{abc24}(i).
It follows from \ref{abc70} and \cite[2.26]{TW}
that the coefficient of each power of $t$ in this sum is zero.
\end{proof}

\begin{proposition}\label{yht2}
$uv\cdot w+uw\cdot v=g(v,f(v,wu)w)+f(v,w)^{\theta}u$ for all $u,v,w\in V$.
\end{proposition}

\begin{proof}
By (R5), we have
\begin{align*}
u(v+w)\cdot (v+w)&=q(v+w)^\theta u\\
&=(q(v)+f(v,w)+q(w))^\theta u\\
&=q(v)^\theta u+f(v,w)^\theta u+q(w)^\theta u.
\end{align*}
By (R5), (R6), (R7) and \ref{abc24}(iii), on the other hand, we have
\begin{align*}
u(v+w)\cdot (v+w)&=(uv+uw+g(v,wu))(v+w)\\
&=uv(v+w)+uw(v+w)+g(v,wu)(v+w)\\
&=uv\cdot v+uv\cdot w+g(v\cdot uv,w)\\
&\qquad+uw\cdot v+uw\cdot w+g(v,w\cdot uw)\\
&\qquad+g(q(v)v,wu)+g(q(w)v,wu)+g(f(v,w)v,wu)\\
&=q(v)^{\theta}u+uv\cdot w+\overbrace{g(q(v)vu,w)}^1\\
&\qquad+uw\cdot v+q(w)^{\theta}u+\overbrace{g(v,q(w)wu)}^2\\
&\qquad+\overbrace{g(q(v)v,wu)}^1+\overbrace{g(q(w)v,wu)}^2+g(v,f(v,wu)w).
\end{align*}
Note that
$$g(f(v,w)v,wu)=[f(v,w)f(v,wu)]=[f(v,f(v,wu)w)]=g(v,f(v,wu)w).$$
Therefore $f(v,w)^{\theta}u=uv\cdot w+uw\cdot v+g(v,f(v,wu)w)$.
\end{proof}

The following observation says that the quadratic form $q$ is multiplicative
``with a twist.''

\begin{proposition}\label{q(uv)}
$q(uv)=q(u)q(v)^{\theta}$ for all $u,v\in V$.
\end{proposition}

\begin{proof}
Choose $u,w\in V$ and recall that $f(V,[K])=0$.
Setting $v=[1]$ in \ref{yht1}, we obtain
$$[1]\cdot wu=f(u,[1]w)u+f(u,[1])uw+q(u)[1]w.$$
By (R3) and (R4), therefore, $[q(wu)]=q(u)[q(w)]=[q(w)q(u)^\theta]$.
\end{proof}

\begin{proposition}\label{q([t])}
$q([t])=t^{\theta}$ for all $t\in K$.
\end{proposition}

\begin{proof}
If $t\in K$, then $[q([t])]=[1][t]=t[1]=[t^\theta]$ by (R2) and (R4).
\end{proof}

\begin{proposition}\label{abc95y}
The following hold for all $u,v,w\in K$:
\begin{enumerate}[\rm(i)]
\item\label{f(uv,uw)} $f(uv,uw)=f(v,wu)^{\theta}+q(u)f(v,w)^{\theta}$
\item\label{f(uv,wv)} $f(uv, wv)=q(v)^{\theta}f(u,w)$.
\item\label{uv-1} $(uv)^{-1}=u^{-1}v^{-1}$ if $u,v\ne0$.
\end{enumerate}
\end{proposition}

\begin{proof}
Let $u,v,w\in K$. We have
\begin{align*}
q(u(v+w)) &= q(u)q(v+w)^{\theta}\\
&=q(u)(q(v)+q(w)+f(v,w))^{\theta}\\
&=q(uv)+q(uw)+q(u)f(v,w)^{\theta}
\end{align*}
by \ref{q(uv)}, whereas
\begin{alignat*}{2}
q(u(v+w))&=q(uv+uw+g(v,wu))& &\qquad\text{by (R7)}\\
&=q(uv)+q(uw)+q([f(v,wu)])+f(uv,uw)& &\qquad\text{by \ref{fat1}}\\
&=q(uv)+q(uw)+f(v,wu)^\theta+f(uv,uw)& &\qquad\text{by \ref{q([t])}}.
\end{alignat*}
Thus (i) holds.

We have
\begin{align*}
q((u+w)v)&=q(u+w)q(v)^{\theta}\\
&=(q(u)+q(w)+f(u,w))q(v)^{\theta}\\
&=q(uv)+q(wv)+q(v)^{\theta}f(u,w),
\end{align*}
by \ref{q(uv)}, whereas
$$q((u+w)v)=q(uv+wv)=q(uv)+q(wv)+f(uv,wv)$$
by (R1). Thus (ii) holds. Finally, we have
\begin{align*}
(uv)^{-1}&=q(uv)^{-1}uv\\
&=q(u)^{-1}q(v)^{-\theta}uv=q(u)^{-1}u\cdot q(v)^{-1}v=u^{-1}v^{-1}
\end{align*}
by (R3) and \ref{q(uv)}. Thus (iii) holds.
\end{proof}

\begin{proposition}\label{abc95x}
The following hold for all $u,v,w\in K$:
\begin{enumerate}[\rm(i)]
\item\label{u-1.vu} $u^{-1}\cdot vu=uv$ if $u\ne0$.
\item\label{(uv)v-1} $uv\cdot v^{-1}=u$ if $v\ne0$.
\end{enumerate}
\end{proposition}

\begin{proof}
These identities follow immediately from (R3), (R5) and (R6).
\end{proof}

\begin{remark}\label{tru3}
Let $v\in V^*$. By (R1) and \ref{abc95x}\eqref{(uv)v-1}, the map $x\mapsto xv$
is an automorphism of $V$. By \ref{q(uv)}, This map is a similitude of $q$
with similarity factor $q(v)^\theta$.
\end{remark}

\begin{proposition}\label{abc95}
The following hold for all $u,v,z,w\in V$:
\begin{enumerate}[\rm(i)]
\item\label{g(uv.w, z.v} $g(uv\cdot w,zv)=f(w,v)g(uv,z)+q(v)g(uw,z).$
\item\label{g(uv.w,uz)} $g(uv\cdot w, uz)=
f(vu,z)w+f(wu,v)z+f(wu,z)v\\
\phantom{a}\qquad+f(w,v)zu+f(w,z)vu+f(v,z)wu.$
\item\label{f(uv.zw)} $f(uv, zw)+f(uw, zv)=f(u,z)f(v,w)^{\theta}$.
\item\label{uv.vu} $uv\cdot vu=q(uv)u$.
\item\label{uv.vw} $uv\cdot vw=q(wv)u+f(u,w)q(v)^{\theta}w+f(uv,w)wv$.
\item\label{uv.vw=(u.vw).v} $uv\cdot vw=(u\cdot vw)\cdot v$.
\item\label{(vv)-1(vv)} $(vv)^{-1}\cdot vv=v$ if $v\ne0$.
\item\label{u-1u.u-1u} $u^{-1}u\cdot u^{-1}u=u$ if $u\ne0$.
\end{enumerate}
\end{proposition}

\begin{proof}
On the one hand,
\begin{alignat*}{2}
w\cdot(u+z)v&=f(w(u+z),v)v+f(w,v)v(u+z)+q(v)w(u+z)& &\qquad\text{by \ref{yht1}}\\
&=f(wu+wz,v)v+f(w,v)\big(vu+vz+g(uv,z)\big)\\
&\qquad\qquad+q(v)\big(wu+wz+g(uw,z)& &\qquad\text{by (R7)},
\end{alignat*}
and on the other,
\begin{alignat*}{2}
w\cdot(u+z)v&=w\cdot(uv+zv)& &\qquad\text{by (R1)}\\
&=w\cdot uv+w\cdot zv+g(uv\cdot w,zv)\\
&=f(v,wu)v+f(v,w)vw+q(v)wu\\
&\qquad+f(v,wz)v+f(v,w)vz+q(v)wz\\
&\qquad\quad+g(uv\cdot w,zv)& &\qquad\text{by \ref{yht1}}.
\end{alignat*}
Thus (i) holds.

Using \ref{yht1}, (R2) and (R7), we have
\begin{align*}
w\cdot u(v+z)&=w\cdot\big(uv+uz+g(vu,z)\big)\\
&=w(uv+uz)+w[f(vu,z)]\\
&=w(uv+uz)+f(vu,z)w\\
&=w\cdot uv+w\cdot uz+g(uv\cdot w,uz)+f(vu,z)w,
\end{align*}
whereas
\begin{align*}
w\cdot u(v+z)&=f(wu,v+z)(v+z)+f(w,v+z)(vu+zu)+q(v+z)wu\\
&=\big(f(wu,v)v+f(w,v)vu+q(v)wu\big)\\
&\qquad+\big(f(wu,z)z+f(w,z)zu+q(z)wu\big)\\
&\quad\qquad+f(wu,v)z+f(wu,z)v+f(w,v)zu+f(w,z)vu+f(v,z)wu\\
&=w\cdot uv\\
&\qquad+w\cdot uz\\
&\quad\qquad+f(wu,v)z+f(wu,z)v+f(w,v)zu+f(w,z)vu+f(v,z)wu.
\end{align*}
by several applications of \ref{yht1}. Thus (ii) holds.

Using \eqref{abc24}(ii) and \ref{yht2}, we obtain
$$uv\cdot w+uw\cdot v=g(v,f(v,wu)w)+f(v,w)^{\theta}u$$
and
\begin{align*}
f(uv,zw)+f(uw,zv)&=f(uv\cdot w,z)+f(uw\cdot v,z)\\
&=f(uv\cdot w+uw\cdot v,z)\\
&=f(f(v,w)^{\theta}u,z)=f(u,z)f(v,w)^{\theta}.
\end{align*}
Thus (iii) holds.

We have
$$uv\cdot vu=f(u,uv\cdot v)u+f(u,uv)uv+q(u)uv\cdot v$$
by \ref{yht1} and
$$uv\cdot v=uv\cdot q(v)v^{-1}=q(v)^{\theta}u$$
by \ref{abc95x},\eqref{(uv)v-1} and (R2).
Hence
$$uv\cdot vu=q(u)q(v)^{\theta}v=q(uv)u$$
by \eqref{abc24}(i), \ref{q(uv)} and \ref{abc95x}\eqref{(uv)v-1}. Thus (iv) holds.

By \ref{yht1}, \ref{q(uv)} and (R5), we have
\begin{align*}
uv\cdot vw&=f(uv\cdot v,w)w+f(uv,w)wv+q(w)uv\cdot v\\
&=f(u,w)q(v)^{\theta}w+f(uv,w)wv+q(wv)u.
\end{align*}
Thus (v) holds.

Using \ref{yht2}, we have
\begin{align*}
\overbrace{u}^P\overbrace{v}^U\cdot\overbrace{vw}^V+(u\cdot vw)\cdot v
&=g(U,f(U,VP)V)+f(U,V)^{\theta}P\\
&=g(v,f(v,vw\cdot u)vw)+f(v,vw)^{\theta}u=0.
\end{align*}
Thus (vi) holds.

By \eqref{uv.vu}, we have $vv\cdot vv=q(vv)\cdot v$.  Hence
$$v=q(vv)^{-1}vv\cdot vv=(vv)^{-1}(vv)$$
and thus (vii) holds.

Using (R3), \ref{q(uv)} and \eqref{uv.vu}, finally, we have
$$u^{-1}u\cdot u^{-1}u=q(u)^{-1}q(u)^{-\theta}uu\cdot uu=q(uu)^{-1}q(uu)u=u.$$
Thus (viii) holds.
\end{proof}

\begin{proposition}\label{tru46}
For each nonzero $a\in V$, there exists $b\in V$
such that $a=b^{-1}b$.
\end{proposition}

\begin{proof}
This holds by \ref{abc95}\eqref{(vv)-1(vv)}.
\end{proof}

\section{Tits Endomorphisms}\label{sec6}

We begin this section by proving some elementary properties of arbitrary Tits endomorphisms
in \ref{abc90} and \ref{abc71}.

\begin{definition}\label{abc88}
Let $(K,\theta)$ be an octagonal set as defined in \ref{abc0}.
We will call an element of $K$ a {\it Tits trace} (with respect to $\theta$) if it is of the form
$x^\theta+x$ for some $x\in K$.
\end{definition}

\begin{proposition}\label{abc90}
Let $(K,\theta)$ be an octagonal set. Then the following hold:
\begin{enumerate}[\rm(i)]
\item The map $x\mapsto x^\theta+x$ from $K$ to itself is additive.
\item If $u^\theta=u$ for some $u\in K$, then $u=0$ or $1$.
\item If $u^\theta+u=v^\theta+v$ for $u,v\in K$, then either $u=v$ or $u=v+1$.
\item If $z^\theta$ is a Tits trace, then so is $z$.
\item $u^2+u$ is a Tits trace for every $u\in K$.
\item $1$ is not a Tits trace.
\end{enumerate}
\end{proposition}

\begin{proof}
Since $\theta$ is additive, (i) holds. Suppose that $u^\theta=u$ for some $u\in K$,
then $u^2=u$ and hence $u=0$ or $1$.
Thus (ii) holds, and (iii) follows from (i) and (ii). Suppose that
$z^\theta=u^\theta+u$ for some $u\in K$. Let $v=u+z$.
Then $v^\theta=u$ and hence $z=v^\theta+v$. Thus (iv) holds. Applying
the map $x\mapsto x^\theta+x$ twice to an element $u\in K$ yields $u^2+u$.
Thus (v) holds. Suppose, finally, that $u^\theta+u=1$ for some $u\in K$.
Then $u^2=(u^\theta)^\theta=(u+1)^\theta=u^\theta+1=u$ and hence
$u=0$ or $1$. Thus (vi) holds.
\end{proof}

\begin{proposition}\label{abc71}
Let $\theta$ be a Tits endomorphism of a field $K$,
let $\delta\in K$, let $L$ be the splitting field over $K$ of the polynomial $x^2+x+\delta$
and let $\chi$ be the non-trivial element of ${\rm Gal}(L/K)$.
Then the following hold:
\begin{enumerate}[\rm(i)]
\item $\theta$ extends to a Tits endomorphism of $L$ if and only if $\delta$
is a Tits trace.
\item If $\theta$ extends to a Tits endomorphism of $L$, then there are
exactly two extensions $\theta_1$ and $\theta_2$, both commute with $\chi$
and $\chi=\theta_2^{-1}\cdot\theta_1$.
\end{enumerate}
\end{proposition}

\begin{proof}
Let $\gamma$ be a root of $x^2+x+\delta$ in $L$.
Suppose that $\delta=\lambda^\theta+\lambda$ for some $\lambda\in K$. We extend $\theta$ to
an endomorphism $\theta_1$ of $L$ by setting $\gamma^{\theta_1}=\gamma+\lambda$ and
$$(a+b\gamma)^{\theta_1}=a^\theta+b^\theta\gamma^{\theta_1}$$
for all $a,b\in K$. We have
\begin{align*}
(\gamma+\lambda)^{\theta_1}&=\gamma^{\theta_1}+\lambda^\theta\\
&=\gamma+\lambda+\lambda^\theta\\
&=\gamma+\delta=\gamma^2.
\end{align*}
Hence $\theta_1^2={\rm Frob}_L$. Thus $\theta_1$ is a Tits endomorphism of $L$.

Suppose, conversely, that $\theta$ extends to a Tits endomorphism $\theta_1$ of $L$.
Then $\gamma^{\theta_1}=a+b\gamma$ for some $a,b\in K$.
Therefore
\begin{align*}
\gamma+\delta=\gamma^2&=a^\theta+b^\theta\gamma^{\theta_1}\\
&=a^\theta+b^\theta(a+b\gamma)\\
&=a^\theta+ab^\theta+b^{\theta+1}\gamma,
\end{align*}
so $\delta=a^\theta+ab^\theta$ and $b^{\theta+1}=1$.
Hence $b=b^{\theta^2-1}=(b^{\theta+1})^{\theta-1}=1$ and therefore
\begin{equation}\label{abc91}
\delta=a^\theta+a,
\end{equation}
so $\delta$ is a Tits trace and hence (i) holds. Furthermore
$$\gamma^{\chi\theta_1}=\gamma^{\theta_1}+1=a+1+\gamma=a+\gamma^\chi
=\gamma^{\theta_1\chi},$$
so $\theta_1$ commutes with $\chi$ and thus the product $\theta_2:=\chi\theta_1$ is a
second Tits endomorphism of $L$ extending $\theta$. By \ref{abc90}(iii) and \eqref{abc91},
there are no others.
\end{proof}

We now go back to assuming that $S$, $K$, $V$, $q$, $f$  and $\Xi={\mathcal Q}(S)$ are as in
\ref{abc83} and \ref{abc61}, that $\rho$, $\varphi$ and $\theta$
are as is as in \ref{abc84}, \ref{abc97} and \ref{abc40} and that the product
$uv$ on $V$ is as in \eqref{abc53}. Our goal for the rest of this section is
to prove \ref{abc92}.

\begin{proposition}\label{abc92}
Let $d$ be a non-zero element of $V$. Then there exists an element $e\in V$ such that
$f(d,e)=1$, $f(d,ed)=0$,
\begin{equation}\label{abc92d}
f(dd,ed)=q(d)^\theta
\end{equation}
as well as
\begin{equation}\label{abc92p}
q(d)^\theta\cdot de=f(de,ed)\cdot dd+q(d)\cdot ed
\end{equation}
and $q(d)q(e)=f(de,ed)+f(de,ed)^\theta$. In particular, $q(d)q(e)$ is a Tits trace.
\end{proposition}

\begin{proof}
By \cite[Thm. 2.1(i)]{tom1}, there exists $e\in V$ such that
\begin{equation}\label{abc92z}
f(d,e)=1\text{ and }f(d,ed)=0.
\end{equation}
By \ref{abc24}(ii), it follows that
\begin{equation}\label{abc92a}
f(dd,e)=0.
\end{equation}
By \ref{abc24}(i), we have
\begin{equation}\label{abc92b}
f(d,dd)=0
\end{equation}
and by \ref{abc24}(ii), we have
\begin{equation}\label{abc92g}
f(d,dd\cdot d)=f(dd,dd)=0.
\end{equation}
By \ref{abc95y}\eqref{f(uv,uw)} and \eqref{abc92z}, we have
\begin{equation}\label{abc92c}
f(dd,de)=f(d,ed)^\theta+q(d)f(d,e)^\theta=q(d)
\end{equation}
and by \ref{abc95}\eqref{f(uv,wv)} and \eqref{abc92z}, we have
$$f(dd,ed)=q(d)^\theta\cdot f(d,e)=q(d)^\theta.$$
Thus \eqref{abc92d} holds.

Choose $\lambda\in K$. The image of the map $g$ is the radical of $f$.
By (R1), (R3) and (R7), therefore, we have
\begin{align*}
f\big((e+\lambda\cdot dd)(e+\lambda\cdot dd),d\big)
&=f(ee,d)+\lambda f(dd\cdot e,d)+\lambda^\theta f(e\cdot dd,d)\\
&\qquad\qquad+\lambda^{\theta+1}f(dd\cdot dd,d).
\end{align*}
We observe that $f(dd\cdot e,d)=f(dd,de)=q(d)$ by \ref{abc24}(ii) and \eqref{abc92c} as well as
\begin{alignat}{2}
f(e\cdot dd,d)&=f(e,d\cdot dd)& &\qquad\text{by \ref{abc24}(ii)}\notag\\
&=q(d)f(e,dd)& &\qquad\text{by (R6)}\notag\\
&=0& &\qquad\text{by \eqref{abc92a},}\notag
\end{alignat}
and
$$f(dd\cdot dd,d)=f(q(dd)\cdot d,d)=q(dd)f(d,d)=0$$
by \ref{abc95}\eqref{uv.vu}. It follows that
$$f\big((e+\lambda\cdot dd)(e+\lambda\cdot dd),d\big)=f(ee,d)+\lambda q(d).$$
Furthermore,
$$f(d,e+\lambda\cdot dd)=f(d,e)=1$$
by \eqref{abc92z} and \eqref{abc92b}
and $f(d,(e+\lambda\cdot dd)d)=f(d,ed)+\lambda f(d,dd\cdot d)$ by (R1),
so, in fact,
$$f\big(d,(e+\lambda\cdot dd)d\big)=0$$
by \eqref{abc92z} and \eqref{abc92g}.
Thus if we replace $e$ by $e+\lambda\cdot dd$
and $\lambda$ by $f(ee,d)/q(d)$, we can assume that $f(ee,d)=0$
while \eqref{abc92z} and therefore also \eqref{abc92d} remain valid. Hence also
\begin{equation}\label{abc92t}
f(e,de)=0
\end{equation}
by \ref{abc24}(ii). By \ref{abc95y}\eqref{f(uv,uw)}, it follows that
\begin{equation}\label{abc92s}
f(ed,ee)=f(e,de)^\theta+q(e)f(d,e)=q(e).
\end{equation}
By \ref{abc24}(i) and \eqref{abc92t}, we have
$de\in\langle d,e\rangle^\perp$ (where $\langle d,e\rangle^\perp$ denotes
the subspace orthogonal to $\langle d,e\rangle$
with respect to the bilinear form $f$). Setting $\xi=d$ in \cite[Thm.~2.1]{tom1}, we obtain
$\langle d,e\rangle^\perp=\langle dd,ed\rangle+[K]$.
Hence
\begin{equation}\label{abc92i}
de=\kappa\cdot dd+\mu\cdot ed+[t]
\end{equation}
for some $\kappa,\mu,t\in K$.
Therefore
$$f(dd,de)=f(dd,\kappa\cdot dd+\mu\cdot ed)=\mu\cdot f(dd,ed)=\mu\cdot q(d)^\theta$$
and
$$f(de,ed)=f(\kappa\cdot dd+\mu\cdot ed,ed)=\kappa\cdot f(dd,ed)=\kappa\cdot q(d)^\theta$$
by \eqref{abc92d}, so
\begin{equation}\label{abc92j}
\mu=q(d)^{1-\theta}
\end{equation}
by \eqref{abc92c} and
\begin{equation}\label{abc92k}
\kappa=f(de,ed)/q(d)^\theta.
\end{equation}
Substituting \eqref{abc92j} but not \eqref{abc92k} in \eqref{abc92i} and
applying $q$, we obtain
\begin{align*}
q(de)=q(d)q(e)^\theta&=\kappa^2q(d)^{\theta+1}+q(d)^{2-2\theta}\cdot q(e)q(d)^\theta\\
&\qquad+t^\theta+\kappa q(d)^{1-\theta}f(dd,ed)
\end{align*}
by \ref{q(uv)} and \ref{q([t])}.
Multiplying by $q(d)^{\theta-1}$ and applying \eqref{abc92d},
we then obtain
$$q(d)^\theta q(e)^\theta=\kappa^2q(d)^{2\theta}+\kappa q(d)^\theta+q(d)q(e)+t^\theta q(d)^{\theta-1}.$$
Thus $\kappa q(d)^\theta$ is a root of the polynomial
\begin{equation}\label{abc92m}
p(x):=x^2+x+q(d)q(e)+q(d)^\theta q(e)^\theta+t^\theta q(d)^{\theta-1}.
\end{equation}
By \ref{abc92k}, $\kappa q(d)^\theta=f(de,ed)$ and
by \ref{abc95}\eqref{f(uv.zw)} and \eqref{abc92z}, we have
\begin{equation}\label{abc92n}
f(de,ed)+f(dd,ee)=f(d,e)f(e,d)^\theta=1.
\end{equation}
We conclude that $f(de,ed)$ and $f(dd,ee)$ are the two roots of the polynomial $p$.

Next we observe that
\begin{alignat}{2}
q(e)^\theta&=q(e)^\theta f(d,e)& &\qquad\text{by \eqref{abc92z}}\notag\\
&=f(de,ee)& &\qquad\text{by \ref{abc95y}\eqref{f(uv,wv)}}\notag\\
&=\kappa\cdot f(dd,ee)+\mu\cdot q(e)& &\qquad\text{by \eqref{abc92s} and \eqref{abc92i}}\notag\\
&=\big(f(de,ed)\cdot f(dd,ee)+ q(d)q(e)\big)/q(d)^\theta& &\qquad\text{by \eqref{abc92j} and
\eqref{abc92k},}\notag
\end{alignat}
so
$$f(de,ed)\cdot f(dd,ee)=q(d)q(e)+q(d)^\theta q(e)^\theta.$$
Thus by \eqref{abc92n}, $f(de,ed)$ is a root of the polynomial.
$$x^2+x+q(d)q(e)+q(d)^\theta q(e)^\theta.$$
Since $f(de,ed)$ is also a root of the polynomial $p(x)$ defined
in \eqref{abc92m}, we conclude that $t=0$. Hence
\begin{alignat}{2}
q(d)^\theta\cdot de&=q(d)^\theta\big(\kappa\cdot dd+\mu\cdot ed\big)
& &\qquad\text{by \eqref{abc92i}}\notag\\
&=f(de,ed)\cdot dd+q(d)\cdot ed& &\qquad\text{by \eqref{abc92j} and \eqref{abc92k}}\notag
\end{alignat}
Thus \eqref{abc92p} holds.

Let $s=f(de,ed)$.
We multiply \eqref{abc92p} on the left by $d$. Applying (R3) and (R7), we obtain
\begin{equation}\label{eq:dde}
q(d)^2 d \cdot de=s^\theta d \cdot dd+q(d)^\theta d \cdot ed
+g\bigl( s \cdot(dd \cdot d), q(d) \cdot ed \bigr).
\end{equation}
By \ref{yht1}, \eqref{abc92z} and \eqref{abc92a}, we have
\[ d \cdot de = f(e, dd) e + f(e, d) ed + q(e) dd = ed + q(e) dd.\]
By (R6), we have $d \cdot dd = q(d) dd$ and $d \cdot ed = q(d) de$. We also have
\[ g(dd \cdot d, ed) = g(q(d)^\theta d, ed) = g(q(d)^\theta dd, e) = 0 \]
by (R5) and \eqref{abc92a}. Hence we can rewrite \eqref{eq:dde} as
\[ q(d)^2 \cdot ed + q(d)^2 q(e) \cdot dd = s^\theta q(d) \cdot dd + q(d)^{\theta+1} \cdot de.\]
Dividing by $q(d)$ and rearranging terms, we obtain
\[ q(d)^\theta \cdot de = \bigl( s^\theta + q(d) q(e) \bigr) \cdot dd + q(d) \cdot ed . \]
Comparing this equation with \eqref{abc92p}, we conclude that
\[ s^\theta + q(d) q(e) = s. \]
In other words, $q(d)q(e)=f(de,ed)+f(de,ed)^\theta$.
\end{proof}

\section{Polar Triples}\label{sec6x}

We continue to assume that
$S$, $K$, $F$, $V$, $q$, $f$, $\Xi={\mathcal Q}(S)$, $\Sigma$, $c$ and $x_1,\ldots,x_4$ are as in
\ref{abc83} and \ref{abc61}, that $\rho$, $\varphi$ and $\theta$
are as is as in \ref{abc84}, \ref{abc97} and \ref{abc40}
and that the product
$uv$ on $V$ is as in \eqref{abc53}. Thus $F=K^\theta$, $\theta$ is a Tits endomorphism
of $K$ by \ref{abc41}(i) and
$$\Xi={\mathcal Q}_{\mathcal F}(K,V,q)=B_2^{\mathcal F}(K,V,q)$$
by \ref{abc85x}. The main results of this section are \ref{tru68} and \ref{onk111}.

\begin{proposition}\label{tru1}
Let $d$ and $e$ as in {\rm\ref{abc92}} and let $\xi=\varphi(d)$, so that
$f(d,e)=1$ and $f(d,e\xi)=0$. Let
$$S_0:=(L/K,K^\theta,\alpha_0,\beta_0)$$
be the standard decomposition of $q$ obtained by applying
{\rm\ref{abc77}} to the triple $d$, $e$ and $\xi$. Then $\alpha_0=\beta_0^{-\theta}$
and $\theta$ has an extension to a Tits endomorphism of $L$.
\end{proposition}

\begin{proof}
By \ref{tru2}, the field $L$ is the splitting field of the polynomial
$$x^2+x+q(d)q(e)$$
over $K$. Hence \ref{abc71}(i) and the last assertion in \ref{abc92}, $\theta$ has an
extension to a Tits endomorphism of $L$. By \ref{abc77},
we have $\alpha_0=f(d\xi,e\xi)$ and $\beta_0=q(d)^{-1}$. Thus
by \eqref{abc92d}, $\alpha_0=q(d)^\theta=\beta_0^{-\theta}$.
\end{proof}

\begin{hypothesis}\label{tru5x}
We assume from now on that $d$ and $e$ are as in \ref{abc92} and
that $\xi=\varphi(d)$ and that the standard decomposition $S$ in \ref{abc83}
is the standard decomposition $S_0$ in \ref{tru1}. Thus $E/K$ is now
the extension called $L/K$ and $\alpha$ and $\beta$ are now the constants called
$\alpha_0$ and $\beta_0$ in \ref{tru1},
the Tits endomorphism $\theta$ has an extension to the
field $E$ and $\beta^\theta=\alpha^{-1}$. The group $U_+$ is unchanged
by this assumption, but we assume that $V$, $\hat V$ and the isomorphisms $x_1,\ldots,x_4$
are as in \ref{abc61} with respect to the new $S$.
\end{hypothesis}

\begin{notation}\label{tru5}
By \ref{abc71}(ii), $\theta$ has exactly
two extensions to $E$. We denote these extensions by $\theta_1$ and $\theta_2$.
Both commute with the non-trivial element $\chi$ of ${\rm Gal}(E/K)$
and $\theta_2=\chi\theta_1$. Let $\bar x=x^\chi$ for all $x\in E$.
\end{notation}

\begin{remark}\label{tru6}
Let $\theta_1$, $\theta_2$ and $\chi$
be as in \ref{tru5}, let $\gamma\in E$ be a root of
$$x^2+x+q(d)q(e)$$
and let $i=1$ or $2$.
Since $\chi$ commutes with $\theta_i$, we have $\gamma^{\theta_i}\not\in K$.
Thus $E=K(\gamma^{\theta_i})$ and hence
$$D=E^2F=F(\gamma^2)=K(\gamma^{\theta_i})^{\theta_i}=E^{\theta_i}.$$
\end{remark}

\begin{proposition}\label{tru4}
For either $i=1$ or $i=2$, the map $\varphi$ from $V=E\oplus E\oplus[F]_K$ to
$\hat V=D\oplus D\oplus[K]_F$ is given by
$$\varphi(a,b,s)=(a^{\theta_i},\beta^{-2}b^{\theta_i},s^{\theta^{-1}})$$
for all $(a,b,s)\in V$.
\end{proposition}

\begin{proof}
Let $\eta=\varphi(e)$. Then by (R5) and \eqref{abc92p}, we have
$$d\eta\cdot\xi\in\langle d,e\rangle.$$
Applying $\varphi$, we obtain
\begin{equation}\label{tru4a}
\xi e\cdot d \in \langle\xi,\eta\rangle
\end{equation}
by \ref{abc15}(i). By \eqref{abc69} and \ref{abc24}(ii), we have
\begin{align*}
\hat f(\xi e\cdot d,\xi)&=\hat f(\varphi(d\eta\cdot\xi),\varphi(d))\\
&=f(d\eta\cdot\xi,d)^\theta=f(d\eta,d\xi)^\theta.
\end{align*}
By \eqref{abc92c}, it follows that $\hat f(\xi e\cdot d,\xi)\ne0$.
Thus $\xi e\cdot d$ and $\xi$ are linearly independent.
By \eqref{tru4a}, therefore,
\begin{equation}\label{tru4b}
\eta\in\langle\xi,\xi e\cdot d\rangle.
\end{equation}
By \ref{abc15}(i) again, we have $\varphi(e\xi)=\eta d$ and $\varphi(d\xi)=\xi d$.
Hence by (R1), (R5) and \eqref{tru4b},
$$\varphi(e\xi)\in\langle\xi d,\xi e\rangle.$$
By \ref{abc41}(ii), $\varphi$ is an isomorphism of vector spaces. Thus
\begin{equation}\label{tru4c}
\varphi(\langle d,e\rangle)=\langle \xi,\xi e\cdot d\rangle
\end{equation}
and
\begin{equation}\label{tru4d}
\varphi(\langle d\xi,e\xi\rangle)=\langle\xi d,\xi e\rangle.
\end{equation}

Let $\gamma$, $\chi$, $\theta_1$ and $\theta_2$ be as in \ref{tru6} and
let $N(x)=x\cdot x^\chi$ and $T(x)=x+x^\chi$ for all $x\in E$.
We set $\omega=\beta\gamma$. Thus
$$\hat q(\xi)=\hat q(\varphi(d))=q(d)^\theta=\beta^{-\theta}=\alpha$$
and $\omega$ is a root in $E$ of $q(d)x^2+x+q(e)$.
By \ref{abc77}, we can assume that $d=(1,0,0)$, $e=(\omega,0,0)$,
$d\xi=(0,1,0)$ and $e\xi=(0,\omega,0)$ in $V$ and
$\xi=(1,0,0)$, $\xi e\cdot d^{-1}=(\omega^2,0,0)$, $\xi d^{-1}=(0,1,0)$ and
$\beta^2\xi e=(0,\omega^2,0)$ in $\hat V$.
Thus $\varphi(d)=\xi=(1,0,0)$.

By \ref{abc40}, \eqref{tru4c} and \eqref{tru4d},
there exist maps $\varphi_1$ and $\varphi_2$ from $E$ to $D$ such that
\begin{equation}\label{tru4f}
\varphi(a,b,s)=(\varphi_1(a),\varphi_2(b),s^{\theta^{-1}})
\end{equation}
for all $(a,b,s)\in V$. By \ref{abc41}(ii), $\varphi_1$ and $\varphi_2$ are $\theta$-linear and
since $\varphi(d)=\xi$, we have $\varphi_1(1)=1$. Furthermore,
$\hat f(\xi,\eta)=\hat f(\varphi(d),\varphi(e))=f(d,e)^\theta=1$ and
$\hat q(\eta)=\hat q(\varphi(e))=q(e)^\theta$. By \eqref{def2x}, therefore,
$T(\varphi_1(\omega))=\alpha^{-1}$ and $N(\varphi_1(\omega))=\alpha^{-1}q(e)^\theta$.
Thus $\varphi_1(\omega)$ is a root of $q(d)^\theta x^2+x+q(e)^\theta$. Hence we can
choose $i\in\{1,2\}$ such that
$$\varphi_1(\omega)=\omega^{\theta_i}.$$
Since $\varphi_1(1)=1$ and $\varphi_1$ is $\theta$-linear, it follows that
\begin{equation}\label{tru4e}
\varphi_1=\theta_i.
\end{equation}

Let $v\in E$. By \cite[16.7]{TW}, we have
$$[x_1(1,0,0),x_4(0,v,0)]_2=x_2(\alpha v,0,0)$$
and
$$[x_4(1,0,0),x_1(0,\varphi_2(v),0)]_3=x_3(\beta^{-2}\varphi_2(v),0,0).$$
Applying $\rho$ to the first of these equations, we obtain
$$[x_4(1,0,0),x_1(0,\varphi_2(v),0)]_3=x_3(\varphi_1(\alpha v),0,0).$$
Hence
$$\varphi_2(x)=\varphi_1(\alpha x)=\beta^{-2}\varphi_1(x)=\beta^{-2}x^{\theta_i}.$$
Thus by \eqref{tru4f} and \eqref{tru4e}, we have
$$\varphi(a,b,s)=(a^{\theta_i},\beta^{-2}b^{\theta_i},s^{\theta^{-1}})$$
for all $(a,b,s)\in V$.
\end{proof}

We summarize our results as follows:

\begin{theorem}\label{tru68}
Let $\Xi=B_2^{\mathcal F}(K,V,q)$ for some quadratic space $(K,V,q)$
of type $F_4$ and suppose that $\rho$ is a polarity of $\Xi$. Then there exists a
standard decomposition
$$S=(E/K,F,\alpha,\beta)$$
of $q$ and a Tits endomorphism $\theta$ of $E$ such that the following hold:
\begin{enumerate}[\rm(i)]
\item $F=K^\theta$.
\item $\alpha=\beta^{-\theta}$.
\item $\Xi$ can be identified with ${\mathcal Q}(S)$ in such a way that
$\rho$ stabilizes $\Sigma$ and $c$
and $x_i(u,v,s)^\rho=x_{5-i}(u^\theta,\beta^{-2}v^\theta,s^{\theta^{-1}})$
for $i=2$ and $4$ and all $(u,v,s)\in V$, where
$\Sigma$, $c$, $V$ and $x_i$ for $i\in[1,4]$ are as in {\rm\ref{abc61}} applied to $S$.
\end{enumerate}
\end{theorem}

\begin{proof}
By \ref{def1}, we can choose a standard decomposition $S$ of $q$ and
by \ref{abc85x}, $\Xi\cong{\mathcal Q}(S)$.
Let $\Sigma$ and $c$ be as in \ref{abc61} applied to $S$.
By \ref{abc84x}, there exists an isomorphism $\xi_S$ from $\Xi$ with ${\mathcal Q}(S)$
such that $\xi_S^{-1}\rho\xi_S$ stabilizes $c$ and $\Sigma$.
By \ref{tru1}, the standard decomposition $S$ and a Tits endomorphism $\theta$ of $E$
can be chosen so that (i) and (ii) hold.
If we identify $\Xi$ with ${\mathcal Q}(S)$ via $\xi_S$
and replace $\rho$ by $\xi^{-1}\rho\xi$ for his choice of $S$, then (iii) holds by \ref{tru4}.
\end{proof}

\begin{remark}\label{onk96}
In \ref{tru7} we give an example of $\Xi$, $\rho$,
$S=(E/K,F,\alpha,\beta)$ and $\theta$ satisfying the conditions~(i) and~(ii) in \ref{tru68}
and a splitting field $\tilde E$ of $q_S$ (as defined in \ref{def1})
such that the restriction of $\theta$ to $K$
does {\it not} have an extension to a Tits endomorphism of $\tilde E$.
See also \ref{tru50}. In \ref{tru51}, we give
an example of a Moufang quadrangle of type $F_4$ that has non-type-preserving
automorphisms but no polarity.
\end{remark}

\begin{notation}\label{tru90}
Let $\Xi$, $\rho$, $S=(E/K,F,\alpha,\beta)$, $\theta$, $\Sigma$, $c$ and the identification of $\Xi$
with ${\mathcal Q}(S)$ be as in \ref{tru68}, let
$$(U_+,U_1,\ldots,U_4)$$
and $x_1,\ldots,x_4$ be as in \ref{abc61} applied to $S$,
let $x\mapsto\bar x$ be as in \ref{tru5}
and let $\iota$ denote the map $(a,b,s)\mapsto(a,b,s^{\theta^{-1}})$ from $V_S=E\oplus E\oplus[F]_K$
to $E\oplus E\oplus[K]$, where $[K]$ is as in \ref{nzz34}. We identify $V=V_S$ with its image
under $\iota$ and we reparametrize $U_+$ by replacing $x_i$ by $\varphi\cdot x_i$ for $i=1$ and $3$
as in \ref{abc62} and then replacing $x_i$ by $\iota\cdot x_i$ for all $i\in[1,4]$.
Thus $x_i$ is an isomorphism from the additive group of $V=E\oplus E\oplus[K]$ to $U_i$
for all $i\in[1,4]$ and the following identities hold:
\begin{align*}
[x_1(a,b,r),x_3(a',b',r')]
&=x_2\bigl(0,\ 0,\ \beta^{-1}\bigl(a\bar{a}'+\bar{a}a'+\alpha(b\bar{b}'+\bar{b}b')\bigr)\bigr),\\
[x_2(u,v,s), x_4(u',v',s')]
&=x_3\bigl(0,\ 0,\ \beta^{-1}\bigl(u\bar{u}'+\bar{u}u'+\alpha(v\bar{v}'+\bar{v}v')\bigr)\bigr),\\
[x_1(a,b,r),x_4(u,v,s)]
&=x_2\bigl(ru+\alpha(\bar{a}^\theta v+\beta^{-1} b^\theta \bar{v}),
\ rv+a^\theta u+\beta^{-1}b^\theta \bar{u}, \\
&\hspace*{10ex} r^\theta s + \beta^{-1} s (a\bar{a} + \alpha b\bar{b})\\
&\hspace*{15ex}+\alpha\beta^{-1}\bigl(u^\theta a\bar{b}
+\bar{u}^\theta\bar{a}b+\beta^{-1}(v^\theta \bar{a}\bar{b}+\bar{v}^\theta ab)\bigr)\bigr) \\
& \hspace*{2.6ex}\cdot x_3\bigl(sa+\alpha(\bar{u}^\theta b+\beta^{-1}v^\theta\bar{b}),
\ sb+u^\theta a+\beta^{-1}v^\theta \bar{a},\\
&\hspace*{10ex}s^\theta r+\beta^{-1}r(u\bar{u}+\alpha v\bar{v})\\
&\hspace*{15ex}+\alpha\beta^{-1}\bigl(a^\theta u\bar{v}+\bar{a}^\theta\bar{u}v
+\beta^{-1}(b^\theta\bar{u}\bar{v}+\bar{b}^\theta uv)\bigr)\bigr)
\end{align*}
for all $(a,b,r),(a',b',r'),(u,v,s),(u',v',s')\in V$,
$$[U_1,U_2]=[U_2,U_3]=[U_3,U_4]=1$$
and
\begin{equation}\label{tru90x}
x_i(v)^\rho=x_{5-i}(v)
\end{equation}
for all $i\in[1,4]$ and all $v\in V$.
\end{notation}

\begin{proposition}\label{tru99}
Let $\Xi$, $S=(E/K,F,\alpha,\beta)$, $\theta$, $\rho$ and the identification of $\Xi$
with ${\mathcal Q}(S)$ be as in {\rm\ref{tru68}},
let $\cdot$ be the multiplication on $V$ defined in {\rm\eqref{abc53}}, let $x\mapsto\bar x$
be as in {\rm\ref{tru5}} and let $V$
be identified with $E\oplus E\oplus[K]$ as in {\rm\ref{tru90}}. Then
\begin{align*}
(a,b,r)\cdot(u,v,s)&=
\bigl(sa+\alpha(\bar{u}^\theta b+\beta^{-1}v^\theta\bar{b}),
\ sb+u^\theta a+\beta^{-1}v^\theta \bar{a},\\
&\hspace*{10ex}s^\theta r+\beta^{-1}r(u\bar{u}+\alpha v\bar{v})\\
&\hspace*{15ex}+\alpha\beta^{-1}\bigl(a^\theta u\bar{v}+\bar{a}^\theta\bar{u}v
+\beta^{-1}(b^\theta\bar{u}\bar{v}+\bar{b}^\theta uv)\bigr)
\end{align*}
for all $(a,b,r),(u,v,s)\in V$.
\end{proposition}

\begin{proof}
This holds by \eqref{abc20} and \ref{tru90}.
\end{proof}

\begin{notation}\label{nzz31}
Let $\Xi$, $(K,V,q)$, $\theta$ and $S=(E/K,F,\alpha,\beta)$ and the identification
of $\Xi$ with ${\mathcal Q}(S)$ be as in \ref{tru68}, so $F=K^\theta$
and $\alpha=\beta^{-\theta}$.
We identify $V$ with $E\oplus E\oplus[K]$ as in \ref{tru90}, so that
\begin{equation}\label{nzz31x}
q(u,v,t)=\beta^{-1}(N(u)+\alpha N(v))+t^\theta
\end{equation}
for all $(u,v,t)\in V$. Let $[t]=(0,0,t)$ for all $t\in K$ and let $\cdot$
be the multiplication on $V$ given by the formula in \ref{tru99}.
By \ref{abc51}, $(K,V,q,\theta,t\mapsto[t],\cdot)$
is a polarity algebra. We denote this polarity algebra by $A=A(E/K,\theta,\beta)$.
\end{notation}

\begin{definition}\label{nzz39}
A {\it polar triple} is a triple $(E/K,\theta,\beta)$, where $E/K$ is a
separable quadratic extension in characteristic~$2$, $\theta$ is a
Tits endomorphism of $E$ such that $F:=K^\theta\subset K$ and $\beta$ is an
element of $K$ such that the quadratic form on $E\oplus E\oplus[K]$ given
by \eqref{nzz31x} is anisotropic, where $[K]$ is as defined in \ref{nzz34}.
\end{definition}

In the next result, we show that every polarity algebra is of the form
$A(E/K,\theta,\beta)$ for some polar triple $(E/K,\theta,\beta)$
as defined in \ref{nzz39}. See also \ref{nzz30}.

\begin{theorem}\label{tru98}
Let $P=(K,V,q,\theta,t\mapsto[t],\cdot)$ be a polarity algebra as defined in {\rm\ref{fat1}}. Then
$q$ is a quadratic form of type $F_4$ and there exists:
\begin{enumerate}[\rm(i)]
\item a standard decomposition $S=(E/K,F,\alpha,\beta)$
of $q$ such that $\alpha=\beta^{-\theta}$ and $F=K^\theta$,
\item an extension of $\theta$
to a Tits endomorphism of $E$ and
\item an identification of $V$ with $E\oplus E\oplus[K]$ with respect to which $t\mapsto[t]$ is
the map $t\mapsto(0,0,t)$, $\cdot$ is given by the formula in {\rm\ref{tru99}} and
$$q(u,v,t)=\beta^{-1}(N(u)+\beta^\theta N(v))+t^\theta$$
for all $(u,v,t)\in E\oplus E\oplus[K]$, where $N$ is the norm of the extension $E/K$.
\end{enumerate}
\end{theorem}

\begin{proof}
Let $F := K^\theta$, and let $\hat V$ be the set consisting of the symbols $\hat v$ for all $v \in V$,
i.e.\@ the map $v \mapsto \hat v$ is a bijection from $V$ to $\hat V$.
We make $\hat V$ into an $F$-vector space by defining
\[ s \cdot \hat v := \widehat{s^{\theta^{-1}} v} \]
for all $s \in F$ and all $\hat v \in \hat V$, or equivalently,
\begin{equation}\label{tthetahatv}
t^\theta \cdot \hat v := \widehat{tv}
\end{equation}
for all $t \in K$ and all $v \in V$.
The map $\hat q\colon\hat V\to F$ given by
\begin{equation}\label{hatqhatv}
\hat q(\hat v) := q(v)^\theta
\end{equation}
for all $\hat v \in \hat V$ is a quadratic form over $F$. For each $t \in K$, we define
\begin{equation}\label{tF}
[t]_F := \widehat{[t]} \in \widehat{[K]} \subset \hat V ,
\end{equation}
and for each $s \in F$, we define
\begin{equation}\label{sK}
[s]_K := [s^{\theta^{-1}}] \in [K] \subset V .
\end{equation}
Next we define maps from $V \times \hat V$ to $V$ and
from $\hat V \times V$ to $\hat V$ (both denoted by juxtaposition) by
\begin{equation}\label{vhatw}
v \hat{w} = vw\text{\quad and \quad} \hat{v} w = \widehat{vw}
\end{equation}
for all $v,w \in V$, where the
multiplication on the right hand side of both equations is the multiplication of the polarity algebra.
We claim that these data satisfy the axioms (F0)--(F12) of~\ref{abc30}.
To illustrate this,
we will prove (F2), (F4) and (F7) and leave the verification of the other axioms to the reader.

So let $v,w \in V$, $s \in F$, and $t \in K$;
then using \eqref{sK}, \eqref{vhatw}, (R2) and \eqref{tthetahatv}, we obtain
\[ \hat v [s]_K = \hat v [s^{\theta^{-1}}] = \widehat{v [s^{\theta^{-1}}]}
= \widehat{s^{\theta^{-1}} v} = s \hat v . \]
Thus (F2) holds.
Next, by \eqref{vhatw}, (R3), \eqref{tthetahatv} again and (F2), we obtain
\[ \hat v \cdot tw = \widehat{v \cdot tw} = \widehat{t^\theta \cdot vw}
= (t^\theta)^\theta \cdot \widehat{vw} = t^2 \hat v w = \hat v w \cdot [t^2]_K . \]
Thus (F4) holds.
By \eqref{vhatw}, (R5), \eqref{hatqhatv} and (F1),  we obtain
\[ v \hat w \cdot \hat w = vw \cdot w = q(w)^\theta \cdot v
= \hat q(\hat w) \cdot v = v \cdot [\hat q(\hat w)]_F . \]
Thus (F7) holds.
This (together with the proof of the remaining identities)
shows that $V$ and $\hat V$, together with the maps we just defined,
form a radical quadrangular system as defined
in~\cite[Appendix A.3.2]{tom2}. We denote this quadrangular system by $\Theta$.

Let $\Omega=(U_+,U_1,\ldots,U_4)$ and $x_1,\cdots,x_4$ be as in \ref{abc12}.
By \cite[Chapter~4]{tom2}, $\Omega$ is a root group sequence and thus
$\Omega$ determines a unique Moufang quadrangle $\Xi$ by \cite[7.5 and 8.5]{TW}.
It follows from \eqref{hatqhatv}, \eqref{tF},
\eqref{sK} and \eqref{vhatw} that there is a unique
anti-automorphism $\rho$ of $\Omega$ extending the maps
$x_i(\hat v)\mapsto x_{5-i}(v)$ for $i=1$ and $3$, and
$x_i(v)\mapsto x_{5-i}(\hat v)$ for $i=2$ and $4$. The square
$\rho^2$ centralizes $U_+$. Thus $\rho$ induces a polarity of the
Moufang quadrangle $\Xi$ (by \cite[7.5]{TW}).

By \ref{fat1}, $\partial f$ is not identically zero. By \ref{abc12}, therefore,
$[U_2,U_4]\ne1$. By \cite[17.4]{TW}, $[U_2,U_4]\ne1$ and the existence of an
anti-automorphism of $\Omega$ imply that $\Xi$ is a quadrangle of type $F_4$. In other words,
$\Xi$ is isomorphic to the
quadrangle ${\mathcal Q}_{\mathcal F}(\tilde\Lambda)=B_2^{\mathcal F}(\tilde\Lambda)$
obtained by applying
\cite[16.7]{TW} to some quadratic space $\tilde\Lambda=(\tilde K,\tilde V,\tilde q)$.
of type $F_4$. Let $Y_1=C_{U_1}(U_3)$, let $Y_3=C_{U_3}(U_1)$ and let
$Y_+=Y_1U_2Y_3U_4$. By (F5) and \ref{abc12}, $Y_i=x_i([K]_F)$ for $i=1$ and $3$,
$Y_+$ is a subgroup of $U_+$ and $(Y_+,Y_1,U_2,Y_3,U_4)$ is a root group sequence isomorphic to
the root group sequence ${\mathcal Q}_{\mathcal Q}(K,V,q)$ obtained by applying
\cite[16.3]{TW} to $(K,V,q)$. By \cite[16.7]{TW}, on the other hand,
$(Y_+,Y_1,U_2,Y_3,U_4)$ is a root group sequence isomorphic to
${\mathcal Q}_{\mathcal Q}(\tilde\Lambda)$. By \cite[35.8]{TW}, it follows
that $K\cong\tilde K$ and $q$ is similar to $\tilde q$. Thus $q$ is of type $F_4$
and $\Xi\cong B_2^{\mathcal F}(K,V,q)$ (by \cite[35.12]{TW}).
We conclude that the quadrangular system $\Theta$ is an extension of the
quadrangular system associated with $(K,V,q)$;
see the beginning of \cite[Chapter~8]{tom2} for the definition of these terms.
This is exactly the situation investigated in \cite[\S8.5]{tom2} (and \cite[Chapter~28]{TW}).
By \cite[Thm.~8.107]{tom2}, there exists a standard decomposition $S$ of $q$ and
an isomorphism $\xi$ from $\Omega$ to the root group sequence $\Omega_S$ obtained by applying
\ref{abc61} to $S$ extending the maps $x_i([t]_F)\mapsto x_i(0,0,t)$ for $i=1$ and $3$
and $x_i([s]_K)\mapsto x_i(0,0,s)$ for $i=2$ and $4$. We now replace $\rho$
by the unique automorphism of ${\mathcal Q}(S)$ obtained by applying \cite[7.5]{TW}
to the automorphism $\xi^{-1}\rho\xi$ of $\Omega_S$.
By \eqref{tF} and \eqref{sK}, we have $x_i(0,0,t)^\rho=x_i(0,0,t^\theta)$
for all $t\in K$. Thus $\theta$ is as in \ref{abc40}.
By \ref{tru1}, we conclude that (i)--(iii) hold.
\end{proof}

\begin{corollary}\label{onk111}
Every polarity algebra is of the form $A(E/K,\theta,\beta)$ for some polar triple
$(E/K,\theta,\beta)$ as defined in {\rm\ref{nzz39}}.
\end{corollary}

\begin{proof}
This holds by \ref{nzz31} and \ref{tru98}.
\end{proof}

\section{Two Examples}\label{sec88}

In this section, we give two examples illustrating the results of the
previous section; see \ref{onk96}.

\begin{example}\label{tru7}
Let $K={\mathbb F}_2(\alpha,\beta)$ be a purely transcendental extension of
the field ${\mathbb F}_2$, let $E$ be the splitting field of the polynomial
$$p(x)=x^2+x+\alpha+\beta^2$$
over $K$, let $\gamma\in E$ be a root of $p(x)$,
let $\theta$ denote the unique Tits endomorphism of $K$ that maps
$\beta$ to $\alpha$ and let $F=K^\theta$. By \ref{abc71}(i), $\theta$ has an extension to a
Tits endomorphism of $E$. We leave it to the reader to check that
$S:=(E/K,F,\alpha^{-1},\beta)$ is an $F_4$-datum
as defined in \ref{def0a}.
Let $q=q_S$ be the quadratic form of type $F_4$ on $E\oplus E\oplus[F]_K$
as defined in \ref{def0a}.
We claim that $\beta$ is not a Tits trace of $K$
(with respect to $\theta$). To show this, we assume that
$$\beta=g(\alpha,\beta)+g(\alpha,\beta)^\theta=g(\alpha,\beta)+g(\beta^2,\alpha)$$
for some rational function $g(\alpha,\beta)\in K$.
Let $k={\rm deg}_\alpha(g)$ and $m={\rm deg}_\beta(g)$.
(If $g=g_1/g_2$ for polynomials $g_1$ and $g_2$ in ${\mathbb F}_2[\alpha,\beta]$
and $u=\alpha$ or $\beta$, then
${\rm deg}_u(g)={\rm deg}_u(g_1)-{\rm deg}_u(g_2)$.) We have
\begin{equation}\label{tru7a}
1={\rm deg}_\beta(\beta)={\rm deg}_\beta\big(g(\alpha,\beta)+g(\beta^2,\alpha)\big)\le
{\rm max}(2k,m)
\end{equation}
and
\begin{equation}\label{tru7b}
1={\rm max}(2k,m)\text{ if $m\ne 2k$}
\end{equation}
and
$$0={\rm deg}_\alpha(\beta)={\rm deg}_\beta\big(g(\alpha,\beta)+g(\beta^2,\alpha)\big)\le
{\rm max}(k,m)$$
and $0={\rm max}(k,m)$ if $k\ne m$.
By \eqref{tru7a}, we have $0\ne{\rm max}(k,m)$. Thus $k=m\ne0$. Hence $m\ne 2k$ and
${\rm max}(2k,m)\ge2$, which is impossible by \eqref{tru7b}. We conclude that $\beta$
is not a Tits trace in $K$ as claimed. By \ref{abc90}(iv), it follows that also
$\beta^2$ is not a Tits trace in $K$. Let $L$ be the splitting field of the polynomial
$$p_1(x)=x^2+x+\beta^2$$
over $K$. By \ref{abc71}(i), $\theta$ does not
have an extension to a Tits endomorphism of $L$. Let $d=(\beta,0,0)$ and $e=(\gamma,\alpha,0)$
in $V$. Then $q(d)=q(e)=\beta$ and $f(d,e)=1$, so $L$ is also the splitting field
of $q(d)x^2+x+q(e)$ over $K$.
Applying \ref{abc77} with $\xi=(1,0,0)\in\hat V$, we
conclude that $L$ is a splitting field of $q$. Thus the Tits endomorphism
$\theta$ of $K$ has an extension to a Tits endomorphism of some of the splitting
fields of $q$ but there are also splitting fields of $q$ to which $\theta$
does not have an extension to a Tits endomorphism.
\end{example}

\begin{proposition}\label{tru50}
Let $S=(E/K,F,\alpha,\beta)$ be an $F_4$-datum, let $\Xi={\mathcal Q}(S)$ and suppose that
$\theta$ is a Tits endomorphism of $K$ such that $F=K^\theta$
and $\alpha=\beta^{-\theta}$. Choose $\lambda\in K$ such that $E$ is the splitting
field of the polynomial $x^2+x+\lambda$ over $K$.
Then $\Xi$ has a polarity if and only if there exists $u\in K$ such that
$$\lambda+\alpha u^2$$
is a Tits trace with respect to $\theta$.
\end{proposition}

\begin{proof}
Let $q=q_S$, let $f=\partial q$ and let $\gamma\in E$ be a root of $x^2+x+\lambda$.
Let $V$, $D$ and $\hat V$ be as in \ref{abc83}, let $d=(1,0,0)$ and $e=(\beta\gamma,0,0)$ in $V$ and
let $\xi=(1,0,0)\in\hat V$. Then $q(d)=\beta^{-1}$ and $q(e)=\beta\lambda$. Hence
$\omega:=\beta\gamma$ is a root of $q(d)x^2+x+q(e)$.

We suppose now that $u$ is an element of $K$ such that $\lambda+\alpha u^2$
is a Tits trace and let $E'$ be the splitting field of $x^2+x+\lambda+\alpha u^2$
over $K$. By \ref{abc71}(i) and
the choice of $u$, we can choose a Tits endomorphism $\theta_1$ of $E'$ extending $\theta$.
As in \ref{tru6}, we have $(E')^{\theta_1}=(E')^2F$. Since $\alpha u^2\in F$,
we can set $e'=e+(0,0,\alpha u^2)$. Thus $f(d,e')=1$ and, by \ref{abc30}(F0) and (F6),
$f(d,e'\xi)=0$. Applying \ref{abc77} with $e'$ in place of $e$, it follows
that we can assume that $E'=E$. Now let
$\varphi\colon V\to\hat V$ be given by the formula in \ref{tru4} and let $\rho$
be defined by the equations in \ref{abc97} with $\varphi_1=\varphi$ and $\hat\varphi=\hat\varphi_1=\varphi^{-1}$.
Then $\rho$ is an automorphism of $U_+$ of order~$2$ mapping $U_i$ to $U_{5-i}$ for
each $i\in[1,4]$. Thus $\Xi$ has a polarity.

Suppose, conversely, that $\Xi$ has a polarity $\rho$. Our goal is to find an
element $u\in K$ such
that $\lambda+\alpha u^2$ is a Tits trace.
By \ref{tru1}, we can choose $e'\in V$ and $\xi'\in\hat V$ such that
$f(d,e')=1$, $f(d,e'\xi')=0$ and $\hat q(\xi')=\alpha$ such that $\theta$ has an extension
to the splitting field of $x^2+x+q(d)q(e')$ over $K$. Thus $q(d)q(e')$ is
a Tits trace. Since $f(d,e')=1$, we have
$e'=(t+\beta\gamma,y+z\gamma,s)$ for some $t,y,z\in K$
and some $s\in F$. Let $e''=e'+(t,0,0)$. By \cite[Lemma~2.1]{tom1},
we have $f(d,e''\xi')=0$. We also have $q(e'')=q(e')+\beta^{-1}t^2+t$ and
hence $q(d)q(e'')+q(d)q(e')=\beta^{-2}t^2+\beta^{-1}t$.
By \ref{abc90}(v), this expression is a Tits trace. It follows that we can assume that
\begin{equation}\label{tru50a}
e'=(\beta\gamma,y+z\gamma,s).
\end{equation}
Hence
\begin{align*}
q(d)q(e')&=\beta^{-2}\big(N(\beta\gamma)+\alpha(y^2+yz+\lambda z^2)\big)+\beta^{-1}s\\
&=\lambda+\alpha\beta^{-2}(y^2+yz+\lambda z^2)+\beta^{-1}s.
\end{align*}
We have $s=x^\theta$ for some $x\in K$ and thus
$$(\beta^{-1}s)^\theta=\alpha s^\theta=\alpha x^2.$$
By \ref{abc90}(i), therefore,
\begin{align}
p:&=\lambda+\alpha\beta^{-2}(y^2+yz+\lambda z^2)+\alpha x^2\notag\\
&=\lambda+\alpha\beta^{-2}\big((y+\beta x)^2+z(y+\lambda z)\big)\label{tru50f}
\end{align}
is a Tits trace.

By \ref{abc30}(F12) and the choice of $e'$ and $\xi'$, we have
\begin{equation}\label{tru50b}
f(d\xi',e')=f(d,e'\xi')=0.
\end{equation}
By \ref{abc30}(F12), we also have $f(d\xi',d)=0$, from which it follows that
there exist $w,u,v,r\in K$ such that
\begin{equation}\label{tru50e}
d\xi'=(w,u+v\gamma,r^\theta).
\end{equation}
By \cite[8.95]{tom2}, we have $q(d\xi')=q(d)\hat q(\xi')=\beta^{-1}\alpha$.
Hence
\begin{equation}\label{tru50c}
w^2+\alpha(u^2+uv+\lambda v^2+1)+\beta r^\theta=0.
\end{equation}
By \eqref{tru50a}, \eqref{tru50b} and \eqref{tru50e}, we have
\begin{equation}\label{tru50d}
\beta w+\alpha(zu+yv)=0.
\end{equation}
Suppose that $v=0$. Then $w^2+\alpha(u^2+1)+\beta r^\theta=0$ by \eqref{tru50c},
hence $\beta r^\theta\in F$ and therefore, $r=0$ since $\beta\not\in F$.
Hence $\alpha(u+1)^2\in K^2$ and therefore $u=1$ since $\alpha\not\in K^2$.
Hence $w=0$. By \eqref{tru50d}, therefore, $z=0$. Thus by \eqref{tru50f},
we have $p=\lambda+\alpha a^2$ for $a=\beta^{-1}(y+\beta x)$.

Suppose, finally, that $v\ne0$. Then $y=v^{-1}(\alpha^{-1}\beta w+zu)$
by \eqref{tru50d}. Hence
\begin{alignat}{2}
\alpha\beta^{-2}z(y+\lambda z)&=\alpha\beta^{-2}zv^{-1}(\alpha^{-1}\beta w+zu+\lambda vz)\notag\\
&=\beta^{-1}zv^{-1}w+\beta^{-2}z^2v^{-2}\cdot\alpha(uv+\lambda v^2)\notag\\
&=\beta^{-1}zv^{-1}w+\beta^{-2}z^2v^{-2}(\alpha(u^2+1)+w^2+\beta r^\theta)
& &\qquad\text{by \eqref{tru50c}}
\notag\\
&=\alpha\big(\beta^{-1}zv^{-1}(u+1)\big)^2\notag\\
&\qquad+\beta^{-1}zv^{-1}w+(\beta^{-1}zv^{-1}w)^2+\beta^{-1}z^2v^{-2}r^\theta.\notag
\end{alignat}
By \ref{abc90}(v) and \eqref{tru50f}, it follows that
\begin{equation}\label{tru50g}
\lambda+\alpha b^2+\beta^{-1}z^2v^{-2}r^\theta
\end{equation}
is a Tits trace for $b=\beta^{-1}\big(zv^{-1}(u+1)+(y+\beta x)\big)$.
Adding the Tits trace
$$\beta^{-1}z^2v^{-2}r^\theta+(\beta^{-1}z^2v^{-2}r^\theta)^\theta$$
to the expression \eqref{tru50g}, we conclude that
$$\lambda+\alpha b^2+(\beta^{-1}z^2v^{-2}r^\theta)^\theta$$
is also a Tits trace. Finally, we observe that
$$(\beta^{-1}z^2v^{-2}r^\theta)^\theta=\alpha c^2$$
for $c=z^\theta v^{-\theta}r$. Thus $\lambda+\alpha(b+c)^2$ is a Tits trace.
\end{proof}

\begin{example}\label{tru51}
Let $K={\mathbb F}_2(\alpha,\beta)$ be a purely transcendental extension of
the field ${\mathbb F}_2$, let $E$ be the splitting field of the polynomial
$$p(x)=x^2+x+1$$
over $K$, let $\gamma\in E$ be a root of $p(x)$,
let $\theta$ denote the unique Tits endomorphism of $K$ that maps
$\beta$ to $\alpha$ and let $F=K^\theta$. By \cite[14.25]{TW},
$S:=(E/K,F,\alpha^{-1},\beta)$ is an $F_4$-datum,
so we can set $\Xi={\mathcal Q}(S)$. There are exactly three elements
of $E^*$ of finite order. Let $\hat\theta$ be the unique extension
of $\theta$ to an endomorphism of $E$ which acts trivially on these
three elements and let $\chi$ denote the non-trivial element of ${\rm Gal}(E/K)$.
The endomorphism $\hat\theta$ is, of course, not a Tits endomorphism of $E$.
(By \ref{abc90}(vi) and \ref{abc71}(i),
$\theta$ does not have an extension to a Tits endomorphism of $E$.)
Let $V$, $\hat V$,
$\Omega:=(U_+,U_1,\ldots,U_4)$ and $x_1,\ldots,x_4$ be as in \ref{abc61},
let $\varphi$ denote
the map from $V$ to $\hat V$ given by
$$\varphi(u,v,s)=(u^{\hat\theta},\beta^{-2}v^{\hat\theta},s^{\theta^{-1}})$$
for all $(u,v,s)\in V$ and let $\psi$ denote the automorphism of $V$ given by
$$\psi(u,v,s)=(u^\chi,v^\chi,s)$$
for all $(u,v,s)\in V$. There is a unique anti-automorphism $\kappa$ of $\Omega$
extending the maps
$x_i(b)\mapsto x_{5-i}(\varphi(\psi(b))$
for $i=2$ and $4$ and
$x_i(a)\mapsto x_{5-i}(\varphi^{-1}(a))$
for $i=1$ and $3$. The square of $\kappa$ is an involution. By \cite[7.5]{TW},
therefore, $\kappa$ gives rise to a non-type-preserving automorphism of $\Xi$ of order~$4$.

We claim that $\Xi$ does not, however,
have any polarities. Let $\Gamma$ be the additive group
$$\{a+b\sqrt{2}\mid a,b\in{\mathbb Z}\},$$
let $k$ denote the field of Hahn series ${\mathbb F}_2(t^\Gamma)$ and
let $\nu\colon k^*\to\Gamma$ be the canonical valuation on $k$ (as described, for example,
in \cite[3.5.6]{prestel}). There is a unique embedding $\pi$ from $K$ to $k$ which sends
$\beta$ to $t$ and $\alpha$ to $t^{\sqrt{2}}$. We identify $K$ with
its image under $\pi$. Modulo this identification, there is a unique
extension of $\theta$ to a Tits endomorphism of $k$ which we also
denote by $\theta$. If $u,v\in k$, then the constant coefficient of
$t^{\sqrt{2}}u^2$ is $0$ and the constant
coefficient of $v$ is the same as the constant coefficient of $v^\theta$.
It follows that there do not exist $u,v\in k$ such that
$$1+t^{\sqrt{2}}u^2=v+v^\theta.$$
By \ref{tru50} with $\alpha^{-1}$ in place of $\alpha$,
it follows that our Moufang quadrangle $\Xi$ does not
have a polarity, as claimed.
\end{example}

\smallskip
\section{Buildings of Type $F_4$}\label{sec25}

The main results of this section are \ref{onk13} and \ref{onk13x}.

\begin{notation}\label{baf103}
Let $L/E$ be a field extension such that ${\rm char}(E)=2$ and $L^2\subset E$
and let $\Delta={\sf F}_4(L,E)$ as defined in \cite[30.15]{affine}.
Let $\Phi$ be a root system of type $F_4$, let $\Sigma$
be an apartment of $\Delta$, let $c$ be a chamber of $\Sigma$ and for
each $\alpha\in\Phi$, let $s_\alpha$ denote the corresponding reflection.
Let $\alpha_1,\ldots,\alpha_4$ be a basis of $\Phi$ ordered
so that $\alpha_1$ and $\alpha_2$ are long and $|s_{\alpha_2}s_{\alpha_3}|=4$, let
$S$ be the set of reflections $s_{\alpha_i}$ for $i\in[1,4]$ and let
$W=\langle S\rangle$ be the Weyl group of $\Phi$. We think of the map
$i\mapsto\alpha_i$ as a bijection from the vertex set of the Coxeter
diagram $F_4$ to $S$. There is a unique action of $W$ on $\Sigma$
with respect to which $s_{\alpha_i}$ interchanges $c$ with the unique
chamber of $\Sigma$ that is $i$-adjacent to $c$, there is a unique chamber
$C$ of $\Phi$ contained in the half-space determined by $\alpha_i$ for all $i\in[1,4]$
and there is a unique
$W$-equivariant bijection $\iota$ from the set of chambers of $\Sigma$ to the set
of chambers of $\Phi$ mapping $c$ to $C$. The bijection $\iota$ induces a
bijection from the set of roots of $\Sigma$ to $\Phi$ and its inverse
induces an injection from ${\rm Aut}(\Phi)$ to ${\rm Aut}(\Sigma)$. From
now on, we identify ${\rm Aut}(\Phi)$ with its image under this injection
and we identify the roots of $\Sigma$ with the corresponding elements of $\Phi$.
Thus for each $\beta\in\Phi$, we have a root group $U_\beta$ of $\Delta$.
\end{notation}

\begin{theorem}\label{wax31}
There exists a collection of isomorphisms $x_\beta\colon E\to U_\beta$, one
for each long root $\beta$ of $\Phi$, and
a collection of isomorphisms $x_\beta\colon L\to U_\beta$, one for each short root $\beta$
of $\Phi$, such that for all $\alpha,\beta\in\Phi$ at an
angle $\omega<180^\circ$ to each other and for
all $s$ in the domain of $x_\alpha$ and all $t$ in the domain of $x_\beta$,
the following hold:
\begin{enumerate}[\rm(i)]
\item If $\omega=120^\circ$, then $\alpha+\beta\in\Phi$ and
$[x_\alpha(s),x_\beta(t)]=x_{\alpha+\beta}(st)$.
\item If $\omega=135^\circ$, then $\alpha$ and $\beta$ have different lengths;
if $\alpha$ is long, then
$\alpha+\beta\in\Phi$, $\alpha+2\beta\in\Phi$ and
$[x_\alpha(s),x_\beta(t)]=x_{\alpha+\beta}(st)x_{\alpha+2\beta}(st^2)$.
\item $[x_\alpha(s),x_\beta(t)]=1$ if $\omega$ is neither $120^\circ$ nor $135^\circ$.
\end{enumerate}
\end{theorem}

\begin{proof}
This holds by \cite[5.2.2]{carter} and \cite[10.3.2]{bn}.
\end{proof}

\begin{definition}\label{baf901}
We call a set $\{x_\beta\}_{\beta\in\Phi}$ satisfying the three conditions in \ref{wax31}
a {\it coordinate system} for $\Delta$.
\end{definition}

\begin{theorem}\label{wax34}
Let $\{x_\beta\}_{\beta\in\Phi}$ be a coordinate system for $\Delta$,
let $\gamma\in{\rm Aut}(\Phi)$,
let $\lambda_1,\lambda_2$ be non-zero elements of $E$,
let $\lambda_3,\lambda_4$ be non-zero elements of $L$ and let $\chi$
be an element of ${\rm Aut}(L)$
stabilizing $E$. Then the following hold:
\begin{enumerate}[\rm(i)]
\item There exists a unique automorphism
$$g=g_{\gamma,\lambda_1,\lambda_2,\lambda_3,\lambda_4,\chi}$$
of $\Delta$ that stabilizes the apartment $\Sigma$ such that
$$x_{\alpha_i}(t)^g=x_{\gamma(\alpha_i)}(\lambda_i t^\chi)$$
for all $t$ in the domain of $x_{\alpha_i}$ and for all $i\in[1,4]$.
\item If
$$\beta=\sum_{i=1}^4c_i\alpha_i\in\Phi,$$
then
$$x_\beta(t)^g=x_{\gamma(\beta)}(\lambda_\beta t^\chi)$$
for all $t$ in the domain of $x_\beta$, where
$$\lambda_\beta=\prod_{i=1}^4\lambda_i^{c_i}.$$
\end{enumerate}
\end{theorem}

\begin{proof}
Inserting $\chi$ into \cite[Lemma~58]{stein} and restricting
scalars to $E$ in the long root groups, we obtain the existence assertion in (i);
(see also \cite[Thm.~29]{stein}). The uniqueness assertion
holds by \cite[9.7]{spherical}. By \ref{wax31}(i)--(iii), \cite[\S10.2, Lemma~A]{humph} and induction,
it follows that (ii) holds for all $\beta\in\Phi^+$ (i.e.~for all $\beta\in\Phi$
that are positive with respect to the basis
$\{\alpha_1,\ldots,\alpha_4\}$).
For each $i\in[1,4]$, there exists a unique $j\in[1,4]$ such that the angle between
$\alpha_i$ and $\alpha_j$ is $120^\circ$. By \ref{wax31}(i),
$\beta:=\alpha_i+\alpha_j\in\Phi$
and $[x_\beta(1),x_{-\alpha_j}(t)]=x_{\alpha_i}(t)$ for all $t$ in the domain of $x_{-\alpha_j}$
(which is the same as the domain of $x_{\alpha_j}$).
Conjugating by $g$ and applying (i),
we conclude that $x_{-\alpha_j}(t)^g=x_{\gamma(-\alpha)}(\lambda_j^{-1}t)$ for all $t$ in
the domain of $x_{\alpha_j}$.
Thus by \ref{wax31}(i)--(iii), \cite[\S10.2, Lemma~A]{humph} and induction again,
(ii) holds for all $\beta\in\Phi^-$.
\end{proof}

\begin{proposition}\label{wax34x}
Every type-preserving automorphism of $\Delta$ that stabilizes $\Sigma$
is of the form
$$g_{\gamma,\lambda_1,\ldots,\lambda_4,\chi}$$
for some
$\gamma\in{\rm Aut}(\Phi)$, some $\lambda_1,\lambda_2\in E$, some $\lambda_3,\lambda_4\in L$
and some $\chi\in{\rm Aut}(L,E)$.
\end{proposition}

\begin{proof}
By \ref{wax34}(i), it suffices to show that every
type-preserving automorphism of $\Delta$ that stabilizes $\Sigma$
pointwise is of the desired form. Let $g$ be such an element.
By \cite[9.7]{spherical}, $g$ is uniquely
determined by its restrictions to the irreducible rank~$2$ residues containing $c$.
These are isomorphic to $A_2(E)$, $B_2^{\mathcal D}(\Lambda)$ and $A_2(L)$,
where $\Lambda$ is the indifferent set $(L,L,E)$.
By \cite[37.13]{TW}, it follows that there exist $\lambda_1,\lambda_2\in E^*$,
$\lambda_3,\lambda_4\in L$, $\chi_E\in{\rm Aut}(E)$ and $\chi_L\in{\rm Aut}(L)$
such that $x_{\alpha_i}(t)^g=x_{\alpha_i}(\lambda_it^{\chi_E})$ for all $t\in E$ if $i=1$ or 2
and $x_{\alpha_i}(t)^g=x_{\alpha_i}(\lambda_i t^{\chi_L})$ for all $t\in L$ if $i=3$ or 4.
By \cite[37.32]{TW} applied to the indifferent set $(L,L,E)$, $\chi_L\in{\rm Aut}(L,E)$
and the restriction of $\chi_L$
to $E$ equals $\chi_E$. Thus $g=g_{{\rm id},\lambda_1,\ldots,\lambda_4,\chi}$ for
$\chi=\chi_L$.
\end{proof}

\begin{remark}\label{onk110}
By \cite[28.8]{MPW}, $\{L/E,E/L\}$ is the pair of defining extensions of $\Delta$; see \ref{onk2x}.
Let $G^\circ$ and $G^\dagger$ be as in \ref{goo3}. By \cite[2.8 and 11.12]{spherical},
the stabilizer $G^\dagger_\Sigma$ induces the same group as the stabilizer $G^\circ_\Sigma$
on $\Sigma$. Thus every element in $G^\circ_{\Sigma}$ is conjugate by an element in $G^\dagger$
to one which fixes the chamber $c$ of $\Sigma$.
By \ref{goo5}(i)--(ii), therefore, we can choose a Galois map $\psi$ of $\Delta$ such that
$$\psi(g_{\gamma,\lambda_1,\ldots,\lambda_4,\chi})=\chi$$
for all $\gamma\in{\rm Aut}(\Phi)$, for all $\lambda_1,\lambda_2\in E$, for all
$\lambda_3,\lambda_4\in L$ and for all $\chi\in{\rm Aut}(L,E)$.
\end{remark}

\begin{notation}\label{onk10}
Let $w_1=(s_{\alpha_2}s_{\alpha_3})^2\in{\rm Aut}(\Phi)$, where $s_{\alpha_2}$ and $s_{\alpha_3}$
are as in \ref{baf103}.
\end{notation}

\begin{notation}\label{onk11}
Let $\chi$ be an involution in the group ${\rm Aut}(L,E)$ defined as in \ref{onk2x},
let $F_0={\rm Fix}_L(\chi)$,
let $K=F_0\cap E$ and
let $N$ be the norm of the extension $L/F_0$. Thus $F_0/K$ is a purely
inseparable extension such that $F_0^2\subset K$, the restriction of $N$ to $E$
is the norm of the extension $E/K$ and $L$ is the composite $EF_0$.
\end{notation}

\begin{notation}\label{onk12}
Let $\chi$, $F_0$ and $K$ be as in \ref{onk11}, let $F=F_0^2$ and suppose that
$$S=(E/K,F,\alpha,\beta)$$
is an $F_4$-datum for some $\alpha\in F$ and
some $\beta\in K$. Let $\lambda_1=\alpha\beta^{-1}$, let $\lambda_2=\alpha^{-1}$, let
$\lambda_3=\beta$, let $\lambda_4$ be the unique element of $F_0$ such that
$\lambda_4^2=\beta^{-2}\alpha$ and let
$$\xi=g_{w_1,\lambda_1,\lambda_2,\lambda_3,\lambda_4,\chi},$$
where $w_1$ is as in \ref{onk10}.
\end{notation}

In the next two results, we use the term ``$\chi$-involution'' (as defined in \ref{nzz10})
with respect to the Galois map $\psi$ chosen in \ref{onk110}.

\begin{theorem}\label{onk13}
Let $\Delta$ be as in {\rm\ref{baf103}}, let
$S$ and $\xi$ be as in {\rm\ref{onk12}} and let $\Gamma=\langle\xi\rangle$.
Then $\xi$ is a type-preserving isotropic $\chi$-involution of $\Delta$, $\Gamma$-chambers
are residues of type $\{2,3\}$ and
$$\Delta^\Gamma\cong{\mathcal Q}(S).$$
\end{theorem}

\begin{proof}
This holds by \cite[p.~368 at the bottom]{canada}. See \cite[17.14]{galois} for a shorter proof.
See also \ref{onk103}.
\end{proof}

\begin{theorem}\label{onk13x}
Let $\Delta$ be as in {\rm\ref{baf103}}, let
$\xi$ be an arbitrary type-preserving $\chi$-involution of $\Delta$ for
some $\chi\in{\rm Aut}(L,E)$, let $\Gamma=\langle\xi\rangle$ and
suppose that $\Gamma$-chambers are residues of type $\{2,3\}$.
Then the following hold:
\begin{enumerate}[\rm(i)]
\item There exist $\alpha\in F$ and $\beta\in K$
such that $\xi$ is conjugate by an element in $G^\dagger$ to
$$g_{w_1,\alpha\beta^{-1},\alpha^{-1},\beta,\beta^{-1}\sqrt{\alpha},\chi}.$$
\item $\Delta^\Gamma$ is a Moufang quadrangle of type $F_4$.
\end{enumerate}
\end{theorem}

\begin{proof}
By \cite[Lemma~3.2]{canada}, for every $\Gamma$-chamber $R$, there exists
an apartment that is stabilized by $\xi$ and contains chambers of $R$.
By \cite[11.12]{spherical}, there exists an element $\delta$ in the group $G^\dagger$ such that
$\xi^\delta$ stabilizes the apartment $\Sigma$ and the unique $\{2,3\}$-residue
containing $c$, where $\Sigma$ and $c$ are as in \ref{baf103}.
By \cite[25.17]{MPW}, $\xi^\delta$ induces the automorphism $w_1$ on $\Sigma$
and by \ref{goo5}(i), $\xi^\delta$ is also a $\chi$-involution.
By \ref{wax34x} and \cite[Lemma~4.3]{canada}, it follows that
there exist $\alpha\in F$ and $\beta\in K$ such that
$\xi^\delta=g_{w_1,\alpha\beta^{-1},\alpha^{-1},\beta,\beta^{-1}\sqrt{\alpha},\chi}$. Thus (i)
holds. By \ref{onk13}, (ii) follows from (i).
\end{proof}

\section{$F_4$-Buildings with Polarity}

The goal of this section is to prove \ref{onk50}.

\begin{notation}\label{onk43}
Suppose now that
$\Xi$, $\rho$, $S=(E/K,F,\alpha,\beta)$, $\theta$ and the identification of $\Xi$
with ${\mathcal Q}(S)$ are as in \ref{tru68}. Let $F_0=F^{1/2}$ in the algebraic closure
of $E$. Thus $K\subset F_0$ and $F_0^2=F$. Let $L$ be the composite field $EF_0$.
Choose $\gamma\in K$ such that $E=K(\gamma)$. Then $L=F_0(\gamma)$.
In particular, $L/F_0$ is a separable quadratic extension. Let
$\chi$ be the generator of ${\rm Gal}(L/F_0)$.
The map $x\mapsto ((x^2)^\theta)^{1/2}$
is the unique extension of $\theta$ to a Tits endomorphism of $L$. We denote this extension
by the same letter $\theta$.
Since $K^\theta=F$, we have $F_0^\theta=K$. Thus
$K=F_0^\theta\ne L^\theta=K(\gamma^\theta)$. Since $E^\theta\subset E$,
it follows that $E=K(\gamma^\theta)$.
Hence $L^\theta=E$. By \ref{abc71}(ii), $\theta$ commutes with $\chi$.
We set $\Delta=F_4(L,E)$.
\end{notation}

Let $c$, $\Sigma$, $\Phi$, $\{\alpha_1,\ldots,\alpha_4\}$, the identification of
$\Phi$ with the set of roots of $\Sigma$, etc., be as in \ref{baf103}
applied to $\Delta=F_4(L,E)$,
let $\{x_\alpha\}_{\alpha\in\Phi}$
be as in \ref{wax31}, let
$$|\Phi|=\{\alpha/|\alpha|\mid \alpha\in\Phi\}$$
and let $\pi$ be denote the bijection $\alpha\mapsto\alpha/|\alpha|$ from $\Phi$
to $|\Phi|$. We now identify the set of roots of $\Sigma$ with $|\Phi|$ via $\pi$.

\begin{notation}\label{nzz59}
Let $\dot x_{\pi(\alpha)}(t)=x_\alpha(t)$ for all $t\in E$
and all long $\alpha\in\Phi$, let $\dot x_{\pi(\alpha)}(t)=x_\alpha(t^{\theta^{-1}})$
for all $t\in E$
and all short $\alpha\in\Phi$ and let $U_{\pi(\alpha)}=U_\alpha$ for all $\alpha\in\Phi$.
Thus $\dot x_\alpha$ is an isomorphism from the
additive group of $E$ to $U_\alpha$ for each $\alpha\in|\Phi|$ and by \ref{wax31},
if $s,t\in E$ and $\alpha$ and $\beta$ are elements of $|\Phi|$ with an angle $\omega<180^\circ$
between them, then the following hold:
\begin{enumerate}[\rm(i)]
\item If $\omega=120^\circ$, then $\alpha+\beta\in|\Phi|$ and
$[\dot x_\alpha(s),\dot x_\beta(t)]=\dot x_{\alpha+\beta}(st)$
\item If $\omega=135^\circ$, then $\sqrt{2}\alpha+\beta\in|\Phi|$,
$\alpha+\sqrt{2}\beta\in|\Phi|$ and
$[\dot x_\alpha(s),\dot x_\beta(t)]
=\dot x_{\sqrt{2}\alpha+\beta}(s^\theta t)\dot x_{\alpha+\sqrt{2}\beta}(st^\theta)$.
\item $[\dot x_\alpha(s),\dot x_\beta(t)]=1$ if $\omega$ is neither $120^\circ$ nor $135^\circ$.
\end{enumerate}
\end{notation}

Let
$$B:=\{\eta_1,\ldots,\eta_4\}$$
be the image of the basis $\{\alpha_1,\ldots,\alpha_4\}$ of $\Phi$ under $\pi$.
We set $m'=\sqrt{2}m$ for each positive integer $m$ and
$$abcd=a\eta_1+b\eta_2+c\eta_3+d\eta_4$$
for all $a,b,c,d\in{\mathbb N}\cup\sqrt{2}{\mathbb N}$. Thus, for example,
$$1'2'21=\sqrt{2}\eta_1+2\sqrt{2}\eta_2+2\eta_3+\eta_4.$$
We then set
\begin{align*}
W_0&=\{0100,0010,011'0,01'10\},\\
W_1&=\{0001,0011,011'1',01'11,01'21\},\\
W_2&=\{111'1',121'1',1'2'32,122'1',132'1'\},\\
W_3&=\{1'1'11,1'1'21,232'1',1'2'21,1'2'31\},\\
W_4&=\{1000,1100,1'1'10,111'0,121'0\}.
\end{align*}
Let $|\Phi^+|$ denote the image under $\pi$ of the set of positive
roots of $\Phi$ with respect to the basis $\{\alpha_1,\ldots,\alpha_4\}$. Then
$$|\Phi^+|=W_0\cup W_1\cup W_2\cup W_3\cup W_4.$$

\begin{notation}\label{onk49}
Let $R_1$ be the unique $\{2,3,4\}$-residue of $\Delta$ containing $c$, let
$R_4$ be the unique $\{1,2,3\}$-residue containing $c$, let $R=R_1\cap R_4$ and
for $i=1$ and $4$, let $R_i'$ be the unique residue such that $R_i'\cap\Sigma$
is opposite $R\cap\Sigma$ in $R_i\cap\Sigma$. Then
$W_i$ is the set of roots of $\Sigma$ that contain $R\cap\Sigma$
but are disjoint from $R_i'\cap\Sigma$ for $i=1$ and $4$.
\end{notation}

\begin{notation}\label{onk101}
There exists a unique set $X$ of $\{2,3\}$-residues of $\Sigma$
containing $R\cap\Sigma$ with the property that there exists a bijection
$i\mapsto T_i$ from ${\mathbb Z}_8$ to $X$ such that
for each $i\in{\mathbb Z}_8$, $T_{i-1}$ and $T_i$ are opposite residues of
a residue of rank~3 of $\Sigma$. We denote by $\Lambda$ the graph with vertex
set $X$, where $T_i$ is adjacent to $T_j$ whenever $i-j=\pm1$. Thus the residues
$R_1'\cap\Sigma$ and $R_4'\cap\Sigma$ are the two vertices adjacent to $R\cap\Sigma$ in
$\Lambda$.
\end{notation}

\begin{notation}\label{onk102}
Let $\tilde X$ be the graph obtained from the set $X$ in \ref{onk101}
by replacing each vertex $T_i$ by the unique residue $\tilde T_i$ of $\Delta$ such
that $\tilde T_i\cap\Sigma=T_i$. Let $\tilde\Sigma$ be the graph with vertex set
$\tilde X$, where $\tilde T_i$ is adjacent to $\tilde T_j$ whenever $i-j=\pm1$.
\end{notation}

\begin{notation}\label{onk40}
Let $\kappa$ denote the unique involutory permutation of $|\Phi|$ which
interchanges $abcd$ with $dcba$ for all $abcd\in|\Phi|$.
Note that $W_0^\kappa=W_0$ and $W_i^\kappa=W_{5-i}$ for each $i\in[1,4]$.
By \cite[1.2]{octagons}, there is a unique polarity of $\Delta$ stabilizing
$c$ and $\Sigma$ and interchanging
$\dot x_\alpha(t)$ and $\dot x_{\kappa(\alpha)}(t)$ for all $\alpha\in|\Phi|$
and all $t\in E$. We denote this polarity by $\sigma$.
\end{notation}

\begin{notation}\label{onk48}
Let $[abcd]$ denote the reflection associated with the vector $abcd$
for all $abcd\in|\Phi|$. Let $r_1=[011'1']$ and $r_4=[1'1'10]$ and let $R$,
$R_1$, $R_4$, $R_1'$ and $R_4'$ be as in \ref{onk49}.
Then $|r_1r_4|=4$. The reflection $r_1$ stabilizes
$R_1\cap\Sigma$ and interchanges $R\cap\Sigma$ with $R_1'\cap\Sigma$ as
well as $W_2$ and $W_4$. The reflection $r_4$ stabilizes
$R_4\cap\Sigma$ and interchanges $R\cap\Sigma$ with $R_4'\cap\Sigma$ as well as
$W_1$ and $W_3$. In particular, $r_1$ induces the reflection on the graph
$\Lambda$ defined in \ref{onk101} that interchanges $R\cap\Sigma$ and $R'_1\cap\Sigma$
and $r_4$ induces the reflection that interchanges $R\cap\Sigma$ and $R'_4\cap\Sigma$.
\end{notation}

\begin{notation}\label{onk47}
We denote by $r$ the square of the product
$$[0100]\cdot[0010].$$
The element $r$ is an involution commuting with $\kappa$ and with $r_1$ and $r_4$.
It stabilizes the residue $R$ and hence acts trivially on the graph $\Lambda$.
It stabilizes the four sets $W_1,\ldots,W_4$ and
fixes the vectors $011'1'\in W_1$, $1'2'32\in W_2$, $232'1'\in W_3$ and $1'1'10\in W_4$,
but does not fix any other elements of $W_1\cup W_2\cup W_3\cup W_4$.
\end{notation}

By \ref{wax34}(i), there exists a unique automorphism $\zeta$
of $\Delta$ stabilizing $\Sigma$ such that
\begin{equation}\label{onk45}
\dot x_v(t)^\zeta=\dot x_{r(v)}(t)
\end{equation}
for all $t\in E$.

We set
$$\lambda^{m'}=\lambda^{m\theta}$$
for all $\lambda\in K$ and all $m\in{\mathbb N}$ and let
$$h_{\lambda_1,\lambda_2,\lambda_3,\lambda_4}=
g_{1,\lambda_1,\lambda_2,\lambda_3^{\theta^{-1}},\lambda_4^{\theta^{-1}},1}$$
for all $\lambda_1,\ldots,\lambda_4\in E^*$.
Let $h=h_{\lambda_1,\ldots,\lambda_4}$ for some choice of $\lambda_1,\ldots,\lambda_4\in E^*$.
By \ref{wax34}(ii), we have
\begin{equation}\label{nzz70}
\dot x_{abcd}(t)^h=\dot x_{abcd}(\lambda t)
\end{equation}
for all $abcd\in|\Phi|$ and all $t\in E$, where
$$\lambda=\lambda_1^a\lambda_2^b\lambda_3^c\lambda_4^d.$$
Thus, for example,
$$\dot x_{1'2'21}(t)^h=\dot x_{1'2'21}(\lambda t)$$
for all $t\in L$, where
$$\lambda=\lambda_1^\theta\lambda_2^{2\theta}\lambda_3^2\lambda_4.$$

\begin{notation}\label{onk51}
We set
$$\xi=g_{w_1,\beta^{-(\theta+1)},\beta^\theta,\beta,\beta^{-(\theta+1)\theta^{-1}},\chi},$$
where $w_1$ is as in \ref{onk10}. By \ref{tru68}(ii) and \ref{onk43}, we have
$\alpha=\beta^{-\theta}$; thus $\xi$ is the same as the element $\xi$ in \ref{onk12}.
Note that
$$\dot x_{abcd}(t)^\xi=\big(\dot x_{abcd}(t^\chi)^h\big)^\zeta$$
for all $abcd\in|\Phi|$, where $\zeta$ is as in \eqref{onk45} and
$$h=h_{\beta^{-(\theta+1)},\beta^\theta,\beta^\theta,\beta^{-(\theta+1)}}.$$
\end{notation}

\begin{notation}\label{onk52}
By \ref{onk13}, we already know that the automorphism $\xi$ is a type-preserving $\chi$-involution
of $\Delta$ and that $\tilde\Xi:=\Delta^{\langle\xi\rangle}$ is isomorphic to $\Xi$.
The polarity $\sigma$ defined in \ref{onk40} commutes with $\xi$
and thus induces a polarity of $\tilde\Xi$ which we denote by $\tilde\rho$. Our goal in \ref{onk50}
is to show that there is an isomorphism from $\tilde\Xi$ to $\Xi$ which carries
$\tilde\rho$ to $\rho$.
\end{notation}

\begin{remark}\label{onk104}
Since (by \ref{onk13}) the minimal residues stabilized by $\xi$
are of type $\{2,3\}$,
the residue $R$ in \ref{onk49} is a chamber of $\tilde\Xi$. Since $\xi$ stabilizes $\Sigma$,
the graph $\tilde\Sigma$ defined in \ref{onk102} is an apartment of $\tilde\Xi$
containing $R$. The polarity $\tilde\rho$ stabilizes both $R$ and $\tilde\Sigma$.
\end{remark}

\begin{remark}\label{onk103}
It might appear that we are giving a new proof of \ref{onk13} in \ref{onk50}. In fact, however,
the proof of \ref{onk50} we give relies on \ref{onk104} which, in turn, relies on the fact that the
the minimal residues stabilized by $\xi$ are of type $\{2,3\}$.
It is exactly in the proof of this fact that the proof of \ref{onk13}
in \cite{galois} differs from the proof in \cite{canada}.
\end{remark}

\begin{proposition}\label{onk50}
Let $\Xi$, $\Sigma$, $c$ and $\rho$ be as {\rm\ref{tru90}} and
let $\tilde\Xi$, $\tilde\Sigma$, $R$ and $\tilde\rho$ be as in
{\rm\ref{onk49}}, {\rm\ref{onk52}} and {\rm\ref{onk104}}.
Then there is an
isomorphism from $\tilde\Xi$ to $\Xi$ mapping the pair $(\tilde\Sigma,R)$ to the pair
$(\Sigma,c)$
that carries the polarity $\tilde\rho$ to $\rho$.
\end{proposition}

\begin{proof}
We define maps $X_1,\ldots,X_4$ from
$V=E\oplus E\oplus[K]$ to ${\rm Aut}(\Delta)$
as follows:
\begin{align*}
X_1(u,v,t)&=\dot x_{0011}(u)\dot x_{01'11}(\beta^{-1}\bar u)
\cdot\dot x_{0001}(v)\dot x_{01'21}(\beta^{-(\theta+1)}\bar v)
\cdot\dot x_{011'1'}(t)\\
X_2(u,v,t)&=\dot x_{121'1'}(u)\dot x_{122'1'}(\beta^{-1}\bar u)
\cdot\dot x_{111'1'}(v)\dot x_{132'1'}(\beta^{-(\theta+1)}\bar v)
\cdot\dot x_{1'2'32}(t)\\
X_3(u,v,t)&=\dot x_{1'1'21}(u)\dot x_{1'2'21}(\beta^{-1}\bar u)
\cdot\dot x_{1'1'11}(v)\dot x_{1'2'31}(\beta^{-(\theta+1)}\bar v)
\cdot\dot x_{232'1'}(t)\\
X_4(u,v,t)&=\dot x_{1100}(u)\dot x_{111'0}(\beta^{-1}\bar u)
\cdot\dot x_{1000}(v)\dot x_{121'0}(\beta^{-(\theta+1)}\bar v)
\cdot\dot x_{1'1'10}(t)
\end{align*}
for all $(u,v,t)\in V$, where $\bar x=x^\chi$ for all $x\in E$. Note that
\begin{equation}\label{nzz89}
X_i(u,v,t)^{\tilde\rho}=X_{5-i}(u,v,t)
\end{equation}
for all $(u,v,t)\in V$.
Let $M_i=X_i(V)$ for all $i\in[1,4]$,
let $M_+$ denote the subgroup generated by $M_1,\ldots,M_4$ and let
$$\tilde\Psi:=(M_+,M_1,M_2,M_3,M_4).$$
We have $M_i=C_{\langle U_\alpha\mid\alpha\in W_i\rangle}(\xi)$ for each $i\in[1,4]$.
By \ref{onk49}, \ref{onk104} and \cite[24.32]{MPW}, $M_1$ and $M_4$ are root groups of $\Delta$
corresponding to the two roots of the apartment $\tilde\Sigma$ containing $R$
but not some chamber of $\tilde\Sigma$ adjacent to $R$. By \ref{onk48}, we
conclude that $\tilde\Psi$ is a root group sequence of the Moufang quadrangle $\tilde\Xi$.

It follows from \ref{nzz59}(i)--(iii) that the map $X_i$ is additive
and thus $M_i$ is abelian for all $i\in[1,4]$, that
$$[M_1,M_2]=[M_2,M_3]=[M_3,M_4]=[M_1,M_3]=1,$$
and that
\begin{align*}
[X_2(a,b,s),X_4(u,v,t)]&=X_3(0,0,\beta^{-1}(u\bar a+a\bar u)+\beta^{-(\theta+1)}(v\bar b+b\bar v))\\
&=X_3(0,0,f((a,b,s),(u,v,t))
\end{align*}
for all $(a,b,s),(u,v,t)\in V$, where $f=\partial q$.
Applying also the identities \eqref{nzz90}, we find that
\begin{align*}
[X_1(a,0,0),X_4(u,0,0)]&=X_2(0,a^\theta u,0)X_3(0,u^\theta a,0)\\
[X_1(a,0,0),X_4(0,v,0)]&=X_2(\beta^{-\theta}\bar a^\theta v,0,0)X_3(0,\beta^{-1}v^\theta\bar a,0)\\
[X_1(a,0,0),X_4(0,0,s)]&=X_2(0,0,s\beta^{-1}N(a))X_3(sa,0,0)\\
[X_1(0,b,0),X_4(0,v,0)]&=X_2(\beta^{-(\theta+1)}b^\theta\bar v,0,0)
X_3(\beta^{-(\theta+1)}v^\theta\bar v,0,0)\\
[X_1(0,b,0),X_4(0,0,s)]&=X_2(0,0,s\beta^{-(\theta+1)}N(b))X_3(0,sb,0)\\
[X_1(0,0,r),X_4(0,0,s)]&=X_2(0,0,r^\theta s)X_3(0,0,s^\theta r)
\end{align*}
for all $(a,b,r),(u,v,s)\in V$. The commutator of an arbitrary element of $M_i$ with
an arbitrary element of $M_j$ for $(i,j)=(1,3)$ and $(1,4)$ is uniquely determined by
the identities \eqref{nzz90}, \eqref{nzz89} and the identities above.
It follows that there is an isomorphism $\omega$ from $\tilde\Psi$ to
the root group sequence $\Psi:=(U_+,U_1,\ldots,U_4)$ defined by the commutator
relations in \ref{tru90}
sending $X_i(u,v,t)$ to $x_i(u,v,t)$
for all $(u,v,t)\in V$ and all $i\in[1,4]$. By \eqref{tru90x} and \eqref{nzz89},
$\omega$ carries the anti-automorphism
of $\tilde\Psi$ induced by $\tilde\rho$ to the anti-automorphism of $\Psi$
induced by $\rho$. By \cite[7.5]{TW}, there exists a unique isomorphism $\omega_1$ from
$\tilde\Xi$ to $\Xi$ mapping the the pair $(\tilde\Sigma,R)$ to the pair $(\Sigma,c)$
and inducing the map
$\omega$ from $\tilde\Psi$ to $\Psi$. By \cite[3.7]{TW}, $\omega_1$ carries
$\tilde\rho$ to $\rho$.
\end{proof}

\section{Moufang Octagons}\label{onk15}

Let $\Omega$ be a Moufang octagon. By \cite[17.7]{TW}, $\Omega={\mathcal O}(E,\theta)$ for
some octagonal set $(E,\theta)$ as defined in \ref{abc0}. Let $\Delta=
F_4(E,\theta)=F_4(E,E^\theta)$ as defined in \ref{abc4}.

\begin{notation}\label{onk56}
Let $L/E$ and the extension of $\theta$ to $L$ be as in \ref{onk43}. Then $\theta$ maps the pair
$(L,E)$ to the pair $(E,E^\theta)$ and hence induces an isomorphism from $F_4(L,E)$ to
$\Delta$. Let $c$, $\Sigma$, $\Phi$ and the identification
of $\Phi$ with the set of roots of $\Sigma$ be as in \ref{baf103}, let
$\{x_\alpha\}_{\alpha\in\Phi}$ be as in \ref{wax31} and let $\sigma$
be the polarity of $F_4(L,E)$ defined in \ref{onk40}. We identify $F_4(L,E)$ with
$\Delta$ via the isomorphism induced by $\theta$.
\end{notation}

\begin{proposition}\label{onk55}
$\Omega\cong\Delta^{\langle\sigma\rangle}$.
\end{proposition}

\begin{proof}
This holds by \cite[Theorem (on p.~540)]{octagons}.
\end{proof}

For the rest of this section, we will simply identify $\Omega$ with $\Delta^{\langle\sigma\rangle}$.

\begin{proposition}\label{onk30}
Every automorphism of $\Omega$ has a unique extension to a type-preserving
automorphism of $\Delta$, and all of these extensions commute with $\sigma$.
\end{proposition}

\begin{proof}
Let $G^\circ$ denote the group of type-preserving automorphisms of $\Delta$ (as in \ref{goo3}
applied to $\Delta$).
By \cite[1.6 and 1.13.1(ii)]{octagons}, every automorphism of $\Omega$
can be extended to an element in the centralizer $C_{G^\circ}(\sigma)$.
Suppose that $g$ is a type-preserving automorphism of $\Delta$ that acts trivially on $\Omega$.
It remains only to show that $g$ is trivial.
Opposite chambers of $\Omega$ are opposite chambers of $\Delta$ and opposite
chambers of $\Delta$ are contained in a unique apartment of $\Delta$. We can
thus assume that $g$ fixes the apartment $\Sigma$ and chamber $c$ in \ref{onk56}.
Since $g$ is type-preserving, it acts trivially on $\Sigma$ and hence normalizes the root group
$U_\alpha$ for all $\alpha\in\Phi$.
By \cite[1.5]{octagons},
the map from the additive group of $E$ to ${\rm Aut}(\Omega)$ which sends $t\in E$
to the element of ${\rm Aut}(\Omega)$ induced by $x_{\alpha_1}(t^\theta)x_{\alpha_4}(t)$
is injective as is the map from the additive group of $E$ to ${\rm Aut}(\Omega)$ which sends $t\in E$
to the element of ${\rm Aut}(\Omega)$ induced by
$x_{\alpha_2}(t^\theta)x_{\alpha_3}(t)x_{\alpha_2+2\alpha_3}(t^{\theta+2})$.
It follows that $g$ centralizes $U_{\alpha_i}$ for each $i\in[1,4]$. By \ref{wax34}(i),
therefore, $g=1$.
\end{proof}

By \cite[28.8]{MPW}, $E$ is the defining field of $\Omega$ and $\{E/E^\theta,E^\theta/E\}$ is the
pair of defining extensions of $\Delta$.

\begin{proposition}\label{onk30x}
Let $A={\rm Aut}(E,E^\theta)$ be as in {\rm\ref{onk2}}, let $\iota$ denote
the inclusion map from $A$ to ${\rm Aut}(E)$ and let $\psi_\Delta$
denote the Galois map of $\Delta$ in {\rm\ref{onk110}}.
Then there is a unique Galois map $\psi_\Omega$ of $\Omega$
such that
\begin{equation}\label{onk30p}
\psi_\Omega(\kappa)=\iota(\psi_\Delta(\zeta))
\end{equation}
for all pairs $(\kappa,\zeta)$, where
$\kappa\in{\rm Aut}(\Omega)$ and $\zeta$ is the unique extension of $\kappa$
to a type-preserving automorphism of $\Delta$.
\end{proposition}

\begin{proof}
Let $G^\circ$ and $G^\dagger$ be as in \ref{goo3} applied to $\Delta$
and let $H={\rm Aut}(\Omega)$.
By \ref{onk30}, there is a unique homomorphism $\psi=\psi_\Omega$ from $H$ to ${\rm Aut}(E)$
such that \eqref{onk30p} holds. Let $\kappa\in{\rm Aut}(\Omega)$ and let $\zeta$ be its
unique extension to an element of $G^\circ$. If $\kappa$ lies in a root group
of $\Omega$, then by \cite[24.32]{MPW}, $\zeta\in G^\dagger$ and hence $\psi(\kappa)=1$.
Thus $\psi$ satisfies \ref{goo5}(i). Let $d$ be the chamber opposite $c$ in $\Sigma$.
Then $d$ is a chamber of $\Omega$ opposite $c$. Hence there exists a unique
apartment $\Sigma_\Omega$ containing $c$ and $d$. Suppose that
$\kappa$ acts trivially on $\Sigma_\Omega$. Then $\zeta$ acts trivially on $\Sigma$ since
it is type-preserving. By \ref{wax34x}, therefore,
$$\zeta=g_{\gamma,\lambda_1,\ldots,\lambda_4,\chi}$$
for some $\gamma\in{\rm Aut}(\Phi)$, some $\lambda_1,\lambda_2\in E^\theta$,
some $\lambda_3,\lambda_4\in E$
and some $\chi\in{\rm Aut}(E,E^\theta)$. By \ref{onk110}, $\psi_\Delta(\zeta)=\chi$.
Let $B=B_\Pi$ be as in \ref{onk1}
for $\Pi=I_2(8)$, let $(s,t)$ be the standard element of $B$ as defined in \ref{onk3} and
let $\Theta=(\hat U_+,\hat U_1,\ldots,\hat U_8)$ be the root group sequence and
$\hat x_1,\ldots,\hat x_8$ the isomorphisms obtained by applying \cite[16.9]{TW} to
the octagonal set $(E,\theta)$. By \cite[1.5--1.7]{octagons}, there exists
an isomorphism $\varphi\colon\Omega_{st}\to\Theta$ such that
there exist $\delta_1,\ldots,\delta_8\in E^*$ so that
for each $i\in[1,8]$,
$\hat x_i(u)^h=\hat x_i(\delta_iu^\chi)$ for all $u\in E$ if $i$ is odd and
$\hat x_i(u,v)^h=\hat x_i(\delta_iu^\chi,\lambda_i^{\theta+1}v^\chi)$
for all $u,v\in E$ if $i$ is even,
where $h:=\varphi^{-1}\kappa\varphi$. Thus $\chi=\psi_\Omega(\kappa)$ equals the element
called $\lambda_\Omega(h)$ in \cite[29.5]{MPW} with $\Omega=\Theta$. By
\ref{goo5}, $\psi_\Omega$ is the unique Galois map of $\Omega$ determined by $\varphi$.
\end{proof}

\begin{remark}\label{onk30a}
Let $\kappa\in{\rm Aut}(\Omega)$ and let
$\zeta$ be the unique extension $\kappa$ to a type-preserving automorphism of $\Delta$.
By \ref{onk30}, $\kappa$ is an involution if and only if $\zeta$ is.
If we choose Galois maps as in \ref{onk30x}, it follows that $\zeta$ is, in fact,
a $\chi$-involution for some $\chi\in A$ if and only if $\kappa$ is a
$\iota(\chi)$-involution, where $A$ and $\iota$ are as in \ref{onk30x}; see \ref{nzz10}.
\end{remark}

\begin{proposition}\label{onk31}
Let $\kappa$ be a Galois involution of $\Omega$ that fixes panels of one
type but none of the other type. Then $\kappa$ has a unique extension to
a type-preserving Galois involution $\zeta$ of $\Delta$, $\langle\zeta\rangle$
is a descent group of $\Delta$ and $\langle\zeta\rangle$-chambers are
residues of type $\{2,3\}$.
\end{proposition}

\begin{proof}
By \ref{onk30a}, $\kappa$ has a unique extension to a type-preserving Galois involution $\zeta$
of $\Delta$. By \ref{toy1}, $\langle\zeta\rangle$ is a descent group of $\Delta$.
By \ref{onk55}, some panels of $\Omega$ are $\{1,4\}$-residues of $\Delta$ and
the others are $\{2,3\}$-residues (with respect to the standard numbering of the
vertex set of the Coxeter diagram $F_4$). Since $\kappa$ fixes panels of $\Omega$,
we can choose a $J$-residue $R$ of $\Delta$ stabilized by $\langle\kappa,\sigma\rangle$,
where $J$ is either $\{1,4\}$ or $\{2,3\}$. Let $R_1$ be a minimal
$\langle\zeta\rangle$-residue contained in $R$ and let $J_1$ be its type.
Since $\zeta$ commutes with $\sigma$, $R_1\cap R_1^\sigma$ is also stabilized by $\zeta$.
By the choice of $R_1$, it follows that $R_1$ is stabilized by $\sigma$.
Thus $J_1$ is a
subset of $J$ invariant under the non-trivial automorphism of the Coxeter diagram of $\Delta$,
so either $J_1=\emptyset$ or $J_1=J$.
Suppose that $J_1=\emptyset$. Then $R_1$ is contained in a unique $J'$-residue $R_2$,
where $J'$ denotes the complement of $J$ in the vertex set of the Coxeter diagram $F_4$.
Since $\langle\sigma,\zeta\rangle$ stabilizes $R_1$, it must stabilize $R_2$ as well.
This contradicts the assumption, however, that $\kappa$ does not fix panels of
$\Omega$ of both types. We conclude that $J_1=J$
and hence $R_1=R$. Thus $R$ is a $\langle\zeta\rangle$-chamber.

Let $\Pi$ be the Coxeter diagram of type $F_4$ and let $\Theta$ denote
the trivial subgroup of ${\rm Aut}(\Pi)$. By \ref{nzz25} and \ref{nzz24}, the triple
$(\Pi,\Theta,\{1,4\})$ is not a Tits index.
By \ref{nzz8}(iii), we conclude that $J=\{2,3\}$.
\end{proof}

\begin{proposition}\label{onk32}
Let $\chi=\psi_\Delta(\xi)$, where $\psi_\Delta$ is as in {\rm\ref{onk30x}} and
$\xi$ is as in {\rm\ref{onk31}}. Then the following hold:
\begin{enumerate}[\rm(i)]
\item $\chi\theta=\theta\chi$ and $\chi\theta$ is a Tits endomorphism of $E$.
\item $\Omega_\xi:=\Delta^{\langle\xi\sigma\rangle}$ is isomorphic to ${\mathcal O}(E,\chi\theta)$.
\item {\rm\ref{onk30}} holds with $\Omega_\xi$ and $\xi\sigma$ in place of $\Omega$ and $\sigma$.
\item There exists a Galois map $\psi_{\Omega_\xi}$ of $\Omega_\xi$ such that
\eqref{onk30p} holds with $\psi_{\Omega_\xi}$ in place of $\psi_\Omega$.
\end{enumerate}
\end{proposition}

\begin{proof}
By \ref{onk13x}(i) and the choice of $\psi_\Delta$ in \ref{onk30x}, we can assume that
$$\xi=g_{w_1,\lambda_1,\ldots,\lambda_4,\chi}$$
for some $\lambda_1,\lambda_2\in E^\theta$ and some $\lambda_3,\lambda_4\in E$. For each
$\alpha\in\Phi$, we denote by $s_\alpha$ the corresponding reflection of $\Phi$ (as in \ref{baf103}).
Let $s=s_{\alpha_3}s_{\alpha_2+2\alpha_3}$, let $\beta_i=\alpha_i^s$
and let $d=c^s$ with respect to the action of $W$ on
$\Sigma$ described in \ref{baf103}. Then $\beta_1,\ldots,\beta_4$ is a basis of $\Phi$
and $\beta_1=\alpha_1+\alpha_2+2\alpha_3$, $\beta_2=-\alpha_2-2\alpha_3$, $\beta_3=\alpha_2+\alpha_3$
and $\beta_4=\alpha_3+\alpha_4$. The restriction of $\sigma\xi$ to $\Sigma$ induces
the the unique automorphism of $\Phi$ that interchanges $\beta_i$ and $\beta_{5-i}$ for all
$i\in[1,4]$
(via the identification of the roots of $\Sigma$ with $\Phi$ in \ref{baf103}).
By \ref{wax34}(ii), there exist non-zero $\varepsilon_3,\varepsilon_3',\varepsilon_4\in E^\theta$ such
that
\begin{equation}\label{onk32x}
x_{\beta_i}(t)^{\xi\sigma}=x_{\beta_{5-i}}(\varepsilon_it^{\chi\theta})
\end{equation}
for $i=3$ and $4$ and all $t\in E^\theta$ and
$$x_{\beta_3}(t)^{\sigma\xi}=x_{\beta_2}(\varepsilon'_3t^{\theta\chi})$$
for all $t\in E^\theta$.
Since $\sigma$ and $\xi$ commute, we conclude that $\varepsilon_3=\varepsilon_3'$ and
$\theta\chi=\chi\theta$. Thus (i) holds.

Let $h=g_{1,\varepsilon_3\varepsilon_4^{-1},\varepsilon_3,\varepsilon_3^{-1},\varepsilon_3,1}$.
By \ref{wax34}(ii) again,
$h$ centralizes $U_{\beta_3}$ and $U_{\beta_4}$ and
$x_{\beta_1}(t)^h=x_{\beta_1}(\varepsilon_4^{-1}t)$ and
$x_{\beta_2}(t)^h=x_{\beta_2}(\varepsilon_3^{-1}t)$
for all $t\in F$. Let $\ddot x_\alpha=x_\alpha\cdot h_\alpha$ for all $\alpha\in\Phi$, where
$h_\alpha$ denotes the automorphism $a\mapsto a^h$ of $U_\alpha$. Then
$\{\ddot x_\alpha\}_{\alpha\in\Phi}$ is a coordinate system for $\Delta$ (as defined in \ref{baf901}),
$\ddot x_{\beta_i}=x_{\beta_i}$ for $i=3$ and $4$ and
$\ddot x_{\beta_1}(\varepsilon_4 t)=x_{\beta_1}(t)$ and
$\ddot x_{\beta_2}(\varepsilon_3t)=x_{\beta_2}(t)$ for all $t\in E$. By \eqref{onk32x},
therefore,
$\ddot x_{\beta_i}(t)^{\sigma\xi}=\ddot x_{\beta_{5-i}}(t^{\chi\theta})$ for all $i\in[1,4]$
and for all $t\in E$.
We can thus apply \ref{onk55}--\ref{onk30x} with $\xi\sigma$
and $\{\ddot x_\alpha\}_{\alpha\in\Phi}$ in place of $\sigma$ and $\{x_\alpha\}_{\alpha\in\Phi}$
to conclude that (ii)--(iv) hold.
\end{proof}

\section{Proofs of \ref{abc1} and \ref{abc2}}\label{nzz40}

We first prove \ref{abc1}.
Suppose that $\Xi$ and $\rho$ satisfy the hypotheses, let
$$S=(E/K,F,\alpha,\beta)$$
and $\theta$ be as in \ref{tru68} and let $\Delta=F_4(L,E)$ and $\chi$ be as in \ref{onk43}.
The Tits endomorphism $\theta$ commutes with $\chi$; it also
maps the pair $(L,E)$ to the pair $(E,E^\theta)$ and hence
induces an isomorphism from $\Delta$ to $F_4(E,\theta)$.
Let $\sigma$ be as in \ref{onk40}
and let $\xi$ be as in \ref{onk51}. By \ref{onk52}, $[\sigma,\xi]=1$, $\xi$ is a type-preserving
$\chi$-involution of $\Delta$ and
$\sigma$ induces a polarity on $\tilde\Xi:=\Delta^{\langle\xi\rangle}$.
By \ref{nzz5}, the restriction of $\langle\sigma\rangle$ to $\tilde\Xi$ is
a descent group of relative rank~$1$. It follows that $\langle\xi,\sigma\rangle$ is a descent group
of $\Delta$. Thus (i) holds. By \ref{onk50}, (ii) holds. By \ref{nzz5} again,
$\Delta^{\langle\sigma\rangle}$ and $\Delta^{\langle\sigma\xi\rangle}$ are Moufang
octagons. By \ref{onk56} and \ref{onk55}, the first of these octagons is isomorphic
to ${\mathcal O}(E,\theta)$ and by \ref{onk32}(ii),
the second is isomorphic to ${\mathcal O}(E,\chi\theta)$. Thus (iii) holds.

By \ref{nzz5}, $\Delta^\Gamma$, $(\Delta^{\langle\xi\rangle})^{\langle\sigma\rangle}$,
$(\Delta^{\langle\sigma\rangle})^{\langle\xi\rangle}$
and $(\Delta^{\langle\sigma\xi\rangle})^{\langle\xi\rangle}$ are all Moufang
sets. The underlying set of each of them is the set $X$ of all $\Gamma$-chambers
and by \ref{nzz8}(v) and \cite[24.32]{MPW},
the root group corresponding to a $\Gamma$-chamber $R$ is the permutation group
induced by $C_{\Gamma}(U_R)$ on $X$, where $U_R$ is the unipotent radical of $R$ in
$\Delta$. Thus (iv) holds.
By \ref{nzz8}(ii) and \ref{onk13}, there are $\{2,3\}$-residues of $\Delta$ stabilized by $\xi$ but
none of type $\{1,4\}$. By \ref{onk30x} and \ref{onk32}(iv), therefore, (v) holds.
This concludes the proof of \ref{abc1}.

We turn now to \ref{abc2}.
Suppose that $\chi$, $(E,\theta)$, $\Delta$, $\Omega$ and $\kappa$ satisfy
the hypotheses, let $\sigma$ be as in \ref{onk56} and let $\xi$
be the type-preserving automorphism of $\Delta$ obtained by applying \ref{onk30}
to $\kappa$. Then $\xi$ and $\sigma$ commute and
by \ref{onk30a}, $\xi$ is a $\chi$-involution. Let $\Gamma=\langle\xi,\sigma\rangle$.
Then $\Delta^\Gamma=\Omega^{\langle\kappa\rangle}$. By \ref{nzz37}, it follows that
$\Gamma$ is a descent group of $\Delta$. Thus (i) holds.
Assertion~(ii) holds by \ref{onk55} and the choice of $\xi$. By \ref{onk31},
$\langle\xi\rangle$-chambers are of type $\{2,3\}$ and by \ref{nzz22},
$\Xi:=\Delta^{\langle\sigma\rangle}$ is a Moufang quadrangle of type $F_4$.
Since $\xi$ and $\sigma$ commute, $\sigma$ induces a polarity on $\Xi$. Thus (iii) holds.
Assertion~(iv) holds for the same reason that \ref{abc1}(iv) holds and assertion~(v)
holds by \ref{onk32}. This concludes the proof of \ref{abc2}.

\section{Moufang Sets of Outer $F_4$-Type}\label{onk60a}

Our goal in the remaining sections is to determine a few essential properties
of the Moufang sets of outer $F_4$-type defined in \ref{nzz29}.

\begin{notation}\label{onk60}
Let $S$, $V=E\oplus E\oplus[K]$, $U_+$, $x_1,\ldots,x_4$ and $\theta$ be as in \ref{tru90},
let $\theta_K$ be the restriction of $\theta$ to $K$,
let $\Xi={\mathcal Q}(S)$, $\Sigma$ and $c$ be as in \ref{abc61},
let $q=q_S$, let $f=\partial q$, let $g$ be as in \ref{nzz34},
let $\rho$ be as in \eqref{tru90x} and let
$M=(X,\{U_x\}_{x\in X})$ be the Moufang set $\Xi^{\langle\rho\rangle}$
obtained by applying \ref{nzz8}(v) with $\langle\rho\rangle$ in
place of $\Gamma$. Note that $c\in X$.
\end{notation}

\begin{notation}\label{onk69}
We have $c\in X$ and by \ref{nzz8}(v),
the centralizer $C_{U_+}(\rho)$ equals the root group $U_c$ of the Moufang set $M$.
We set $U=U_c$ and write $U$ additively even though it
is not, as we will see, abelian. In this section we use \eqref{abc20} and \ref{abc51}
to compute a few basic properties of $U$.
\end{notation}

Let $\eta\in U_+$. Thus
$$\eta=x_1(b)x_2(w)x_3(v)x_4(u)$$
for some $b,w,u,v\in V$. Note that $f(a,g(y,z))=0$ for all $a,y,z\in V$
by \ref{abc25} and \ref{nzz34} (and, of course, that the characteristic of $K$ is~$2$). Thus
\begin{align*}
\eta^\rho &= x_4(b)x_3(w)x_2(v)x_1(u)\\
&= x_4(b)x_1(u)x_3(w)x_2(v+g(u,w))\\
&=x_1(u)x_2(bu)x_3(ub)x_4(b)x_3(w)x_2(v+g(u,w))\\
&=x_1(u)x_2(bu)x_3(ub)x_3(w)x_3(g(b,v))x_2(v+g(u,w))x_4(b)\\
&=x_1(u)x_2(bu+v+g(u,w))x_3(ub+w+g(b,v))x_4(b),
\end{align*}
so $\eta\in U$ if and only if $b=u$, $w=bu+v+g(u,w)$ and $v=ub+w+g(b,v)$.
Note that $g(u,w)=g(u,uu)+g(u,v)=g(u,v)$
by \ref{abc24}(i). It follows that
\begin{equation}\label{eq U with uw}
U:=\{x_1(u)x_2(w)x_3(uu+w+g(u,w))x_4(u)\mid u,w\in V\}.
\end{equation}
Let
\begin{equation}\label{nzz32}
\lbrace u,w\rbrace=x_1(u)x_2(w)x_3(uu+w+g(u,w))x_4(u)
\end{equation}
for all $u,w\in V$. Then
\begin{align*}
\lbrace u,w\rbrace+\lbrace a,b\rbrace&=x_1(u)x_2(w)x_3(uu+w+g(u,w))x_4(u)\\
&\qquad\qquad\cdot x_1(a)x_2(b)x_3(aa+b+g(a,b))x_4(a)\\
&\in x_1(u+a)x_2\big(w+b+ua+g(a,uu+w+g(u,w))\big)U_{[3,4]}
\end{align*}
and thus
\begin{equation}\label{tru64}
\lbrace u,w\rbrace+\lbrace a,b\rbrace=\lbrace u+a, w+b+ua+g(a,w)+g(a,uu)\rbrace
\end{equation}
for all $u,w,a,b\in V$. It follows that
$$-\lbrace u,w\rbrace=\lbrace u,w+uu+g(u,w)\rbrace$$
and the commutator $-\lbrace u,w\rbrace-\lbrace a,b\rbrace+\lbrace u,w\rbrace+\lbrace a,b\rbrace$ equals
\begin{equation}\label{eq commutator}
\lbrace0,\ ua+au+g(u,b)+g(a,w)+g(u,aa)+g(a,uu)\rbrace
\end{equation}
for all $u,w,a,b\in V$.

\begin{proposition}\label{tru61}
$U'=\lbrace0,V\rbrace$ and $[U,U']=Z(U)=\lbrace0,[K]\rbrace$.
In particular, $U$ is nilpotent and has nilpotency class~$3$.
\end{proposition}

\begin{proof}
Setting $a=[1]$ in \eqref{eq commutator}, we obtain
$$\lbrace0,u+[q(u)+f(u,b)]\rbrace\in U'$$
for all $u,b\in V$.
For each $u\in V\backslash[K]$, there exists $b$ such that $q(u)=f(u,b)$.
Therefore $\lbrace0,u\rbrace$ is in the commutator group $U'$ of $U$
for all $u\in V\backslash [K]$. Hence $U'=\lbrace0, V\rbrace$.
If we set $u=0$ in \eqref{eq commutator}, we are left with only $\lbrace0,g(a,w)\rbrace$.
It follows that $[U',U]=\lbrace0,[K]\rbrace\subset Z(U)$,
$Z(U)\subset\lbrace[K],V\rbrace$ and
\begin{equation}\label{tru60}
Z(U)\cap\lbrace0,V\rbrace=\lbrace0,[K]\rbrace.
\end{equation}
Let $t\in K^*$. By \ref{abc70}, we can choose $s\in K$ with $s \neq 0$ and $s \neq t$.
Since $(x^{\theta-1})^{\theta+1}=x$ for all $x\in K^*$, it follows that
$s^{\theta-1}\ne t^{\theta-1}$ and hence $s^\theta t+st^\theta\ne0$.
Setting $u=[s]$ and $a=[t]$ in \eqref{eq commutator}, we obtain
$\lbrace0,[s^\theta t+st^\theta]\rbrace$.
It follows that $Z(U)\subset\lbrace0,V\rbrace$.
By \eqref{tru60}, we conclude that $Z(U)=\lbrace0,[K]\rbrace$.
\end{proof}

\section{The Element $\tau$}\label{onk81}

Let $\Xi$, $\Sigma$, $c$,
$U_+$, $x_1,\ldots,x_4$, $\rho$, $X$, etc., be as in \ref{onk60} and let $\phi$ and
$\Omega={\mathcal G}(\Theta,U_+,\phi)$ be as in \cite[7.2]{TW} with $n=4$, where
$\Theta$ is a circuit of length~8 whose vertex set $V(\Theta)$ has been
numbered by the integers modulo~8 so that the vertex $x$ is adjacent to the vertex $x-1$
for all $x$. The vertex set of $\Omega$ consists of pairs
$(x,B)$, where $x\in V(\Theta)$ and $B$ is a right coset in $U_+$ of the subgroup $\phi(x)$.
The vertices $(x,\phi(x))$ span an apartment of $\Omega$ which we identify with $\Theta$
via the map $x\mapsto(x,\phi(x))$.
We set $\bullet=(4,\phi(4))$ and $\star=(5,\phi(5))$.
Thus $e:=\{\bullet,\star\}$ is an edge of $\Omega$.
For all vertices $(x,B)$ of $\Omega$ other than $\bullet$ and $\star$, the vertex $x$
of $\Theta$ is uniquely determined by $B$ and we can denote the vertex $(x,B)$
simply by $B$. The elements of $U_+$ fix $\bullet$ and $\star$ and
acts on all other vertices by right multiplication.

By \cite[8.11]{TW}, $\Omega$ is a Moufang quadrangle. We identify $\Xi$
with the corresponding bipartite graph as described in \cite[1.8]{spherical} and
let $\pi$ be an isomorphism $\Sigma$ to $\Theta$
mapping the chamber $c$ to the edge $\{\bullet,\star\}$. By \cite[7.5]{TW}, $\pi$ extends to
a unique $U_+$-equivariant isomorphism from $\Xi$ to $\Omega$.
We identify $\Xi$ with $\Omega$ via this extension, so that $\Sigma=\Theta$, $c=e$ and
the polarity $\rho$ is an element of ${\rm Aut}(\Xi)$
stabilizing $\Sigma$ and interchanging the vertices $\bullet$ and $\star$.
In particular, $c$ and $d$ are in $X$, where $d=\{U_1,U_4\}$ is the chamber of $\Sigma$
opposite $e=\{\bullet,\star\}$.
Let $U=U_e=U_c$ be as in \ref{onk69}.

Let $m_1=\mu_\Sigma(x_1(0,0,1))$ and $m_4=\mu_\Sigma(x_4(0,0,1))$ be as in \cite[11.22]{spherical}.
By \eqref{tru90x}, conjugation by the polarity
$\rho$ interchanges $x_1(0,0,1)$ and $x_4(0,0,1)$. By \cite[11.23]{spherical}, therefore,
conjugation by $\rho$ interchanges $m_1$ and $m_4$. By the identities in \cite[14.18 and 32.11]{TW},
$m_1$ and $m_4$ both have order~$2$. By \cite[6.9]{TW}, therefore, $\langle m_1,m_4\rangle$
is a dihedral group of order~$8$. In particular, $(m_1m_4)^2=(m_4m_1)^2$ and hence
\begin{equation}\label{onk120}
\nu:=(m_1m_4)^2
\end{equation}
is centralized by $\rho$.
The element $m_1$ fixes the vertices $\star$ and $U_1$ and reflects $\Sigma$ onto itself
and the element $m_4$ fixes the vertices $\bullet$ and $U_4$ and reflects $\Sigma$ onto itself.
Thus, in particular, $d=c^\nu$. Hence $U_d=\nu U_c\nu=\nu U\nu$ and thus
$\nu U\nu$ acts sharply transitively on $X\backslash\{d\}$.
The map $u\mapsto d^u$ is a bijection from the root group $U$ of $M$ to the
set $X\backslash\{e\}$. Hence there exists a unique permutation $\tau$ of $U^*$ such that
\begin{equation}\label{onk79}
d^{(u)\tau}=c^{\nu u\nu}
\end{equation}
for all $u\in U^*$.

The Moufang set $M$ is isomorphic to the Moufang set ${\mathbb M}(U,\tau)$
defined in \cite[\S3]{jordan}, where $\tau$ is as in \eqref{onk79}.
As the results \cite[Thms.~3.1 and~3.2]{jordan} indicate,
$\tau$ is an essential structural feature of the Moufang set $M$.
Our goal in this section is to compute the formula for $\tau$ in \ref{th:tau-wu}.

Using the definition of the graph $\Omega$ in \cite[7.1]{TW}, one can check that
the permutation of the vertex set of $\Omega$ given in Table~1
(where $U_{ij}:=U_{[i,j]}$ is as in \cite[5.1]{TW}) is
an automorphism of $\Omega$ which, like $m_1$, fixes
the vertices $\star$ and $U_1$ and reflects $\Sigma$ onto
itself. It follows from \cite[32.11]{TW} (even though we have reparametrized $U_+$)
that $m_1$ centralizes $U_3$ and $x_4(u)^{m_1}=x_2(u)$ for all $u\in V$.
Since $m_1$ maps the vertex $U_{13}$ to the vertex $U_{34}$, it maps the
image $U_{13}x_4(u)$ of the vertex $U_{13}$ under the element $x_4(u)\in U_+$
to the image $U_{34}x_2(u)$ of the vertex $U_{34}$ under the element $x_4(u)^{m_1}=x_2(u)$ of $U_+$
for all $u\in V$. Similarly, it maps $U_{12}x_3(u)$ to $U_4x_3(u)$ for all $u\in V$.
It follows by \cite[3.7]{TW} that the automorphism in Table~1 is $m_1$. By
similar arguments, the action of $m_4$ on the vertex set of $\Omega$ is as in in Table~$2$.

\begin{table}[ht!]
\small
\begin{align*}
	\star &\leftrightarrow \star \\
	\bullet &\leftrightarrow U_{24} \\[1.5ex]
	U_{1}\,x_2(u)\,x_3(v)\,x_4(w) &\leftrightarrow U_{1}\,x_2(w)\,x_3(v+g(u,w))\,x_4(u) \\
	U_{12}\,x_3(v)\,x_4(w) &\leftrightarrow U_{4}\,x_2(w)\,x_3(v) \\
	U_{13}\,x_4(w) &\leftrightarrow U_{34}\,x_2(w) \\[1.5ex]
	U_{4}\,x_1(u)\,x_2(v)\,x_3(w) &\xleftrightarrow{u \neq 0} U_{4}\,x_1(u^{-1})\,x_2(v u^{-1})\,x_3(u^{-1}v + w) \\
	U_{4}\,x_2(v)\,x_3(w) &\leftrightarrow U_{12}\,x_3(w)\,x_4(v) \\
	U_{34}\,x_1(u)\,x_2(v) &\xleftrightarrow{u \neq 0} U_{34}\,x_1(u^{-1})\,x_2(v u^{-1}) \\
	U_{34}\,x_2(v) &\leftrightarrow U_{13}\,x_4(v) \\
	U_{24}\,x_1(u) &\xleftrightarrow{u \neq 0} U_{24}\,x_1(u^{-1})
\end{align*}
\caption{The Involution $m_1$}

\begin{align*}
	\star &\leftrightarrow U_{13} \\
	\bullet &\leftrightarrow \bullet \\[1.5ex]
	U_{1}\,x_2(u)\,x_3(v)\,x_4(w) &\xleftrightarrow{w \neq 0} U_{1}\,x_2(w^{-1}v + u + w^{-1}g(u,w))\,x_3(v w^{-1})\,x_4(w^{-1}) \\
	U_{1}\,x_2(u)\,x_3(v) &\leftrightarrow U_{34}\,x_1(v)\,x_2(u) \\
	U_{12}\,x_3(v)\,x_4(w) &\xleftrightarrow{w \neq 0} U_{12}\,x_3(v w^{-1})\,x_4(w^{-1}) \\
	U_{12}\,x_3(v) &\leftrightarrow U_{24}\,x_1(v) \\
	U_{13}\,x_4(w) &\xleftrightarrow{w \neq 0} U_{13}\,x_4(w^{-1}) \\[1.5ex]
	U_{4}\,x_1(u)\,x_2(v)\,x_3(w) &\leftrightarrow U_{4}\,x_1(w)\,x_2(v + g(u,w))\,x_3(u) \\
	U_{34}\,x_1(u)\,x_2(v) &\leftrightarrow U_{1}\,x_2(v)\,x_3(u) \\
	U_{24}\,x_1(u) &\leftrightarrow U_{12}\,x_3(u)
\end{align*}
\caption{The Involution $m_4$}
\end{table}

\begin{table}[ht!]
\small
\begin{align*}
	\star &\mapsto U_{13} \\
	\bullet &\mapsto U_{12} \\
	U_{34} &\mapsto \star \\
	U_{24} &\mapsto \bullet \\[2ex]
	U_{1}\,x_2(u)\,x_3(v)\,x_4(w) &\xmapsto{u \neq 0} U_{1}\,x_2(u^{-1}v+w)\,x_3\bigl( vu^{-1}+g(u^{-1},w) \bigr)\,x_4(u^{-1}) \\
	U_{1}\,x_3(v)\,x_4(w) &\mapsto U_{34}\,x_1(v)\,x_2(w) \\
	U_{12}\,x_3(v)\,x_4(w) &\mapsto U_{4}\,x_1(v)\,x_2(w) \\
	U_{13}\,x_4(w) &\mapsto U_{1}\,x_2(w) \\[2ex]
	U_{4}\,x_1(u)\,x_2(v)\,x_3(w) &\xmapsto{u \neq 0} U_{4}\,x_1(u^{-1}v+w)\,x_2\bigl( vu^{-1}+g(u^{-1},w) \bigr)\,x_3(u^{-1}) \\
	U_{4}\,x_2(v)\,x_3(w) &\xmapsto{v \neq 0} U_{12}\,x_3(wv^{-1})\,x_4(v^{-1}) \\
	U_{4}\,x_3(w) &\mapsto U_{24}\,x_1(w) \\
	U_{34}\,x_1(u)\,x_2(v) &\xmapsto{u \neq 0} U_{1}\,x_2(vu^{-1})\,x_3(u^{-1}) \\
	U_{34}\,x_2(v) &\xmapsto{v \neq 0} U_{13}\,x_4(v^{-1}) \\
	U_{24}\,x_1(u) &\xmapsto{u \neq 0} U_{12}\,x_3(u^{-1})
\end{align*}
\caption{The Product $m_1 m_4$}\label{mouf-m1m4}
\end{table}

In Table~\ref{mouf-m1m4}, which is derived from Tables~1 and~2, we have displayed
the action of the product $m_1 m_4$. We consider the vertex
\[ U_1 \{ w,u \} = U_1\,x_2(u)\,x_3(ww + u + g(u,w))\,x_4(w) \]
with $u,w \in V$, where $\{w,u\}$ is as in \eqref{tru64},
and compute the image of this vertex under $\nu=(m_1m_4)^2$ in
\ref{th:tau-wu} using Table~\ref{mouf-m1m4}. First, though, we need to make
a few preparations.

\begin{lemma}\label{u-1v+w}
Let $u,w \in V$ with $u \neq 0$, and let $v=ww+u+g(u,w) \in V$.
Then $u^{-1}v + w \neq 0$.
\end{lemma}

\begin{proof}
Suppose by contradiction that $w = u^{-1}v$.
By~\ref{abc24}(i),
\begin{equation}\label{guw}
g(u,w) = g(u, q(u)^{-1} uv) = 0 ,
\end{equation}
so $v = ww + u$.
Then $u^{-1} = wv^{-1}$ by \ref{abc95x}(ii),
so by~\ref{abc95y}\eqref{uv-1}, $u = w^{-1} v = w^{-1} (ww + u)$.
By (R7), \ref{abc95x}\eqref{u-1.vu} and~\eqref{guw},
\[ u = w^{-1} \cdot ww + w^{-1} u + g(ww \cdot w^{-1}, u) = ww + w^{-1} u + g(w,u) = ww + w^{-1} u . \]
Hence $v = ww + u = w^{-1}u$, so~\ref{abc95x}\eqref{u-1.vu}
implies $w = u^{-1} v = u^{-1} \cdot w^{-1}u = u w^{-1}$.
Then $u = ww$, and hence $v = 0$, so $w = u^{-1}v = 0$, and then $u = ww = 0$, a contradiction.
We conclude that $u^{-1}v + w \neq 0$.
\end{proof}

\begin{notation}\label{def:Nwu}
We set
	\[ \N(\{w,u\}) :=
	\begin{cases}
		q(u) q\big(u^{-1}(ww+u+g(u,w)) + w\big) & \text{if } u \neq 0 , \\
		q(w)^{\theta+2} & \text{if } u = 0
	\end{cases}
	\]
for all $\{w,u\}\in U$. By~\ref{u-1v+w}, $\N(\{w,u\}) = 0$ only if $w = u = 0$.
We call $\N$ the {\it norm} of $M$.
\end{notation}

\begin{lemma}\label{N-expl}
Let $\{ w,u \} \in U$.
Then
\[\N(\{w,u\})=q(u)^\theta + q(u)q(w) + q(w)^{\theta+2} + f(u,ww)^\theta + f(u,wu) + q(w)f(u,ww). \]
\end{lemma}
\begin{proof}
This is obvious if $u=0$, so assume that $u \neq 0$. Let $v=ww+u+g(u,w)$.
Then
\begin{equation}\label{Nwu}
\N(\{w,u\}) = q(u) q(u^{-1}v + w) = q(v)^\theta + f(uv, w) + q(u) q(w)
\end{equation}
by \ref{q(uv)}. We have
\begin{align*}
q(v)&= q(ww) + q(u) + f(u,w)^\theta + f(ww, u) \\
&= q(w)^{\theta+1} + q(u) + f(u,w)^\theta + f(ww, u) ,
\end{align*}
and hence
\begin{equation}\label{qvth}
q(v)^\theta = q(w)^{\theta+2} + q(u)^\theta + f(u,w)^2 + f(ww, u)^\theta .
\end{equation}
We also have
\begin{align*}
f(uv,w)=f(u, wv)&= f\bigl( u, w \cdot (ww + u + g(u,w)) \bigr) \\
&= f\bigl( u, w \cdot ww + wu + f(u,w)w + g(uw, ww) \bigr) \\
&= q(w) f(u, ww) + f(u, wu) + f(u,w)^2 .
\end{align*}
Combining this with \eqref{Nwu} and \eqref{qvth}, we obtain the required formula.
\end{proof}

\begin{theorem}\label{th:tau-wu}
Let $\{ w,u \} \in U^*$. Then
\[ \{ w,u \}^\tau
= \left\{ \frac{q(u) w + f(u,w)u + u(ww + u)}{\N(\{w,u\})}
\ , \ \frac{q(w) u + w(ww + u)}{\N(\{w,u\})} \right\} . \]
\end{theorem}

\begin{proof}
Assume first that $u=0$.
Then $U_1 \{ w,u \} = U_1 \{ w,0 \} = U_1\,x_3(ww)\,x_4(w)$.
Since $\{w,u\} \in U^*$, we have $w\ne0$ and hence $ww\neq 0$.
Using Table~\ref{mouf-m1m4}, we obtain
\[ U_1\,x_3(ww)\,x_4(w) \xmapsto{m_1m_4} U_{34}\,x_1(ww)\,x_2(w)
\xmapsto{m_1m_4} U_1\,x_2(w \cdot (ww)^{-1})\,x_3((ww)^{-1}). \]
By \ref{abc95x}(i), we have
 $w \cdot (ww)^{-1} = w \cdot w^{-1}w^{-1} = w^{-1}w^{-1} = (ww)^{-1}$ and hence
\[ U_1\,x_2(w \cdot (ww)^{-1})\,x_3((ww)^{-1})
= U_1\,x_2((ww)^{-1})\,x_3((ww)^{-1}) = U_1 \{ 0, (ww)^{-1} \} . \]
By \eqref{onk79}, therefore,
\[ \{ w,0 \}^\tau = \{ 0, (ww)^{-1} \} . \]
Since
$$\frac{w \cdot ww}{\N(\{w,0\})}=q(w)ww/q(w)^{\theta+2}=(ww)^{-1},$$
we obtain the required formula.

	Assume now that $u \neq 0$, and let $v := ww + u + g(u,w)$.
	By~\ref{u-1v+w}, $u^{-1}v + w \neq 0$.
	Let
	\begin{align*}
		a &= u^{-1}v + w , \\
		b &= vu^{-1} + g(u^{-1},w)\text{ and} \\
		c &= u^{-1} ,
	\end{align*}
	so that
\[ U_1 \{ w,u \} = U_1\,x_2(u)\,x_3(v)\,x_4(w) \xmapsto{m_1m_4} U_1\,x_2(a)\,x_3(b)\,x_4(c) , \]
	and hence
\[ U_1 \{ w,u \} \xmapsto{(m_1m_4)^2} U_1 \{ a^{-1}, a^{-1}b + c \} . \]
Thus
\[ \{ w,u \}^\tau = \{ a^{-1}, a^{-1}b + c \}  \]
by \eqref{onk79}. Observe that
\begin{equation}\label{tau0}
	a = u^{-1} \bigl( ww + u + g(u,w) \bigr) + w = w + f(u,w) u^{-1} + u^{-1}(ww + u)
	\end{equation}
by (R2)	and
\begin{equation}\label{tau00}
		b = (ww + u)u^{-1} + g(u,w)u^{-1} + g(u^{-1},w) = (ww + u)u^{-1}
\end{equation}
	by~\ref{abc24}(iii).
	Also notice that
\begin{equation}\label{qa}
		q(a) = q(u^{-1}v + w) = q(u)^{-1} \N(\{w,u\})
\end{equation}
	by \ref{def:Nwu},
	and hence
\begin{equation}\label{tau000}
		a^{-1} = \frac{q(u) a}{\N(\{w,u\})} = \frac{q(u) w + f(u,w)u + u(ww + u)}{\N(\{w,u\})} .
\end{equation}
To compute $a^{-1}b + c$, we first notice that, by~\ref{yht1},
\begin{multline*}
	w\cdot (ww+u)u^{-1} \\
		= f(u^{-1}, w(ww+u))u^{-1} + f(u^{-1}, w)u^{-1}(ww+u) + q(u^{-1})w(ww+u) ,
\end{multline*}
	and hence
\begin{multline}\label{tau1}
		q(u) w\cdot (ww+u)u^{-1} \\
			= f(u^{-1}, w(ww+u))u + f(u,w)u^{-1}(ww+u) + w(ww+u) .
	\end{multline}
	Next, by \ref{abc95x}\eqref{u-1.vu},
	\begin{equation}\label{tau2}
		f(u,w) u \cdot (ww+u)u^{-1} = f(u,w)u^{-1}(ww+u) .
	\end{equation}
	Finally, by~\ref{abc95}\eqref{uv.vu},
\begin{multline}\label{tau3}
u(ww+u) \cdot (ww+u)u^{-1}= q(u) u^{-1}(ww+u) \cdot (ww+u)u^{-1} \\
		= q(u) q\bigl( u^{-1}(ww+u) \bigr) u^{-1}
		= q\bigl( u^{-1}(ww+u) \bigr) u .
\end{multline}
Also observe that
\begin{equation}\label{tau4}
c = u^{-1} = q(u)^{-1} u = q(a)u / \N(\{w,u\})
\end{equation}
by~\eqref{qa}.
By \eqref{tau00}, \eqref{tau000}, \eqref{tau1}, \eqref{tau2}, \eqref{tau3} and \eqref{tau4},
we obtain
\[ a^{-1} b + c
= \frac{f(u^{-1}, w(ww+u))u + w(ww+u) + q\bigl( u^{-1}(ww+u) \bigr) u + q(a)u}{\N(\{w,u\})} . \]
It remains to show that
\begin{equation}\label{tau5}
f(u^{-1}, w(ww+u)) + q\bigl( u^{-1}(ww+u) \bigr) + q(a) = q(w) .
\end{equation}
By~\eqref{tau0}, however, we have
\begin{multline*}
q(a) = q(w) + f(u,w)^2 q(u^{-1}) + q\bigl( u^{-1}(ww+u) \bigr) + f(u,w) f(w, u^{-1}) \\
+ f\bigl( w, u^{-1}(ww + u) \bigr) + f(u,w) f\bigl( u^{-1}, u^{-1}(ww+u) \bigr) .
\end{multline*}
Notice that the last term is $0$ by~\ref{abc24}(i)
and that
$$f(u,w)^2 q(u^{-1}) = f(u,w) f(w, u^{-1}).$$
We conclude that
\begin{align*}
q(a)&= q(w) + q\bigl( u^{-1}(ww+u) \bigr) + f\bigl( w, u^{-1}(ww + u) \bigr) \\
&= q(w) + q\bigl( u^{-1}(ww+u) \bigr) + f\bigl( u^{-1}, w(ww + u) \bigr)
\end{align*}
(by~\ref{abc24}(ii)). Thus \eqref{tau5} holds.
\end{proof}

\section{Moufang Subsets}
Our next goal is to identify three Moufang subsets of the Moufang set $M$.
We continue with the notation in \ref{onk60}. Let $G^\dagger$ be as in \ref{abc6}
applied to $M$.

\begin{remark}\label{onk80}
Let $U$ and $\tau$ be as in \S\ref{onk81} and suppose that $R$ is a subgroup of $U$
such that $R^*$ is $\tau$-invariant.
Let $\tau_R$ denote the restriction of $\tau$ to $R$. By \cite[6.2.2(1)]{yoav},
$M_R:={\mathbb M}(R,\tau_R)$ (as defined in \cite[\S3]{jordan}) is a Moufang set.
Let $G_R^\dagger$ be as in \ref{abc6} applied to $M_R$.
Let $c$, $d$ and $\nu$ be as in \eqref{onk79}, let $X_R=\{c\}\cup d^R$
and let $N=\langle R,R^\nu\rangle$. By \cite[6.2.6]{yoav}, $X_R$ is an $N$-orbit,
$R$ acts faithfully on $X_R$ and the map from the underlying set
$\{\infty\}\cup R$ of $M_R$ to $X_R$ sending $\infty$ to $c$ and
$u$ to $d^u$ for all $u\in R$ is a bijection which induces an isomorphism
from the group induced by $N$ on $X_R$ to $G_R^\dagger$ with kernel $Z(N)$.
\end{remark}

\begin{notation}\label{onk91}
Let $\Lambda=(L,\kappa)$ be an arbitrary octagonal set as defined in \ref{abc0}.
We denote by $\MSuz(\Lambda)$ the
Moufang set corresponding to the group $\Suz(\Lambda)$.
The root groups of $\MSuz(\Lambda)$
are the root groups of $\Suz(\Lambda)$. Each of them
is isomorphic to the group $P_\Lambda$ with underlying set $L\times L$, where
\begin{equation}\label{onk91x}
(a,b)\cdot(u,v)=(a+u,b+v+au^\kappa)
\end{equation}
for all $(a,b),(u,v)\in L\times L$. For each $z\in L^*$ and for each automorphism $\sigma$
of $L$ commuting with $\kappa$ (possibly trivial), let
$$(a,b)^{\tau_{z,\sigma}}=\left(z\Bigl(\frac{b}{a^{\kappa+2}+ab+b^\kappa}\Bigr)^\sigma,
           z^{\kappa+1}\Bigl(\frac{a}{a^{\kappa+2}+ab+b^\kappa}\Bigr)^\sigma\right)$$
for all $(a,b)\in P_\Lambda^*$.
By \cite[Exemple~2]{ree}, we have $\MSuz(\Lambda)\cong{\mathbb M}(P_\Lambda,\tau_{z,\sigma})$
for all $z\in L^*$ and all $\sigma\in{\rm Aut}(L)$ that commute with $\kappa$.
\end{notation}

\begin{remark}\label{onk91p}
Let $\Lambda=(L,\kappa)$ be as in \ref{onk91} and suppose that $|L|>2$. Let
$\MSuz(\Lambda)=(X,\{U_x\})_{x\in X}$ and let $B^\dagger$ be the group obtained
by applying \ref{abc6} to $\MSuz(\Lambda)$. By \cite[33.17]{TW}, we have
$U_x=[B^\dagger_x,U_x]$ for all $x\in X$.
\end{remark}

\begin{notation}\label{onk123}
Let $\chi$, $\theta$ and $\theta_K$ be as in \ref{tru5} and \ref{onk60},
let $\theta_1=\theta$ and let $\theta_2=\chi\theta_1$. Thus
$\chi$ commutes with $\theta_1$, and $\theta_2$ is also
a Tits endomorphism of $E$.
\end{notation}

\begin{proposition}\label{pil0}
Let $\theta_K$, $\theta_1$, $\theta_2$ and $\chi$ be as in {\rm\ref{onk123}}.
\begin{align*}
R_0 &:= \{ \{ (0,0,s), (0,0,t) \} \mid s,t \in K \} , \\
R_1 &:= \{ \{ (a,0,0), (0,b,0) \} \mid a,b \in E \} \text{ and} \\
R_2 &:= \{ \{ (0,b,0), (a,0,0) \} \mid a,b \in E \}
\end{align*}
and let $\tau_i$ denote the restriction of $\tau$ to $R_i$ for $i\in[0,2]$.
Then $\theta_2$ is a Tits endomorphism of $E$,
$R_0$, $R_1$ and $R_2$ are $\tau$-invariant
subgroups of $U$, ${\mathbb M}(R_0,\tau_0)\cong\MSuz(K,\theta_K)$
and ${\mathbb M}(R_i,\tau_i)\cong\MSuz(E,\theta_i)$ for $i\in[1,2]$.
\end{proposition}

\begin{proof}
We calculate using \ref{tru99}. First note that
\begin{align*}
\lbrace[s],[t]\rbrace\cdot\lbrace[a],[b]\rbrace&=
\lbrace[s]+[a],[t]+[b]+[s]\cdot[a]\rbrace\\
&=\lbrace[s+a],[t+b+sa^\theta]\rbrace
\end{align*}
for all $s,t,a,b\in K$ by \ref{tru64}.
Let $w = [s]$ and $u = [t]$ for some $s,t\in K$.
Then $f(u,w) = 0$, $q(u) = t^\theta$ and $q(w) = s^\theta$, $ww + u = [s^{\theta+1} + t]$.
Hence $\N(\{w,u\}) = \bigl( s^{\theta+2} + st + t^\theta \bigr)^\theta$ by~\ref{N-expl}.
Thus
\[\{ [s],[t]\}^\tau=\left\{\Big[\frac{t}{s^{\theta+2}+st+t^\theta}\Big] \ ,
\ \Big[\frac{s}{s^{\theta+2} + st + t^\theta}\Big] \right\}  \]
by \ref{th:tau-wu}. Therefore $R_0$ is $\tau$-invariant subgroup of $U$
and by \ref{onk91}, the map $(s,t)\mapsto\{[s],[t]\}$ is an isomorphism from $P_{(K,\theta_K)}$ to
$R_0$ which carries the map $\tau_{1,1}$ to $\tau_0$.
Hence ${\mathbb M}(R_0,\tau_0)\cong\MSuz(K,\theta_K)$.

Let $w = (a,0,0) \in V$ and $u = (0,b,0) \in V$ for some $a,b \in E$.
Then $q(w)=\beta^{-1}N(a)$ and $q(u)=\beta^{-\theta_1-1}N(b)$,
where $N$ is the norm of the extension $E/K$, and
$v:=ww+u=\bigl(0,\ a^{\theta_1+1} + b,\ 0 \bigr)$. Hence
$$q(u)w+u(ww+u)=\bigl(\beta^{-\theta_1-1} (a^{\theta_1+2}+ab+b^{\theta_1})b^\chi,\ 0,\ 0 \bigr)$$
and
$$q(w)u + w(ww+u)= \bigl( 0,\ \beta^{-1} (a^{\theta_1+2}+ab+b^{\theta_1})a^\chi,\ 0 \bigr).$$
By \eqref{Nwu}, we also have
$$\N(\{w,u\})=q(u)^{-1}q\big(uv+q(u)w\big)
=\beta^{-\theta_1-2}N\big(a^{\theta_1+2}+ab+b^{\theta_1}\big).$$
Setting $[a,b]=\{(a,0,0),(0,b,0)\}$ for all $a,b\in E$, we obtain
$$[a,b]^\tau=\left[\beta\Bigl(\frac{b}{a^{\theta_1+2}+ab+b^{\theta_1}}\Bigr)^\chi,
      \beta^{\theta_1+1} \Bigl( \frac{a}{a^{\theta_1+2}+ab+b^{\theta_1}}\Bigr)^\chi\right]$$
for all $a,b\in E$ by \ref{th:tau-wu}. By \eqref{tru64}, we have
$$[a,b]\cdot[u,v]=[a+u,b+v+au^{\theta_1}]$$
for all $a,b,u,v\in E$. Hence $R_1$ is a $\tau$-invariant subgroup of $U$ and by \ref{onk91},
$(a,b)\mapsto[a,b]$ is
an isomorphism from $P_{(E,\theta_1)}$ to $R_1$ that carries the map $\tau_{\beta,\chi}$
to $\tau_1$. Hence ${\mathbb M}(R_1,\tau_1)\cong\MSuz(E,\theta_1)$.

The computations for $R_2$ are similar. Setting
$[a,b]=\{(0,\beta a^\chi,0),(b,0,0)\}$, we find that
$$[a,b]\cdot[u,v]=[a+u,b+v+au^{\theta_2+1}]$$
for all $a,b,u,v\in E$ and
$$[a,b]^\tau=\left[ \beta^{\theta_2-1}\Bigl( \frac{b}{a^{\theta_2+2}+ab+b^{\theta_2}}\Bigr)^\chi ,
\beta\Bigl(\frac{a}{a^{\theta_2+2}+ab+b^{\theta_2}}\Bigr)^\chi \right]$$
for all $a,b\in E$. Hence $R_2$ is a $\tau$-invariant subgroup of $U$ and by \ref{onk91}, the
the map $(a,b)\mapsto[a,b]$ is an isomorphism from $P_{(E,\theta_2)}$ to $R_2$ that
carries the map $\tau_{\beta^{\theta_2-1},\chi}$ to $\tau_2$.
Hence ${\mathbb M}(R_2,\tau_2)\cong\MSuz(E,\theta_2)$.
\end{proof}

\section{Simplicity}\label{nzz60}

We can now deduce \ref{nzz13} as a corollary of \ref{pil0}.
Let $M$, $X$, $c$, $\Sigma$ and $U_c$ be as in \ref{onk60},
let $G^\dagger$ be as in \ref{abc6}
applied to $M$ and let $R_0$, $R_1$ and $R_2$ be as in \ref{pil0}.
Let $i\in[0,2]$ and let $N$ be as in \ref{onk80} with $R_i$ in place of $R$.
By \ref{abc70}, \ref{onk80}, \ref{onk91p} and \ref{pil0}, we have
$R_i\subset [N_c,R_i]\subset[G^\dagger,G^\dagger]$. Thus
$\langle R_0,R_1,R_2\rangle\subset[G^\dagger,G^\dagger]$.
By \ref{tru99} and \eqref{tru64} and some calculation, on the other hand,
$U_c=\langle R_0,R_1,R_2\rangle$.
Since $G^\dagger$ is generated by conjugates
of $U_c$, it follows that $G^\dagger$ is perfect.
To finish the proof, we proceed with a standard argument
which goes back to \cite{iwa}: Let $I$ be a non-trivial normal subgroup of $G^\dagger$.
Since $I$ is normal, the product $IU_c$ is a subgroup of $G^\dagger$.
Since $G^\dagger$ acts $2$-transitively on $X$, the subgroup $I$ acts transitively. Hence
the subgroup $IU_c$ contains all the root groups of $G^\dagger$. Therefore, $G^\dagger=IU_c$. Thus
$$G^\dagger/I\cong IU_c/I\cong U_c/U_c\cap I.$$
Since $U_c$ is nilpotent, it follows that $G^\dagger/I$ is nilpotent.
Since $G^\dagger$ is perfect, the quotient $G^\dagger/I$ is also perfect.
A perfect nilpotent group must be trivial. It follows that $I=G^\dagger$.
Thus $G^\dagger$ is simple. This concludes the proof of \ref{nzz13}.

\section{Invariants}\label{onk60x}

In this last section,
we show that $q$ is an invariant of $M$ (where $q$ and $M$ are as in \ref{onk60}).

\begin{theorem}\label{nzz30}
Let $(K,V,q,\theta,t\mapsto[t],\cdot)$ and
$(\tilde K, \tilde V, \tilde q, \tilde\theta, t\mapsto[t], *)$
be two polarity algebras as defined in {\rm\ref{fat1}}
and assume that the corresponding Moufang sets $M$ and $\tilde M$ are isomorphic.
Then there is a field isomorphism $\psi\colon K\to\tilde K$, an additive
bijection $\zeta\colon V\to\tilde V$ and an element $e\in\tilde K^\times$ such that
\begin{enumerate}[\rm(i)]
\item $\zeta(v \cdot w) = e^{-1} \zeta(v) * \zeta(w)$ for all $v,w \in V$;
\item $\zeta(tv) = \psi(t) \zeta(v)$ for all $t \in K$ and all $v \in V$;
\item $\tilde q(\zeta(v)) = e^{\tilde\theta} \psi(q(v))$ for all $v \in V$; and
\item $\psi(t^\theta) = \psi(t)^{\tilde\theta}$ for all $t \in K$.
\end{enumerate}
In particular, the quadratic forms $q$ and $\tilde q$ are similar.
\end{theorem}

\begin{proof}
Let $\pi$ be an isomorphism from $M=(X,\{U_x\}_{x\in X})$ to
$\tilde M=(\tilde X,\{\tilde U_{\tilde x}\}_{\tilde x\in\tilde X})$, i.e.~a bijection
from $X$ to $\tilde X$ such that $\pi^{-1}U_z\pi=\tilde U_{(z)\pi}$ for all $z\in X$.
We can assume that $M$, $c$ and $\Sigma$ are as in \ref{onk60}.
Thus by \ref{onk69}, $U_x$ is the group $U$ described in \S\ref{onk81}.
Let $d\in\Sigma$ be the unique
chamber of $\Sigma$ opposite $c$ (as in \S\ref{onk81})
and let $\tilde c$ and $\tilde d$ be the images of $c$ and $d$
under $\pi$. Let $H$ denote the pointwise stabilizer
of $\{c,d\}$ in $M$ and let $\tilde H$ denote the pointwise stabilizer
of $\{\tilde c,\tilde d\}$ in $\tilde M$.
For each $b\in U_c^*$, we denote by $m_b$ the
unique element in $U_dbU_d$ interchanging $c$ and $d$ and
by $\mu_b$ be the unique permutation of $U_c^*$ such that
\begin{equation}\label{onk121}
d^{(a)\mu_b}=c^{m_b^{-1}um_b}
\end{equation}
for each $a\in U_c^*$. We define
$\tilde\mu_{\tilde b}$ for each $\tilde b\in\tilde U_{\tilde c}$ analogously.
Let $\varphi \colon U_c \to \tilde U_{\tilde c}$ be the isomorphism
induced by $\pi$. Then
\begin{equation}\label{eq:moufisom}
\varphi((a)\mu_b)=(\varphi(a))\tilde\mu_{\varphi(b)}\text{ for all $a,b\in U$ with $b\ne1$}.
\end{equation}
Let $\{u,v\}$ for $u,v\in V$ be as in \eqref{nzz32}; we define $\{\tilde u,\tilde v\}$
for $\tilde u,\tilde v\in\tilde V$ analogously. Thus
$U=\{V,V\}$ and $\tilde U=\{\tilde V,\tilde V\}$.
Recall from \ref{tru61} that $U'=\{0,V\}$ and $\tilde U'=\{0,\tilde V\}$.
Therefore
\begin{equation}\label{eq:phi0V}
\varphi(\{0,V\})=\{0,\tilde V\}
\end{equation}
and
\begin{equation}\label{eq:Hinv}
\{0,\tilde V\}\text{ is $\tilde H$-invariant}.
\end{equation}
Let $\tilde c = \varphi(\{0,[1]\})$, $\tau_1=\mu_{\{0,[1]\}}$, let $\nu$ be
as in \eqref{onk120}, let $\tau$ be as in \eqref{onk79} and let
$\tilde\tau=\varphi^{-1}\tau\varphi$. The product $m_{\{0,[1]\}}\nu$ fixes $c$ and $d$ and
hence lies in $H$. By \eqref{onk79} and \eqref{onk121}, it follows that
$\tau_1\in H^\circ\tau$, where $H^\circ$ denotes the
permutation group
$$\{u\mapsto h^{-1}uh\mid h\in H\}$$
on $U_c^*$. Similarly, $\tilde\mu_c\in\tilde H^\circ\tilde\tau$, where
$\tilde H^\circ$ is defined analogously.
By \ref{th:tau-wu} and \eqref{eq:Hinv}, it follows that $\{0,V^* \}\tau_1=\{V^*,0\}$
and $\{0,\tilde V^*\}\tilde\mu_c=\{\tilde V^*,0\}$.
By~\eqref{eq:moufisom}, \eqref{eq:phi0V} and~\eqref{eq:Hinv}, therefore,
\begin{equation}\label{eq:phiV0}
\varphi(\{V^*,0\})=\varphi(\{ 0,V^* \}\tau_1)=\varphi(\{0,V^*\}) \tilde\mu_c
=\{\tilde V^*,0\}.
\end{equation}
By~\eqref{eq:phi0V} and~\eqref{eq:phiV0}, we conclude that there exist
maps $\zeta \colon V \to \tilde V$ and $\gamma \colon V \to \tilde V$ such that
$\varphi(\{u, 0\}) = \{ \zeta(u), 0 \}$ and $\varphi(\{0, v\}) = \{ 0, \gamma(v) \}$
for all $u,v \in V$.
Since $\{ u,v \} = \{ u,0 \} + \{ 0,v \}$ by~\eqref{tru64}, we have
$$\varphi(\{u,v\}) = \{ \zeta(u), \gamma(v) \}$$
for all $u,v \in V$. Therefore $\zeta$ is additive and
\begin{multline}\label{eq:M-beta}
        \gamma\bigl(w + b + ua + g(a,w) + g(a,uu)\bigr) \\
            = \gamma(w) + \gamma(b) +
\zeta(u)*\zeta(a) + \tilde g\bigl( \zeta(a), \gamma(w) \bigr)
+ \tilde g\bigl( \zeta(a), \zeta(u) * \zeta(u) \bigr)
    \end{multline}
for all $u,w,a,b \in V$ by~\eqref{tru64} since $\phi$ is a homomorphism.
    Setting $a=0$ in~\eqref{eq:M-beta}, we see that also $\gamma$ is additive.
By~\eqref{eq:M-beta}, therefore,
\begin{align*}
\gamma(ua) + \gamma\bigl(g(a,w)& + g(a,uu)\bigr) \\
&= \zeta(u)*\zeta(a) + \tilde g\bigl( \zeta(a), \gamma(w) \bigr)
+ \tilde g\bigl( \zeta(a), \zeta(u) * \zeta(u) \bigr)
\end{align*}
for all $u,w,a \in V$.
Substituting $uu$ for $w$ in this identity, we obtain
    \begin{equation}\label{eq:M-beta3}
        \gamma(ua) = \zeta(u)*\zeta(a) + \tilde g\bigl( \zeta(a), \gamma(uu) \bigr) + \tilde g\bigl( \zeta(a), \zeta(u) * \zeta(u) \bigr)
    \end{equation}
and if we set $a=u$ in this identity, we have
$$\gamma(uu) = \zeta(u)*\zeta(u) + \tilde g\bigl( \zeta(u), \gamma(uu) \bigr)
 + \tilde g\bigl( \zeta(u), \zeta(u) * \zeta(u) \bigr).$$
Applying the map $\tilde x\mapsto\tilde g(\zeta(a),\tilde x)$ to this last identity, we obtain
$$\tilde g \bigl( \zeta(a), \gamma(uu) \bigr) = \tilde g \bigl( \zeta(a), \zeta(u)*\zeta(u) \bigr).$$
Substituting this back into~\eqref{eq:M-beta3} now yields
\begin{equation}\label{eq:M-beta6}
        \gamma(ua) = \zeta(u)*\zeta(a)
    \end{equation}
    for all $u,a \in V$. In particular,
    \begin{equation}\label{eq:M-beta7}
        \gamma(u) = \zeta(u)*\zeta([1]).
    \end{equation}
Now observe that by~\ref{tru61} again, $\gamma$ maps $[K]$ onto $[\tilde K]$,
so by~\eqref{eq:M-beta7} and (R4), also $\zeta([K]) = [\tilde K]$.
In particular, there exists an $e \in \tilde K^\times$ such that $\zeta([1]) = [e]$, and hence
$\gamma(u)=e\zeta(u)$ for all $u \in V$ by (R2).
Substituting this back into~\eqref{eq:M-beta6} now yields (i).

Since $\zeta([K]) = [\tilde K]$, there is a unique map $\psi \colon K \to \tilde K$ such that
\begin{equation}\label{eq:epsi}
\zeta([t]) = [e \psi(t)]
\end{equation}
for all $t \in K$.
Substituting $[t]$ for $w$ in (i) and applying (R2), we obtain
$\zeta(tv) = \zeta(v [t]) = e^{-1} \zeta(v)[e \psi(t)] = \psi(t) \zeta(v)$.
Thus (ii) holds.

Since $\zeta$ is additive, so is $\psi$.
By~(ii), we have
\[ \psi(st) \zeta(u) = \zeta(stu) = \psi(s)\zeta(tu) = \psi(s)\psi(t)\zeta(u) \]
for all $s,t\in K$ and all $u\in V$, so $\psi$ is multiplicative.
By the definition of $e$, we have $\psi(1) = 1$. Thus $\psi$ is a field isomorphism.

We have $[1]v = [q(v)]$ for all $v \in V$ by (R1). Applying $\zeta$, we obtain using~(i)
that $[e] * \zeta(v) = e \zeta([q(v)])$.
By~\eqref{abc50a}, \eqref{eq:epsi} and (R1), this implies that
$[e \tilde q(\zeta(v))] = e[e\psi(q(v))] = [e^{\tilde \theta+1} \psi(q(v))]$ for all $v\in V$.
Thus (iii) holds.

Finally, we have $t[1] = [t^\theta]$ for all $t \in K$ by~\eqref{abc50a}.
Applying $\zeta$ again, we obtain using (ii) that
$\psi(t) \zeta([1]) = \zeta([t^\theta])$ and hence, by~\eqref{eq:epsi},
$\psi(t) [e] = [e \psi(t^\theta)]$.
Another application of~\eqref{abc50a}
now yields $[\psi(t)^{\tilde\theta} e] = [e \psi(t^\theta)]$ for all $t\in K$.
We conclude that (iv) holds.
\end{proof}

\nocite{*}
\bibliographystyle{alpha}
\bibliography{endo}

\end{document}